\documentclass[11pt]{article}
\usepackage{amsmath, amssymb, amsthm, esint, hyperref, cite}
\usepackage{geometry}

\usepackage{graphics}   % graphics package 1
\usepackage{graphicx}   % graphics package 2
\usepackage{epsfig}    

\usepackage{xcolor}

\def\R{{\mathbb R}}

\def\P{{\mathbb P}}

\theoremstyle{definition}

\newtheorem{theorem}{Theorem}[section]
\newtheorem{corollary}{Corollary}[section]
\newtheorem{definition}{Definition}[section]
\newtheorem{proposition}{Proposition}[section]
\newtheorem{remark}{Remark}[section]
\newtheorem{lemma}{Lemma}[section]
\begin{document}

\title{\sc{A stochastically perturbed fluid-structure interaction problem modeled by a stochastic viscous wave equation}}
\author{Jeffrey Kuan and Sun\v{c}ica \v{C}ani\'{c}\\
Department of Mathematics\\
University of California Berkeley}
\maketitle
\begin{abstract}
We study well-posedness for fluid-structure interaction driven by stochastic forcing.
This is of particular interest in real-life applications where forcing and/or data have a strong stochastic component.
The prototype model studied here is 
a stochastic viscous wave equation, which arises in modeling the interaction
between Stokes flow and an elastic membrane. To account for stochastic perturbations,
 the viscous wave equation is perturbed by
spacetime white noise scaled by a nonlinear Lipschitz function, which depends on the solution. We 
prove the existence of a unique function-valued stochastic mild solution to the corresponding Cauchy problem  in spatial dimensions one and two. 
Additionally, we show that up to a modification, the stochastic mild solution is $\alpha$-H\"{o}lder continuous 
for almost every realization of the solution's sample path, where {{$\alpha\in[0,1)$}} for spatial dimension $n=1$, and
{{$\alpha \in [0,1/2)$}} for spatial dimension $n=2$. 
This result contrasts the known results for the heat and wave equations perturbed
by spacetime white noise, including the damped wave equation perturbed by spacetime white noise,
for which a function-valued mild solution exists only in spatial 
dimension one and not higher. Our results show that  dissipation due to fluid viscosity, which is in the form
of the Dirichlet-to-Neumann operator applied to the time derivative of the membrane displacement, 
sufficiently regularizes the roughness of white noise in the stochastic viscous wave equation to allow the stochastic mild solution to exist even in
dimension two, which is the physical dimension of the problem.
To the best of our knowledge, this is the first result on well-posedness for a stochastically perturbed fluid-structure interaction problem.
\end{abstract}

\section{Introduction}\label{introduction}

We propose a stochastic model for fluid-structure interaction given by a stochastic wave equation augmented by dissipation 
associated with the effects of an incompressible, viscous fluid:
\begin{equation}\label{model}
u_{tt} + {2}\mu \sqrt{-\Delta} u_{t} - \Delta u = f(u)W(dt, dx), \qquad \text{ in } \mathbb{R}^{n}.
\end{equation}
The wave operator models the elastodynamics of a linearly elastic membrane, where $u$ denotes membrane displacement,
while the dissipative part,
which is in the form of the  Dirichlet-to-Neumann operator applied to the time derivative of displacement,
accounts for dissipation due to fluid viscosity, where $\mu$ denotes the fluid viscosity coefficient. 
The equation is forced by spacetime white noise $W(dt, dx)$, which accounts
for stochastic effects in real-life problems.  The spacetime white noise is scaled by a nonlinear, Lipschitz function $f(u)$.
We show below how this equation is  derived from a coupled fluid-structure interaction problem involving the Stokes equations 
describing the flow of an incompressible, viscous fluid, and the wave equation modeling the elastodynamics of a (stretched) linearly elastic membrane. 
 We consider  equation \eqref{model} in $\mathbb{R}^{n}$ with  $n = 1$ and $n = 2$, focusing primarily on $n=2$, which is the physical dimension.
 
We prove the existence of a function-valued {\emph{mild solution}} to a Cauchy problem for equation \eqref{model}, 
which holds both in dimensions $1$ and $2$. Here, by ``mild solution'' we refer to a stochastic mild solution defined
via stochastic integration involving the Green's function, specified below in Definition~\ref{milddef}. 
This is interesting because our result contrasts the results that hold for the stochastic heat and wave equations: 
the stochastic heat and 
the stochastic wave equations do not have function-valued mild solutions in {{spatial}} dimension $2$ or higher. 
Additionally, we prove that sample paths of the stochastic mild solution {{for the stochastic viscous wave equation}} are H\"{o}lder continuous with H\"{o}lder exponents
$\alpha \in [0,1)$ for $n = 1$, and $\alpha \in [0,1/2)$ for $n = 2$. 

Our results show that the viscous fluid dissipation in fluid-structure interaction is sufficient
to smooth out the rough stochastic nature of the real-life data in the problem modeled by the spacetime white noise. 
In particular, the Dirichlet-to-Neumann operator controls the high frequencies in the structure (membrane) {{displacement}} that are 
driven by the spacetime white noise.
 To the best of our knowledge, this is the first result on stochastic fluid-structure interaction.

We begin by describing the fluid-structure interaction model from which the equation \eqref{model} arises. Consider a prestressed infinite elastic membrane surface, which is modeled by the linear wave equation
\begin{equation}\label{wave}
u_{tt} - \Delta u = F, \ \  \text{ on } \Gamma := \{(x_1, x_2, 0) \in \mathbb{R}^{3}: (x_1, x_2) \in \mathbb{R}^{2}\},
\end{equation}
where $u(x_1, x_2)$ denotes the transverse displacement (in the $x_3$ direction) of the elastic surface from its reference configuration $\Gamma$.
See Figure \ref{domain}.

\begin{figure}[htp!]
\center
                    \includegraphics[width = 0.5 \textwidth]{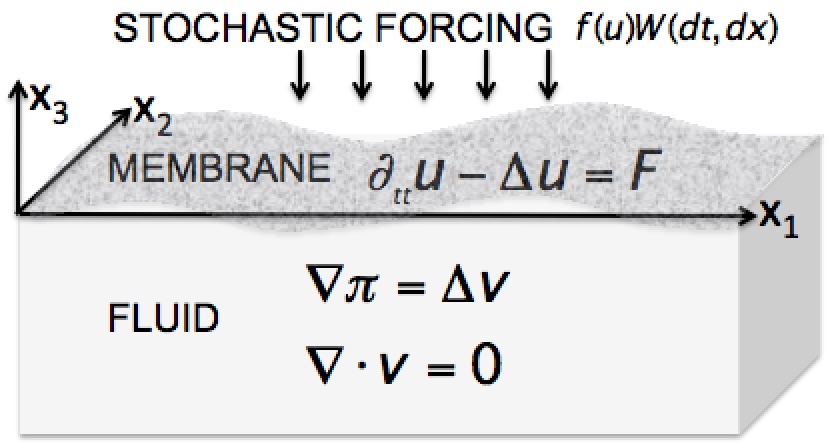}
  \caption{{\small\emph{A sketch of the fluid and structure domains.}}}
\label{domain}
\end{figure}

Beneath this elastic drum surface, we  consider a viscous, incompressible fluid, 
which resides in the lower half-space in $\mathbb{R}^{3}$, 
\begin{equation}\label{Omega}
\Omega = \{(x_1, x_2, x_3) \in \mathbb{R}^{3}: x_3 < 0\},
\end{equation}
modeled by the stationary Stokes equations for an incompressible,
viscous fluid:
\begin{equation}\label{Stokes1}
\left.
\begin{array}{rcl}
 \nabla \cdot \boldsymbol{\sigma}(\pi,{\boldsymbol v})&=& \boldsymbol{0}, \\
\nabla \cdot {\boldsymbol v} &=& 0,
\end{array}
\right\} \quad {\rm in} \ \Omega = \{ (x_1,x_2,x_3) \in \R^3 : x_3 <0\},
\end{equation}
where $\boldsymbol{\sigma} $ is the Cauchy stress tensor, and 
the unknown quantities are the fluid pressure $\pi: \Omega \to \mathbb{R}$ and the fluid velocity $\boldsymbol{v}: \Omega \to \mathbb{R}^{3}$. 
We will be assuming that the fluid is Newtonian, so that 
\begin{equation*}
\boldsymbol{\sigma} = -\pi \boldsymbol{I} + 2\mu \boldsymbol{D}(\boldsymbol{v}),
\end{equation*}
where $\mu$ denotes the fluid viscosity coefficient, 
$\boldsymbol{I}$ is the three by three identity matrix, and $\boldsymbol{D}(\boldsymbol{v})$ is the symmetrized gradient
of fluid velocity $\boldsymbol{D}(\boldsymbol{v}) = (\nabla \boldsymbol{v} + (\nabla \boldsymbol{v})^T)/2$.
%%%%%%%%
\if 1 = 0
\begin{equation*}
\boldsymbol{D}(\boldsymbol{v}) = 
\begin{pmatrix}
\frac{\partial v_{1}}{\partial x_1} & \frac{1}{2}\left(\frac{\partial v_{1}}{\partial x_2} + \frac{\partial v_{2}}{\partial x_1}\right) & \frac{1}{2}\left(\frac{\partial v_{1}}{\partial x_3} + \frac{\partial v_{3}}{\partial x_1}\right) \\
\frac{1}{2}\left(\frac{\partial v_{1}}{\partial x_2} + \frac{\partial v_{2}}{\partial x_1}\right) & \frac{\partial v_{2}}{\partial x_2} & \frac{1}{2}\left(\frac{\partial v_{2}}{\partial x_3} + \frac{\partial v_{3}}{\partial x_2}\right) \\
\frac{1}{2}\left(\frac{\partial v_{1}}{\partial x_3} + \frac{\partial v_{3}}{\partial x_1}\right) & \frac{1}{2}\left(\frac{\partial v_{2}}{\partial x_3} + \frac{\partial v_{3}}{\partial x_2}\right) & \frac{\partial v_{3}}{\partial x_3}
\end{pmatrix}.
\end{equation*}
\fi
%%%%%%%%%%%%
Therefore, the Stokes equations now read
\begin{equation}\label{Stokes}
\left.
\begin{array}{rcl}
 \nabla \pi&=& \mu \triangle \boldsymbol v, \\
\nabla \cdot {\boldsymbol v} &=& 0,
\end{array}
\right\} \quad {\rm in} \ \Omega = \{ (x_1,x_2,x_3) \in \R^3 : z<0\},
\end{equation}
where we require that the fluid velocity is bounded in the lower half space,
and the pressure $\pi \to 0$ as $|x|\to\infty$.

We consider the problem in which the elastic surface is displaced from its reference configuration $\Gamma$
with some given initial displacement and velocity,
allowing only vertical displacement,
{{where the elastodynamics of the elastic surface}} is driven by the total force exerted onto the membrane,
 which comes from the fluid on one side, 
 and an external stochastic forcing on the other. See Figure \ref{domain}.

Inifinite domains are considered to simplify the analysis, since the main purpose of this work is to understand
 the interplay between the dispersion effects in the 2D wave equation, dissipation due to fluid viscosity, and 
 stochasticity imposed by the external forcing, which can be related to the stochasticity of not only the external forcing,
 but also to the stochasticity of data (e.g., inlet/outlet data) in real-life applications, such as blood flow through arteries. 

%%%%%%%%%%%%%%
\if 1 = 0
We begin by describing the fluid and the structure subproblems. We first consider the structure subproblem. We model the elastic structure with reference configuration $\Gamma$ by the wave equation. We assume that the elastic structure only has displacements in the transverse $z$ direction, and hence we model it by the wave equation,
\begin{equation}\label{wave}
u_{tt} - \Delta u = F, \qquad \text{ on } \Gamma = \{(x_1, x_2, 0) \in \mathbb{R}^{3}: (x_1, y) \in \mathbb{R}^{2}\},
\end{equation}
where $u(x_1, x_2)$ denotes the transverse displacement (in the $z$ direction) of the elastic surface from its reference configuration at the point $(x_1, x_2, 0) \in \Gamma$. 

Next, we consider the fluid subproblem. We will model the fluid by the stationary Stokes equation for incompressible,
viscous fluids,
\begin{equation}\label{Stokes1}
\left.
\begin{array}{rcl}
 \nabla \cdot \boldsymbol{\sigma}(\pi,{\boldsymbol v})&=& \boldsymbol{0}, \\
\nabla \cdot {\boldsymbol v} &=& 0,
\end{array}
\right\} \quad {\rm in} \ \Omega = \{ (x_1,x_2,x_3) \in \R^3 : x_3 <0\},
\end{equation}
where $\boldsymbol{\sigma} $ is the Cauchy stress tensor, and 
the unknown quantities are the fluid pressure $\pi: \Omega \to \mathbb{R}$ and the fluid velocity $\boldsymbol{v}: \Omega \to \mathbb{R}^{3}$. 
We will be assuming that the fluid is Newtonian, so that 
\begin{equation*}
\boldsymbol{\sigma} = -\pi \boldsymbol{I} + 2\mu \boldsymbol{D}(\boldsymbol{v}),
\end{equation*}
where $\mu$ denotes the fluid viscosity coefficient, 
$\boldsymbol{I}$ is the three by three identity matrix, and $\boldsymbol{D}(\boldsymbol{v})$ is the symmetrized gradient
of fluid velocity
\begin{equation*}
\boldsymbol{D}(\boldsymbol{v}) = 
\begin{pmatrix}
\frac{\partial v_{1}}{\partial x_1} & \frac{1}{2}\left(\frac{\partial v_{1}}{\partial x_2} + \frac{\partial v_{2}}{\partial x_1}\right) & \frac{1}{2}\left(\frac{\partial v_{1}}{\partial x_3} + \frac{\partial v_{3}}{\partial x_1}\right) \\
\frac{1}{2}\left(\frac{\partial v_{1}}{\partial x_2} + \frac{\partial v_{2}}{\partial x_1}\right) & \frac{\partial v_{2}}{\partial x_2} & \frac{1}{2}\left(\frac{\partial v_{2}}{\partial x_3} + \frac{\partial v_{3}}{\partial x_2}\right) \\
\frac{1}{2}\left(\frac{\partial v_{1}}{\partial x_3} + \frac{\partial v_{3}}{\partial x_1}\right) & \frac{1}{2}\left(\frac{\partial v_{2}}{\partial x_3} + \frac{\partial v_{3}}{\partial x_2}\right) & \frac{\partial v_{3}}{\partial x_3}
\end{pmatrix}.
\end{equation*}
Therefore, the Stokes equations now read
\begin{equation}\label{Stokes}
\left.
\begin{array}{rcl}
 \nabla \pi&=& \mu \triangle \boldsymbol v, \\
\nabla \cdot {\boldsymbol v} &=& 0,
\end{array}
\right\} \quad {\rm in} \ \Omega = \{ (x_1,x_2,x_3) \in \R^3 : z<0\}.
\end{equation}
\fi
%%%%%%%%%%%%%%%%%%%%%%%%%

The fluid and the structure are coupled via two coupling conditions, the kinematic and dynamic coupling conditions,
giving rise to the so-called two-way coupled fluid-structure interaction problem. 
The coupling conditions in the present study are evaluated at a {\emph{fixed (linearized) fluid-structure interface}} corresponding
to the structure's reference configuration $\Gamma$. This is known as {\emph{linear coupling}}.
%We will be assuming small deformation gradients and small displacements of the interface (membrane), 
%so that the fluid and structure 
%can be coupled across the {\textit{linearized interface}}, i.e., across the reference configuration $\Gamma$.
%This is known as \textit{linear coupling}. 
The two conditions read:

%We emphasize that even though the structure itself with reference configuration $\Gamma$ is in motion, 
%the linear coupling is a good approximation of the fluid-structure interaction problem under those conditions. 
%We then note that the first equation in \eqref{Stokes}, which describes Newton's second law of motion for the fluid, is given simply by 
%\begin{equation*}
%\nabla \cdot \boldsymbol{\sigma} = 0.
%\end{equation*}
%The second equation in \eqref{Stokes} is the well-known incompressibility condition, describing the fact that the fluid is incompressible. 

%Now that we have described each of the fluid and structure subproblems separately, we couple them together with appropriate coupling conditions. There will be two coupling conditions: the kinematic coupling condition (describing no-slip continuity of velocities at the fluid-structure interface) and the dynamic coupling condition (describing the balance of forces at the fluid-structure interface). 
\begin{itemize}
\item \textbf{\textit{Kinematic coupling condition.}} The kinematic coupling condition describes the coupling between the kinematic quantities such as velocity. We will be assuming the no-slip condition, meaning that the fluid and structure velocities 
are continuous at the interface (there is no slip between the two):
\begin{equation}\label{kinematic}
u_{t} = \boldsymbol{v}|_{\Gamma}, \ \ {\rm for} \ x = (x_1,x_2)\in \mathbb{R}^{2},   t \ge 0.
\end{equation}
%This is the no-slip condition, which intuitively states that particles of the fluid stick to the structure and thus have the same velocity as the structure, which is $u_{t}$ in this model. 
\item \textbf{\textit{Dynamic coupling condition.}} The dynamic coupling condition describes the balance of forces at the  fluid-structure interface $\Gamma$,
namely, it states that the elastodynamics of the membrane is driven by the 
force corresponding to the jump in traction (normal stress) across the membrane.
On the fluid side, the traction (normal stress) at the interface is given by 
$-\boldsymbol{\sigma} \boldsymbol{e_{x_3}}$, where $\boldsymbol{e_{x_3}}$ is the normal vector to $\Gamma$,
while on the other side, we are assuming a given loading $F_{ext}(u)$ to be a stochastic process $f(u)W(dt,dx)$ in the $\boldsymbol{e_{x_3}}$ direction.
Examples of such a loading can be found in cardiovascular applications, see e.g., \cite{Astorino}.
Since we assume that the structure only has transversal displacement, the dynamic coupling condition reads:
\begin{equation}\label{dynamic}
u_{tt} - \Delta u = -\boldsymbol{\sigma} \boldsymbol{e_{x_3}} \cdot \boldsymbol{e_{x_3}} +  f(u(t, {x})) W(dt, dx) \quad {\rm where} \ x = (x_1,x_2)\in \mathbb{R}^{2}, t \ge 0.
\end{equation}
\end{itemize}

In fluid-structure interaction problems and physical problems in general, physical phenomena are subject to small random deviations that cause deviations from deterministic behavior. The consideration of such stochastic effects in partial differential equations can give rise to new phenomena, and is an area of active research. 
Furthermore, in real-life data, one observes such stochastic noise both in terms of the force exerted 
onto the structure, as well as in the data that drives the problem. 
For example, the measured inlet/outlet pressure data in a fluid-structure interaction problem
describing arterial blood flow, has similar stochastic noise deviations to $F_{ext}(u)= f(u(t, {x})) W(dt, dx)$.
%
%
%As a preliminary model for such stochastic deviations we model the external force $F_{ext}(u)$ as a stochastic process 
%\begin{equation}\label{stochasticFext}
%F_{ext}(u) = f(u(t, {x})) W(dt, dx) \quad {\rm where} \ x = (x_1,x_2)\in \mathbb{R}^{2}.
%\end{equation}
Here $W(dt,dx)$ is spacetime white noise in $(t, x) \in \mathbb{R}^{+} \times \mathbb{R}^{2}$, whose properties we will recall in Sec.~\ref{prelim}.  We will assume that $f: \mathbb{R} \to \mathbb{R}$ is a Lipschitz continuous function. In particular, the case 
$f \not\equiv 1$ allows dependence of the magnitude of the stochastic noise at each point on the structure displacement $u(t, x)$ itself. 

To derive equation \eqref{model} as a model which describes the fluid-structure interaction problem \eqref{wave}-\eqref{dynamic},
we focus on the dynamic coupling condition \eqref{dynamic}. The goal is to try to express the effects of fluid normal stress
via the Dirichlet-to-Neumann operator defined entirely in terms of $u$ and/or its derivatives. 
In this derivation we also use the kinematic coupling condition \eqref{kinematic} as explained below. 

First notice that the right hand-side of \eqref{dynamic} is given by
\begin{equation*}
-\boldsymbol{\sigma} \boldsymbol{e_{x_3}} \cdot \boldsymbol{e_{x_3}} = \pi - {2}\mu \frac{\partial v_{x_3}}{\partial x_3} 
\ {\rm on} \ \Gamma.
\end{equation*}
Since the tangential displacements 
are assumed to be zero, the kinematic coupling condition \eqref{kinematic}
implies that the $x_1$ and $x_2$ components of the fluid velocity $\boldsymbol v$ are zero on $\Gamma$, namely 
$v_{x_1} = v_{x_2} = 0$ on $\Gamma$. By the divergence free condition, one immediately gets that
$
{\partial v_{x_3}}/{\partial x_3} = 0
$
on $\Gamma$.
 Therefore,
\begin{equation}\label{pressure}
-\boldsymbol{\sigma} \boldsymbol{e_{x_3}} \cdot \boldsymbol{e_{x_3}} = \pi \quad {\rm on} \ \Gamma,
\end{equation}
where $\pi$ is the fluid pressure given as a solution to the Stokes equations \eqref{Stokes}.
So it remains to find an appropriate expression for $\pi$ on $\Gamma$
 in terms of the structure displacement $u$. In fact, we will 
show that the following formula holds
\begin{equation}\label{pressureDN}
\pi = -2\mu \sqrt{-\Delta} u_{t} \quad { \text{on} \ \Gamma},
\end{equation}
under the assumption that $u$ and $u_{t}$, along with their spatial derivatives, are smooth functions that are rapidly decreasing at infinity. We will also impose the boundary conditions on \eqref{Stokes}, {{stating}} that the fluid velocity is bounded on the lower half {{space}}, and the pressure $\pi$ has a limit equal to zero as $|x| \to \infty$ in the lower half {{space}}.
Details of this calculation are presented in \cite{KuanCanicNVWE}. Here we present the main steps.

To derive the formula \eqref{pressureDN}, we note that by taking the inner product of the first equation in \eqref{Stokes} with $\boldsymbol{e_{x_3}}$, we obtain
\begin{equation}\label{Neumannpressure}
\frac{\partial \pi}{\partial x_3} = \mu \Delta_{x_1, x_2} v_{x_3} + \mu \frac{\partial^{2} v_{x_3}}{\partial x_3^{{2}}}
= \mu \Delta v_{x_3} ,
\end{equation}
where $\Delta_{x_1, x_2} := \frac{\partial^{2}}{\partial x_1^{2}} + \frac{\partial^{2}}{\partial x_2^{2}}$. Furthermore, by taking the divergence of the first equation in \eqref{Stokes}, and by using the divergence-free condition, we get that the pressure $\pi$ is harmonic. Thus, if we can compute the right hand side of \eqref{Neumannpressure} on $\Gamma$, we can recover $\pi$ as the solution to a Neumann boundary {{value}}  problem for Laplace's equation in the lower half {{space}}, with the boundary condition 
{{requiring}} that $\pi$ goes to zero at infinity. 

To compute the right hand side of \eqref{Neumannpressure}, 
we need to compute $v_{x_3}$.
Taking the Laplacian {{on both sides}} of \eqref{Neumannpressure}, and using the fact that $\pi$ is harmonic, we obtain
\begin{equation}\label{biharmonic}
\Delta^{2} v_{x_3} = 0, \qquad \text{ on } \Omega = \{(x_1, x_2, x_3) \in \R^{3} : x_3 < 0\}.
\end{equation}
{{Thus, $v_{x_3}$ satisfies the}} biharmonic equation with {{the following}} two boundary conditions:
 {{from}} the kinematic coupling condition, we get 
\begin{equation}\label{bc1}
v_{x_3}(x_1, x_2, 0) = u_{t}(x_1, x_2, 0), \qquad \text{ on } \Gamma = \{(x_1, x_2, x_3) \in \R^{3}: x_3 = 0\},
\end{equation}
and from {{the fact that $v_{x_{1}} = v_{x_{2}} = 0$ on $\Gamma$, by the kinematic coupling condition and the fact that $\boldsymbol{v}$ is divergence free, we get}}
\begin{equation}\label{bc2}
\frac{\partial v_{x_3}}{\partial x_3}(x_1, x_2, 0) = 0, \qquad \text{ on } \Gamma = \{(x_1, x_2, x_3) \in \R^{3}: x_3 = 0\}.
\end{equation}
We solve \eqref{biharmonic} with boundary conditions \eqref{bc1} and \eqref{bc2} by taking a Fourier transform in the variables $x_1$ and $x_2$, but not in $x_3$. We will denote the Fourier variables associated with $x_1$ and $x_2$ 
by $\xi_{1}$ and $\xi_{2}$, and we will denote ${\xi} = (\xi_{1}, \xi_{2})$, $|{\xi}|^{2} = \xi_{1}^{2} + \xi_{2}^{2}$. 
{{The Fourier transform equation then reads:}}
 \begin{equation}\label{Fourierstokes}
|\xi|^{4} \widehat{v_{x_3}}({\xi}, x_3) - 2|{\xi}|^{2} \frac{\partial^{2}}{\partial x_3^{2}} \widehat{v_{x_3}}({\xi}, x_3) + \frac{\partial^{4}}{\partial x_3^{4}} \widehat{v_{x_3}}({\xi}, x_3) = 0.
\end{equation}
The solution is given by
\begin{equation}\label{finalFourier}
\widehat{v_{x_3}}({\xi}, x_3) = \widehat{u_{t}}({\xi}) e^{|{\xi}| x_3} - |\xi| \widehat{u_{t}}({\xi}) x_3 e^{|{\xi}| x_3}.
\end{equation}
We can now compute the right hand side of \eqref{Neumannpressure}. 
Taking the Fourier transform of \eqref{Neumannpressure} in the $x_1$ and $x_2$ variables, and evaluating 
{{the equation}} on $\Gamma$ by using {{the kinematic coupling condition}} \eqref{kinematic}, {{we get}}
\begin{equation}\label{NeumannFourier}
\frac{\partial \widehat{\pi}}{\partial x_3}({\xi}, 0) = -\mu |\xi|^{2} \widehat{u_{t}}({\xi}) + \mu \frac{\partial^{2} \widehat{v_{x_3}}}{\partial x_3^{{2}}}({\xi}, 0) 
= -2\mu |\xi|^{2} \widehat{u_{t}}({\xi}),
\end{equation}
where the last equality follows by using the explicit formula for $\widehat{v_{x_3}}({\xi}, z)$ in \eqref{finalFourier}.
%\begin{equation}\label{NeumannFourier}
%\frac{\partial \widehat{\pi}}{\partial x_3}({\xi}, 0) = -2\mu |\xi|^{2} \widehat{u_{t}}({\xi}).
%\end{equation}

We now know that the pressure $\pi$ is a harmonic function in the lower half space,
satisfying the Neumann boundary condition \eqref{NeumannFourier} given in Fourier space.
To recover formula
\eqref{pressureDN} we want  $\pi$ on $\Gamma$.
This can be obtained via the Neumann to Dirichlet operator.
It is well known that the {\textit{Dirichlet to Neumann}} operator for
Laplace's equation in the lower half {{space}} is given by $\sqrt{-\Delta}$,
thereby having a Fourier multiplier $|\xi|$, see e.g., \cite{CS}. 
Therefore, the  Neumann to Dirichlet operator for Laplace's equation in the lower half {{space}} (with the solution to Laplace's equation having a limit of zero at infinity) is a Fourier multiplier of the form $\frac{1}{|\xi|}$. 
Thus, the Neumann to Dirichlet operator applied to the Neumann data \eqref{NeumannFourier} gives the 
pressure as Dirichlet data:
\begin{equation*}
\widehat{\pi}(\xi) = -2\mu |\xi| \widehat{u_{t}}(\xi) \qquad \text{ on } \Gamma,
\end{equation*}
which establishes the desired formula \eqref{pressureDN}.

The  dynamic coupling condition \eqref{dynamic}, together with \eqref{pressureDN}
gives the stochastic model
\begin{equation*}\label{modelfin}
u_{tt} + \sqrt{-\Delta} u_{t} - \Delta u = f(u)W(ds, dy) \ \  \text{ on } \mathbb{R}^{2},
\end{equation*}
where we have set the fluid viscosity $\mu = 1/2$. We will refer to this model as the 
{\textit{stochastic viscous wave equation}}, as it is a stochastic wave equation augmented by the viscous effects of the fluid, which are captured by the 
Dirichlet to Neumann operator acting on the structure velocity $u_{t}$. 

The study of fluid-structure interaction, which concerns the coupled {\emph{dynamical}} interactions between fluids and  deformable structures/solids, has been the focus of many {{works}}.
% in particular related to numerical methods developments \cite{peskin1,peskin2,fauci1,fogelson,peskin3,peskin4,Griffith1,donea1983arbitrary, hughes1981lagrangian,heil2004efficient,QTV00,baaijens,heart3,figueroa,cervera,nobile,deparis,FernandezLifeV,heil2004efficient,matthies,deparis2,MarSunLongitudinal,}. 
The results related to the analysis of fluid-structure interaction started coming out only within
the past 20 years. In particular, fluid-structure interaction problems with \textit{linear coupling}, which is considered in the current work, have  been investigated in, e.g., \cite{BarGruLasTuff2,BarGruLasTuff,Gunzburger,KukavicaTuffahaZiane}. The more general case of \textit{nonlinear coupling}, which has been studied in \cite{BdV1,CDEM,ChengShkollerCoutand,ChenShkoller,CSS1,CSS2,CG,Grandmont16,FSIforBIO_Lukacova,IgnatovaKukavica,ignatova2014well,Kuk,LengererRuzicka,Lequeurre,MuhaCanic13,
BorSun3d,BorSunMultiLayered,BorSunNonLinearKoiter,BorSunSlip,Raymond}, allows the fluid domain to change as a function of time, and the coupling conditions between the fluid and structure are evaluated at the current location of the interface, 
not known {\textit a priori}. This creates additional (geometric) nonlinearities and generates additional mathematical difficulties. 
In all of these works the focus is on deterministic models, in which there are no stochastic effects. 

The study of stochasticity in PDEs has been of recent interest. Most physical phenomena occurring in real-life applications feature the presence of some sort of random noise that perturbs the system from what may be deterministically expected. The types of noise added to the equation can vary, from the simplest spacetime white noise, which intuitively is noise that is ``independent at each time and space", to forms of noise with smoother spatial correlation so that the noise is still independent at separate times, but the noise in space allows for correlation between points and is hence ``smoother" than white noise.

Many classical partial differential equations, such as the heat and wave equations, have been studied with the addition of stochastic random forcing. There are many approaches to the study of such stochastic PDEs. One approach uses Walsh's theory of martingale measures, of which white noise is an example. The theory of integration against such martingale measures
can be found in Walsh's work \cite{Walsh}. Upon defining such an appropriate theory of stochastic integration, one can define what is called a {\textit {mild solution}} to a given stochastic partial differential equation by means of the Green's function. 

The choice of the stochastic noise used in the PDE being studied is of utmost importance. White noise, which is noise that is intuitively ``independent at all spaces and times", is a starting point for many studies. 
In formal mathematical notation, the time and space independence property of white noise is expressed
via expectation as
\begin{equation*}\label{white}
\mathbb{E}(\dot{W}(t, x) \dot{W}(s, y)) = \delta_{0}(t - s) \delta_{0}(x - y),
\end{equation*}
where $\delta_{0}$ is the Dirac delta function, and $\dot{W}(t,x)$, or $W(dt, dx)$, denotes spacetime white noise. 
In the rest of this manuscript, we will be  using $W(dt,dx)$ to denote spacetime white noise
and stochastic integration against white noise.

The white noise perturbed heat and wave equations have interesting properties. 
When perturbed by white noise, the two equations: 
\begin{equation*}
u_{t} - \Delta u = W(dx, dt)\quad \text{and} \quad  
u_{tt} - \Delta u = W(dx, dt) \qquad \text{ in } \mathbb{R}^{n},
\end{equation*}
do not allow function-valued mild solutions in {{spatial}} dimensions two and higher,
while they do in {{spatial}} dimension one.
See, for example,  \cite{DMini} and \cite{KMini},
where questions of  existence and uniqueness of mild solutions are addressed. 
This interesting property related to {{spatial}} dimensions two and higher is
due to the lack of square integrability of the Green's function in time and space, as we discuss later. 
Because of this property, ``smoother" types of stochastic noise, such as spatially homogeneous {{Gaussian}} noise, are used to perturb the stochastic heat and wave equations in higher dimensions in order to yield function-valued {{mild}} solutions. 
In formal mathematical notation, such spatially homogeneous Gaussian noise $\dot{F}(t, x)$ has a covariance structure
\begin{equation}\label{spatialhom}
\mathbb{E}(\dot{F}(t, x) \dot{F}(s, y)) = \delta_{0}(t - s) k(x - y),
\end{equation}
where $k: \mathbb{R}^{n} \to \mathbb{R}$. Note that the formal case of setting $k$ to be the Dirac delta ``function" recovers the previous white noise case, though choosing smoother functions $k$ allows us to formulate ``smoother" types of noise. See Sec.~3 of \cite{DMini} and the work in \cite{DF} for more information about spatially homogeneous Gaussian noise. 
One of the key questions in studying stochastically perturbed 
PDEs is what  conditions on $k$ need to be imposed so that the resulting equation with the spatially homogeneous Gaussian noise with covariance structure \eqref{spatialhom} has function-valued mild solutions?
See, for example \cite{DF}, and for more general contexts \cite{Dalang} and \cite{KZ}.

In the current manuscript, we do not need the properties of general spatially homogeneous Gaussian noise.
This is because,
as we shall see below, the viscous wave equation 
\begin{equation}\label{vwe}
u_{tt} + \sqrt{-\Delta} u_{t} - \Delta u = F, \qquad \text{ on } \mathbb{R}^{n},
\end{equation}
first considered in \cite{KuanCanicNVWE} as a model for fluid-structure interaction,
combines the following two desirable properties:
 the ``right'' spacetime scaling (c.f. wave equation), and adequate dissipative effects.
 The resulting behavior 
is ``in between'' the wave and heat equations.
The viscous wave equation \eqref{vwe} turns out to have just the right scaling and dissipation to 
 allow function-valued mild solutions even in {{spatial}} dimension two for the white noise perturbed equation
\begin{equation}\label{whitevwe}
u_{tt} + \sqrt{-\Delta} u_{t} - \Delta u = W(dx, dt) \qquad \text{ in } \mathbb{R}^{n}.
\end{equation}
\if 1 = 0
%%%%%%%%%%%%%
To see that solutions of the deterministic viscous wave equation
the deterministic viscous wave equation, 
which was first considered in \cite{JC} as a fluid-structure interaction model, is given by
\begin{equation}\label{vwe}
u_{tt} + \sqrt{-\Delta} u_{t} - \Delta u = F, \qquad \text{ on } \mathbb{R}^{n}.
\end{equation}
This equation, see \cite{JC}, has behavior that is ``in between" the wave and heat equations. This can be seen by the Fourier transform, in which there is motion of waves, which is damped over time by the parabolic effects of the dissipative Dirichlet to Neumann operator $\sqrt{-\Delta}$. 

Since the viscous wave equation \eqref{vwe} is in some sense ``in between" the wave and heat equation, the following result for the white noise perturbed equation is surprising. While the stochastic wave and heat equations with white noise perturbation given by
\begin{equation*}
u_{t} - \Delta u = W(dx, dt) \qquad \text{ in } \mathbb{R}^{n},
\end{equation*}
\begin{equation*}
u_{tt} - \Delta u = W(dx, dt) \qquad \text{ in } \mathbb{R}^{n},
\end{equation*}
have function-valued mild solutions only in dimension $1$ but not $2$ or above, we show in this paper that the following stochastic viscous wave equation
\begin{equation}\label{whitevwe}
u_{tt} + \sqrt{-\Delta} u_{t} - \Delta u = W(dx, dt) \qquad \text{ in } \mathbb{R}^{n},
\end{equation}
has function-valued mild solutions in both dimensions $1$ and $2$ (but not $3$ or higher). In particular, unlike for the heat and wave equations, perturbing the viscous wave equation by spacetime white noise in dimension $2$ still yields function-valued mild solutions. 
%%%%%%%%%%%%%%
\fi
This is  of great interest, since equations \eqref{model} and \eqref{whitevwe} in two {{spatial}} dimensions correspond exactly to the physical fluid-structure interaction model we are considering, and hence have direct physical significance. 

The main results  of this work are: (1) the existence of a function-valued mild solution for the white noise perturbed 
viscous wave equation \eqref{model} (and \eqref{whitevwe}) for dimensions $n = 1$ and $n=2$, and
(2) H\"{o}lder continuity $C^{0,\alpha }$ with $\alpha \in [0,1)$ for $n = 1$, and $\alpha \in [0,1/2)$ for $n=2$,
of {{``every" realization of the displacement $u$}}, obtained as a mild solution to the randomly perturbed viscous wave equation.

In particular, in terms of H\"{o}lder continuity, our results imply that the stochastic mild solution to equation \eqref{model} with zero initial data has a continuous {\textit {modification}} that is $\alpha$-H\"{o}lder continuous in time and space, with $\alpha \in [0,1)$ for $n = 1$, and $\alpha \in [0,1/2)$ for $n=2$.
Here, a {\textit{modification}} of a stochastic process $\{X_i\}_{i\in I}$ is defined to be a stochastic process
$\{\tilde{X}_i\}_{i\in I}$ such that the probability $\P(\tilde{X}_i = X_i) = 1$, for all $i \in I$.
Thus, we show that the stochastic function-valued mild solution {{has a modification that is a H\"{o}lder continuous function for
every realization $u$ of the displacement, obtained as a mild solution to the randomly perturbed viscous wave equation \eqref{model}.}} 

Even in dimension $n=1$, this contrasts the results for the stochastically perturbed heat and wave equations:
\begin{equation}\label{Lipheatwave}
u_{t} - \Delta u = f(u)W(dx, dt) \quad \text{ and } \quad u_{tt} - \Delta u = f(u)W(dx, dt) \qquad \text{ on } \mathbb{R}.
\end{equation}
Namely, in $n=1$, the function-valued mild solutions {{(up to modification)}} for zero initial data 
are $\alpha$-H\"{o}lder continuous in time and $\beta$-H\"{o}lder continuous in space, where $\alpha \in [0, 1/4)$, $\beta \in [0, 1/2)$ for the stochastic heat equation, and $\alpha, \beta \in [0, 1/2)$ for the stochastic wave equation, 
see \cite{KMini},  \cite{DMini}, and  \cite{Hairer}. The difference in space and time H\"{o}lder regularity between the heat and wave equations is due to the scaling of space and time, where for the heat equation, one time derivative ``corresponds'' to two spatial derivatives, while for the wave equation, one time derivative ``corresponds'' to one spatial derivative. 
In the stochastic viscous wave equation, the additional regularizing effect of the fluid viscosity implies 
{{improved H\"{o}lder regularity. In spatial dimension one, the solution is H\"{o}lder continuous of order $\alpha \in [0, 1)$ in space and time, which is an improvement over the results for both the stochastic heat and wave equations.  
In spatial dimension two, the solution is H\"{o}lder continuous of order $\alpha \in [0, 1/2)$ in space and time, whereas the stochastic heat and wave equations do not have function-valued mild solutions in spatial dimension two.}}

\if 1= 0
Since the spacetime scaling of the viscous wave equation is the same as that for the wave equation, it is to be expected that we obtain the same H\"{o}lder continuity result. However, we emphasize that our result for H\"{o}lder continuity of \eqref{Lipeqn} holds not only in dimension one, but also in dimension two. 
\fi 

The literature on the  H\"{o}lder continuity properties of the solutions to the heat equation and the wave equation { {with random noise}} in {{spatial}} dimensions two and higher,  is an area of extensive study. However, we emphasize again that for these equations, the stochastic noise is not white noise, but something smoother, such as, e.g., spatially homogeneous Gaussian noise, {{as a function-valued mild solution does not exist with spacetime white noise in spatial dimensions two and higher for these equations}}. We refer the reader to \cite{CD}, \cite{DSS}, \cite{SSS} for more details.

 This paper is organized as follows.
In Sec.~\ref{prelim}, we recall the properties of white noise, stochastic integration, and the deterministic forms of the heat, wave, and viscous wave equation solutions that will be necessary to show the main result. 
In Sec.~\ref{whitenoise}, we show the existence and uniqueness of a stochastic function-valued mild solution for equation \eqref{model},
%stated here again:
%\begin{equation*}\label{Lipeqn}
%u_{tt} + \sqrt{-\Delta} u_{t} - \Delta u = f(u)W(dx, dt) \qquad \text{ in } \mathbb{R}^{n},
%\end{equation*}
in dimensions $n = 1, 2$, where $f: \mathbb{R}^{n} \to \mathbb{R}$ is a Lipschitz continuous function,
and in Sec.~\ref{Holder}, we study the H\"{o}lder continuity of sample paths of solutions to \eqref{model}.

\section{Preliminaries}\label{prelim}

In this section, we first recall some basic facts about the deterministic linear heat, wave, and viscous wave equations, and about stochastic processes that we will need in the upcoming sections.  Then, we discuss white noise as a Gaussian process, and recall some elementary properties of stochastic integration. 

\subsection{The viscous wave equation}

We begin by considering the (linear) viscous wave equation
\begin{equation}\label{zerolin}
u_{tt} + \sqrt{-\Delta} u_{t} - \Delta u = 0,
\end{equation}
with initial data
\begin{equation*}
u(0, x) = g(x), \qquad \partial_{t}u(0, x) = h(x).
\end{equation*}
We recall some basic properties related to the analysis of this equation here, and recommend
 \cite{KuanCanicNVWE} for more information. 
Assuming that the initial data $g, h$ are regular enough, for example $g, h \in \mathcal{S}(\mathbb{R}^{n})$, we can explicitly solve this equation using the Fourier transform to obtain
\begin{equation}\label{FTv}
\widehat{u}(t, \xi) = \widehat{g}(\xi)e^{-\frac{|\xi|}{2}t}\left(\text{cos}\left(\frac{\sqrt{3}}{2}|\xi|t\right) + \frac{1}{\sqrt{3}}\text{sin}\left(\frac{\sqrt{3}}{2}|\xi|t\right)\right) + \widehat{h}(\xi)e^{-\frac{|\xi|}{2}t}\frac{\text{sin}\left(\frac{\sqrt{3}}{2}|\xi|t\right)}{\frac{\sqrt{3}}{2}|\xi|}.
\end{equation}
From the Fourier representation one can see that this equation has both parabolic and wave-like properties. 
The wavelike behavior is represented by the presence of cosine and sine, and
the strong parabolic dissipation is given by the exponential factor $e^{-\frac{|\xi|}{2}t}$, 
which causes damping of frequencies over time. 

Of particular interest to this work is the solution to the general inhomogeneous problem
\begin{equation}\label{inhomlin}
u_{tt} + \sqrt{-\Delta} u_{t} - \Delta u = F,
\end{equation}
with initial data $u(0, x) = g(x)$, $\partial_{t}u(0, x) = h(x)$,
which can be obtained using Duhamel's principle: 
\begin{eqnarray}\label{Fouriersoln1}
\widehat{u}(t, \xi) &= \widehat{g}(\xi)e^{-\frac{|\xi|}{2}t}\left(\text{cos}\left(\frac{\sqrt{3}}{2}|\xi|t\right) + \frac{1}{\sqrt{3}}\text{sin}\left(\frac{\sqrt{3}}{2}|\xi|t\right)\right) + \widehat{h}(\xi)e^{-\frac{|\xi|}{2}t}\frac{\text{sin}\left(\frac{\sqrt{3}}{2}|\xi|t\right)}{\frac{\sqrt{3}}{2}|\xi|} 
\nonumber
\\
&+ \displaystyle{\int_{0}^{t} \widehat{F}(\tau, \xi)e^{-\frac{|\xi|}{2}(t - \tau)}\frac{\text{sin}\left(\frac{\sqrt{3}}{2}|\xi|(t - \tau)\right)}{\frac{\sqrt{3}}{2}|\xi|} d\tau.}
\qquad\qquad\qquad
\end{eqnarray}
The inverse Fourier transform gives the solution $u$ in physical space:
\begin{eqnarray}\label{Fouriersoln2}
u(t, \cdot) &= e^{-\frac{\sqrt{-\Delta}}{2}t}\left(\text{cos}\left(\frac{\sqrt{3}}{2}\sqrt{-\Delta}t\right) + \frac{1}{\sqrt{3}}\text{sin}\left(\frac{\sqrt{3}}{2}\sqrt{-\Delta}t\right)\right)g + e^{-\frac{\sqrt{-\Delta}}{2}t}\frac{\text{sin}\left(\frac{\sqrt{3}}{2}\sqrt{-\Delta}t\right)}{\frac{\sqrt{3}}{2}\sqrt{-\Delta}}h
\nonumber
\\
&+ \displaystyle{\int_{0}^{t} e^{-\frac{\sqrt{-\Delta}}{2}(t - \tau)}\frac{\text{sin}\left(\frac{\sqrt{3}}{2}\sqrt{-\Delta}(t - \tau)\right)}{\frac{\sqrt{3}}{2}\sqrt{-\Delta}}F(\tau, \cdot) d\tau.}\qquad\qquad\qquad
\end{eqnarray}

Of considerable importance in future sections will be the effect of an inhomogeneous source term on the linear operator. 
In particular, the solution to \eqref{inhomlin} with zero initial data is given by the formula:
\begin{equation}\label{duhamel}
u(t, \cdot) = \displaystyle{\int_{0}^{t} e^{-\frac{\sqrt{-\Delta}}{2}(t - s)}\frac{\text{sin}\left(\frac{\sqrt{3}}{2}\sqrt{-\Delta}(t - s)\right)}{\frac{\sqrt{3}}{2}\sqrt{-\Delta}}F(s, \cdot) d\tau.}
\end{equation}

By recalling that the Fourier transform interchanges multiplication of functions and convolution, we can rewrite the formula \eqref{duhamel} in a more explicit manner. Let us define the kernel $K_{t}(x)$ by the inverse Fourier transform,
\begin{equation}\label{kernel}
K_{t}(x) = \frac{1}{(2\pi)^{n}} \int_{\mathbb{R}^{n}} e^{ix\cdot \xi} e^{-\frac{|\xi|}{2}t}\frac{\sin\left(\frac{\sqrt{3}}{2}|\xi|t\right)}{\frac{\sqrt{3}}{2}|\xi|} d\xi.
\end{equation}
To take advantage of the scaling of this PDE, we introduce the unit scale kernel $K(x)$, defined by
\begin{equation}\label{unitkernel}
K(x) = \frac{1}{(2\pi)^{n}} \int_{\mathbb{R}^{n}} e^{ix\cdot \xi} e^{-\frac{|\xi|}{2}}\frac{\sin\left(\frac{\sqrt{3}}{2}|\xi|\right)}{\frac{\sqrt{3}}{2}|\xi|} d\xi,
\end{equation}
which is just the kernel $K_{t}(x)$ at unit time $t = 1$. A simple change of variables shows the following crucial scaling relation:
\begin{equation}\label{viscouskernelscale}
K_{t}(x) = t^{1 - n} K\left(\frac{x}{t}\right).
\end{equation}

Equipped with the notation above, we can  rewrite \eqref{duhamel} in physical spatial variables as
\begin{equation}\label{convolution}
u(t, \cdot) = \int_{0}^{t} K_{t - s}(\cdot) * F(s, \cdot) ds = \int_{0}^{t} \int_{\mathbb{R}^{n}} K_{t - s}(x - y) F(s, y) dy ds,
\end{equation}
where the convolution operator $*$ denotes a convolution only in the spatial variables. 

The importance of using the kernel $K_{t}(x)$ to express the solution to the viscous wave equation explicitly as \eqref{convolution} lies in the fact that we have the following strong estimate for the unit-scale kernel $K(x)$, which carries over to the general kernel $K_{t}(x)$ by the scaling relation \eqref{viscouskernelscale}. The following lemma reflects the strong dissipative effects of the fluid viscosity, represented by the presence of the Dirichlet to Neumann operator. 

\begin{lemma}\label{kernelLq}
For all dimensions $n$, the kernel $K(x)$ is in $L^{q}(\mathbb{R}^{n})$ for all $1 \le q \le \infty$.
\end{lemma}

\begin{proof}
The proof of this lemma is by estimates using a repeated integration by parts. We refer the reader to the proof of Lemma 3.3 in \cite{KuanCanicNVWE}.
\end{proof}

\if 1 = 0
%%%%%%%%%%%%%%%%%%%%%%%%%
We can also give the following more explicit form of the kernel $K_{t}(x)$.

\begin{proposition}\label{kernelformula}
Define the corresponding kernel for the wave equation
\begin{equation*}
K^{W}_{t}(x) = \frac{1}{(2\pi)^{n}} \int_{\mathbb{R}^{n}} e^{ix\cdot \xi} \frac{\sin(|\xi| t)}{|\xi|} d\xi,
\end{equation*}
where superscript $W$ stands for the wave equation. Then, $K_{t}(x)$ in dimension $n$, defined by \eqref{kernel}, is given by the convolution
\begin{equation*}
K_{t}(x) = c_{n} \int_{\mathbb{R}^{n}} \frac{t}{(t^{2} + 4|x - y|^{2})^{\frac{n + 1}{2}}}\  K^{W}_{\frac{\sqrt{3}}{2}t}(y) dy,
\end{equation*}
where $c_{n}$ is a constant depending only on the dimension $n$. 
\end{proposition}

\begin{proof}
We use formula \eqref{kernel} and recall that the Fourier transform interchanges multiplication and convolutions. The inverse Fourier transform of $e^{-\frac{|\xi|}{2}t}$ is
\begin{equation*}
\frac{1}{(2\pi)^{n}} \int_{\mathbb{R}^{n}} e^{ix\cdot \xi} e^{-\frac{|\xi|}{2}t} d\xi = \tilde c_{n} \frac{t}{(t^{2} + 4|x - y|^{2})^{\frac{n + 1}{2}}},
\end{equation*}
where $\tilde c_{n}$ depends only on $n$. From the definition of $K_t^{W}$, we get
\begin{equation}\label{explicitkernel}
K_t^{W}(x) = \frac{1}{(2\pi)^{n}} \int_{\mathbb{R}^{n}} e^{ix\cdot \xi} \frac{\sin\left(\frac{\sqrt{3}}{2}|\xi|t\right)}{\frac{\sqrt{3}}{2}|\xi|} d\xi = \frac{2}{\sqrt{3}} K^{W}_{\frac{\sqrt{3}}{2}t}(x).
\end{equation}
The result then follows by using the fact that the Fourier transform interchanges multiplication and convolution, where we replaced the constant $\frac{2}{\sqrt{3}} \tilde c_{n}$ by $c_{n}$.

\end{proof}
{{This will be used later in TODO when we compare the behaviors of the stochastic viscous wave equation and the stochastic wave equation.}} 

%%%%%%%%%%%%%%%%%%%
\fi

This representation of solution will be important later in Section \ref{WellPosed}
when we discuss well-posedness   of the {\em{stochastic}} viscous wave equation. 
To compare the stochastic viscous wave equation with the stochastic heat and the stochastic wave equations
as will be done in Section \ref{WellPosed},
we now give the analogue of the above analysis using a convolution kernel for the heat and wave equations, focusing on the inhomogeneous forms of these equations with zero initial data. 

First, we consider the inhomogeneous heat equation.
Define the heat equation kernel and the corresponding unit scale kernel:
\begin{equation}\label{kernelheat}%\label{unitkernelheat}
K^{H}_{t}(x) = \frac{1}{(2\pi)^{n}} \int_{\mathbb{R}^{n}} e^{ix \cdot \xi} e^{-|\xi|^{2}t} d\xi = \frac{1}{(4\pi t)^{n/2}}e^{-\frac{x^{2}}{4t}},
\
K^{H}(x) = \frac{1}{(2\pi)^{n}} \int_{\mathbb{R}^{n}} e^{ix \cdot \xi} e^{-|\xi|^{2}} d\xi = \frac{1}{(4\pi)^{n/2}} e^{-\frac{x^{2}}{4}}.
\end{equation}
A simple change of variables shows the following scaling relation for the heat equation kernel:
\begin{equation}\label{scalingheat}
K^{H}_{t}(x) = t^{-\frac{n}{2}} K^{H}\left(\frac{x}{t^{1/2}}\right).
\end{equation}
In terms of the heat equation kernel, the solution to the inhomogeneous heat equation 
\begin{equation*}
u_{t} - \Delta u = F 
\end{equation*}
with zero initial data  is given by the formula:
\begin{equation}\label{heatformula}
u(t, x) = \int_{0}^{t} e^{-(t - s)\Delta} F(s, \cdot) ds = \int_{0}^{t} K^{H}_{t - s}(\cdot) * F(s, \cdot) ds = \int_{0}^{t} \int_{\mathbb{R}^{n}} K^{H}_{t - s}(x - y) F(s, y) dy ds.
\end{equation}

Note that in all dimensions $n$, and for all times $t > 0$, the kernel $K^{H}_{t}(x)$ defined in \eqref{kernelheat}, is function-valued and is, in fact, a Schwartz function. 

Next, we carry out the same analysis for the inhomogeneous wave equation. The wave equation kernel and the 
corresponding unit scale kernel can be defined similarly as
\begin{equation}\label{kernelwave}%\label{unitkernelwave}
K^{W}_{t}(x) = \frac{1}{(2\pi)^{n}} \int_{\mathbb{R}^{n}} e^{ix \cdot \xi} \frac{\sin(|\xi|t)}{|\xi|} d\xi,
\
K^{W}(x) = \frac{1}{(2\pi)^{n}} \int_{\mathbb{R}^{n}} e^{ix \cdot \xi} \frac{\sin(|\xi|)}{|\xi|} d\xi.
\end{equation}
The corresponding scaling relation is
\begin{equation}\label{scalingwave}
K^{W}_{t}(x) = t^{1 - n} K^{W}\left(\frac{x}{t}\right),
\end{equation}
which is the same as the scaling \eqref{viscouskernelscale} of the kernel for the viscous wave equation. 
The solution to the inhomogeneous wave equation 
\begin{equation*}
u_{tt} - \Delta u = F
\end{equation*}
with zero initial data is then given by the formula:
\begin{equation}\label{waveformula}
u(t, x) = \int_{0}^{t} \frac{\sin((t - s)\sqrt{-\Delta})}{\sqrt{-\Delta}} F(s, \cdot) ds = \int_{0}^{t} K^{W}_{t - s}(\cdot) * F(s, \cdot) ds = \int_{0}^{t} \int_{\mathbb{R}^{n}} K^{W}_{t - s}(x - y) F(s, y) dy ds.
\end{equation}

It is important to note that unlike the viscous wave and heat equation, the kernel $K_{t}(x)$ is no longer necessarily function-valued. In fact, we have, for example, the following well-known formulas for the kernel $K_{t}(x)$, giving the fundamental solution for the wave equation:
\begin{equation}\label{wavefundamental}
K^{W}_{t}(x) = 
\left\{
\begin{array}{ll}
\displaystyle{\frac{1}{2} 1_{|x| < t}} \qquad &\text{ for } n = 1,
\\
\displaystyle{\frac{1}{\sqrt{2\pi}} \frac{1}{\sqrt{t^{2} - |x|^{2}}} 1_{|x| < t}} \qquad &\text{ for } n = 2, 
\\
\displaystyle{\frac{1}{4\pi} \sigma_{t}(dx)} \qquad &\text{ for } n = 3,
\end{array}
\right.
\end{equation}
where $\sigma_{t}(dx)$ in the last expression denotes the surface measure on the sphere of radius $t$ centered at the origin. There are more complicated formulas for higher dimensions also, but $K^{W}_{t}(x)$ is function-valued only in dimensions one and two. In fact, $K^{W}_{t}(x)$ becomes increasingly singular as the dimension increases. 

%%%%%%%%%%%%%%%%%%%%%%%%%
It is interesting to note that the kernel for the wave equation $K^{W}_{t}$ can be tied to the kernel for the viscous wave equation $K_{t}$ by the following result.

\begin{proposition}\label{kernelformula}
The kernel $K_{t}(x)$ for the viscous wave equation in dimension $n$, defined by \eqref{kernel}, is given by the convolution
\begin{equation*}
K_{t}(x) = c_{n} \int_{\mathbb{R}^{n}} \frac{t}{(t^{2} + 4|x - y|^{2})^{\frac{n + 1}{2}}}\  K^{W}_{\frac{\sqrt{3}}{2}t}(y) dy,
\end{equation*}
where $c_{n}$ is a constant depending only on the dimension $n$. 
\end{proposition}

\begin{proof}
We use formula \eqref{kernel} and recall that the Fourier transform interchanges multiplication and convolutions. The inverse Fourier transform of $e^{-\frac{|\xi|}{2}t}$ is
\begin{equation*}
\frac{1}{(2\pi)^{n}} \int_{\mathbb{R}^{n}} e^{ix\cdot \xi} e^{-\frac{|\xi|}{2}t} d\xi = \tilde c_{n} \frac{t}{(t^{2} + 4|x|^{2})^{\frac{n + 1}{2}}},
\end{equation*}
where $\tilde c_{n}$ depends only on $n$. From the definition of $K_t^{W}$, we get
\begin{equation}\label{explicitkernel}
\frac{1}{(2\pi)^{n}} \int_{\mathbb{R}^{n}} e^{ix\cdot \xi} \frac{\sin\left(\frac{\sqrt{3}}{2}|\xi|t\right)}{\frac{\sqrt{3}}{2}|\xi|} d\xi = \frac{2}{\sqrt{3}} K^{W}_{\frac{\sqrt{3}}{2}t}(x).
\end{equation}
The result then follows by using the fact that the Fourier transform interchanges multiplication and convolution, where we replaced the constant $\frac{2}{\sqrt{3}} \tilde c_{n}$ by $c_{n}$.
\end{proof}
%{This will be used later in Section \ref{WellPosed} when we compare the behaviors of the stochastic viscous wave equation and the stochastic wave equation.} 

%%%%%%%%%%%%%%%%%%%

We have so far considered the inhomogeneous viscous wave equation with zero initial data. However, we will consider eventually the stochastic form of this equation with continuous bounded initial data, and will hence have to consider the full form of the solution given in \eqref{Fouriersoln2}, which takes into account the possibility of nonzero initial displacement and velocity. Note that the convolution kernel $K(x)$ defined in \eqref{kernel} can be used to describe the effect of an inhomogeneous source term and an initial velocity, as seen in \eqref{Fouriersoln2}. However, the kernel $K(x)$ does not describe the effect of an initial displacement $g$ on the solution. For this reason, we introduce the corresponding convolution kernel $J_{t}(x)$ and the respective unit scale kernel $J(x)$ associated to the propagation of $g(x)$ in \eqref{Fouriersoln2}:
\begin{equation}\label{Jkernel}
J_{t}(x) = \frac{1}{(2\pi)^{n}} \int_{\mathbb{R}^{n}} e^{ix \cdot \xi} e^{-\frac{|\xi|}{2}t}\left(\cos\left(\frac{\sqrt{3}}{2}|\xi| t\right) + \frac{1}{\sqrt{3}} \sin\left(\frac{\sqrt{3}}{2}|\xi|t\right)\right) d\xi,
\end{equation}
\begin{equation}\label{Junitkernel}
J(x) = \frac{1}{(2\pi)^{n}} \int_{\mathbb{R}^{n}} e^{ix \cdot \xi} e^{-\frac{|\xi|}{2}}\left(\cos\left(\frac{\sqrt{3}}{2}|\xi|\right) + \frac{1}{\sqrt{3}} \sin\left(\frac{\sqrt{3}}{2}|\xi| \right)\right) d\xi.
\end{equation}
A change of variables shows that
\begin{equation}\label{Jscale}
J_{t}(x) = t^{-n} J\left(\frac{x}{t}\right).
\end{equation}
We can then write the representation formula \eqref{Fouriersoln2} for the solution of the general viscous wave equation with nonzero initial data $u(0, x) = g(x)$ and $\partial_{t}u(0, x) = h(x)$ and inhomogeneous source term $F$, as 
\begin{equation}\label{generalviscous}
u(t, x) = \int_{\mathbb{R}^{n}} J_t(x - y) g(y) dy + \int_{\mathbb{R}^{n}} K_t(x - y) h(y) dy + \int_{0}^{t} \int_{\mathbb{R}^{n}} K_{t-s}(x - y) F(s, y) ds dy.
\end{equation}

In analogy to Lemma \ref{kernelLq}, one can show the following lemma, which shows that the unit scale kernel $J(x)$ has strong integrability properties. 

\begin{lemma}\label{Jlemma}
For all dimensions $n$, the kernel $J(x)$ is in $L^{q}(\mathbb{R}^{n})$ for all $1 \le q \le \infty$. Furthermore, we have the estimate, 
\begin{equation*}
|J(x)| \le C_{N} |x|^{-1 - n\left(\frac{N - 1}{N}\right)},
\end{equation*}
for any $N \ge n + 1$, where $C_{N}$ is a constant depending on $N$. 
\end{lemma}

\begin{proof}
We refer the reader to the proof of Lemma 3.3 in \cite{KuanCanicNVWE}. While the proof there is for the slightly different unit kernel $K(x)$, the corresponding proof for $J(x)$ is just a slight modification of the proof given there. 
\end{proof}

Finally, we establish a final lemma in this section, which shows the effect of the viscous wave operator on continuous functions $u(0, x) = g(x)$ and $\partial_{t}u(0, x) = h(x)$ that are both in $H^{2}(\mathbb{R}^{n})$. This will be useful when showing existence and uniqueness of a mild solution to the stochastic viscous wave equation with {{continuous initial data in $H^{2}(\mathbb{R}^{n})$}}, see Section \ref{WellPosed}. 

\begin{lemma}\label{boundedcontinuous}
Let $n = 1$ or $n = 2$, and let $g \in H^{2}(\mathbb{R}^{n})$ and $h \in H^{2}(\mathbb{R}^{n})$ be continuous functions on $\mathbb{R}^{n}$. Then, for any positive time $T > 0$, the solution to 
\begin{equation*}
u_{tt} + \sqrt{-\Delta} u_{t} - \Delta u = 0,
\end{equation*}
with initial data
\begin{equation*}
u(0, x) = g(x), \qquad \partial_{t}u(0, x) = h(x),
\end{equation*}
is a bounded, continuous function on $[0, T] \times \mathbb{R}^{n}$. Furthermore, the solution $u(t, x)$ has the following H\"{o}lder continuity properties depending on the dimension $n$:
\begin{itemize}
\item If $n = 1$, then for every $\rho \in [0, 1]$, there exists a constant $C_{\rho, T}$ depending only on $\rho$ and $T$ such that for all $t, t' \in [0, T]$, $x, x' \in \mathbb{R}$, 
\begin{equation*}
|u(t, x) - u(t, x')| \le C_{\rho, T} |x - x'|^{\rho}, \qquad |u(t, x) - u(t', x)| \le C_{\rho, T} |t - t'|^{\rho}.
\end{equation*}
\item If $n = 2$, then for every $\rho \in [0, 1)$, there exists a constant $C_{\rho, T}$ depending only on $\rho$ and $T$ such that for all $t, t' \in [0, T]$, $x, x' \in \mathbb{R}^{2}$, 
\begin{equation*}
|u(t, x) - u(t, x')| \le C_{\rho, T} |x - x'|^{\rho}, \qquad |u(t, x) - u(t', x)| \le C_{\rho, T} |t - t'|^{\rho/2}.
\end{equation*}
\end{itemize}
\end{lemma}

\begin{proof}
First we show that $u(t, x)$ is bounded. By Lemma \ref{kernelLq} and Lemma \ref{Jlemma}, $J, K \in L^{1}(\mathbb{R}^{n})$. Therefore, by using the scaling relations \eqref{Jscale} and \eqref{viscouskernelscale}, we have that for $t \in [0, T]$,
\begin{equation}\label{JL1}
\int_{\mathbb{R}^{n}} |J_{t}(x)| dx = t^{-n}\int_{\mathbb{R}^{n}} \left|J\left(\frac{x}{t}\right)\right| dx = ||J||_{L^{1}} < \infty,
\end{equation}
\begin{equation}\label{KL1}
\int_{\mathbb{R}^{n}} |K_{t}(x)| dx = t^{1 - n} \int_{\mathbb{R}^{n}} \left|K\left(\frac{x}{t}\right)\right| dx = t ||K||_{L^{1}} \le T ||K||_{L^{1}} < \infty.
\end{equation}
Using these facts along with the fact that $g, h$ are bounded (by Sobolev embedding since they are in $H^{2}(\mathbb{R}^{n})$),
the explicit formula
\begin{equation}\label{lemmasol} 
u(t, x) = \int_{\mathbb{R}^{n}} J_{t}(y) g(x - y) dy + \int_{\mathbb{R}^{n}} K_{t}(y) h(x - y) dy 
\end{equation}
implies that $u(t, x)$ is bounded on $[0, T] \times \mathbb{R}^{n}$. 

Next, we establish continuity. First, we consider spatial increments. Since $g, h \in H^{2}(\mathbb{R}^{n})$ are continuous, they are $\rho$-H\"{o}lder continuous for $\rho \in [0, 1)$ by Sobolev embedding, and in fact also Lipschitz continuous in dimension one. Then, for $x, x' \in \mathbb{R}^{n}$ and $t > 0$,
\begin{align}\label{spatialestimate}
|u(t, x) - u(t, x')| &\le \int_{\mathbb{R}^{n}} |J_{t}(y)| \cdot |g(x - y) - g(x' - y)| dy + \int_{\mathbb{R}^{n}} |K_{t}(y)| \cdot |h(x - y) - h(x' - y)| dy \nonumber \\
&\le C_{\rho} |x - x'|^{\rho} \left(\int_{\mathbb{R}^{n}} |J_{t}(y)| dy + \int_{\mathbb{R}^{n}} |K_{t}(y)| dy\right) \le C_{\rho, T} |x - x'|^{\rho},
\end{align}
where $\rho = [0, 1]$ if $n = 1$, and $\rho \in [0, 1)$ if $n = 2$, with $C_{\rho}$ depending on $\rho$. In particular, the Lipschitz or H\"{o}lder continuity of the initial data is propagated in time on a finite time interval. We have also shown that $u(t, x)$ at each fixed time is a continuous function. 

Next, we consider time increments. Consider $0 \le t' < t \le T$. We want to estimate  the quantity $|u(t, x) - u(t', x)|$ for arbitrary $x \in \mathbb{R}^{n}$. We consider the two cases of $n = 1$ and $n = 2$ separately. 

\vspace{0.1in}

\noindent \textbf{Case 1:} If $n = 1$, then we have that $(u, u_{t}) \in C([0, T]; H^{2}(\mathbb{R})) \times C([0, T]; H^{1}(\mathbb{R}))$. Since $H^{1}(\mathbb{R})$ embeds continuously into the bounded continuous functions on $\mathbb{R}$, $u_{t}$ is bounded and continuous on $[0, T] \times \mathbb{R}$. Hence, $|u_{t}(t, x)| \le C_{T}$, $\forall x \in \mathbb{R}$ and $t \in [0, T]$. By the fundamental theorem of calculus,
\begin{equation}\label{dim1time}
|u(t, x) - u(t', x)| \le \int_{t'}^{t} |u_{t}(s, x)| ds \le C_{T} |t - t'|.
\end{equation}

%\vspace{0.1in}

\noindent  \textbf{Case 2:} If $n = 2$, by uniqueness of the solution in $C([0, T]; H^{2}(\mathbb{R}^{2})) \times C([0, T]; H^{1}(\mathbb{R}^{2}))$,
we can consider $u(t', x)$ and $\partial_{t}u(t', x)$ as initial data at time $t'$, to get
\begin{equation*}
|u(t, x) - u(t', x)| = \left|\int_{\mathbb{R}^{n}} J_{t - t'}(y) u(t', x - y) dy - u(t', x) + \int_{\mathbb{R}^{n}} K_{t - t'}(y) \partial_{t}u(t', x - y) dy\right|.
\end{equation*}
Since
$
\int_{\mathbb{R}^{n}} J_{t}(y) dy = \widehat{J_{t}}(\xi = 0) = 1,
$
the following estimate holds:
\begin{align}\label{timeestimate}
|u(t, x) - u(t', x)| &\le \int_{\mathbb{R}^{n}} |J_{t - t'}(y)| \cdot |u(t', x - y) - u(t', x)| dy + \int_{\mathbb{R}^{n}} |K_{t - t'}(y)| \cdot |\partial_{t}u(t', x - y)| dy \nonumber \\
&:= I_{1} + I_{2}.
\end{align}
To complete the estimate, we first consider integral $I_{1}$. {{We break up the integral $I_{1}$ into two parts,}}
\begin{align}
I_{1} &= \int_{|y| \le |t - t'|^{1/2}} |J_{t - t'}(y)| \cdot |u(t', x - y) - u(t', x)| dy + \int_{|y| > |t - t'|^{1/2}} |J_{t - t'}(y)| \cdot |u(t', x - y) - u(t', x)| dy 
\nonumber
\\
&= I_{1, 1} + I_{1, 2}.
\label{splitI1}
\end{align}
Using the H\"{o}lder continuity in space from \eqref{spatialestimate}, and using the estimate \eqref{JL1}, we get {{for $\rho \in [0, 1)$,}}
\begin{equation}\label{I11}
I_{1, 1} \le C_{\rho, T} \int_{|y| \le |t - t'|^{1/2}} |J_{t - t'}(y)| \cdot |y|^{\rho} dy \le C_{\rho, T} |t - t'|^{\rho/2} \int_{|y| \le |t - t'|^{1/2}} |J_{t - t'}(y)| dy \le C_{\rho, T} |t - t'|^{\rho/2}.
\end{equation}
To estimate $I_{1, 2}$, we recall that 
we already showed that $u(t, x)$ is bounded on $[0, T] \times \mathbb{R}^{n}$ by some constant $M_{T}$. 
Therefore, by the scaling relation \eqref{Jscale}, and by using a change of variables, we get
\begin{align}\label{I12first}
I_{1, 2} &\le 2M_{T} \int_{|y| > |t - t'|^{1/2}} |J_{t - t'}(y)| dy = 2M_{T} (t - t')^{-n} \int_{|y| > |t - t'|^{1/2}} \left|J\left(\frac{y}{t - t'}\right)\right| dy \nonumber \\
&= 2M_{T} \int_{|z| > |t - t'|^{-1/2}} |J(z)| dz.
\end{align}
To estimate the last integral, we recall the  estimate stated in Lemma \ref{Jlemma}
%\begin{equation*}
%|J(x)| \le C_{N} |x|^{-1 - n\left(\frac{N - 1}{N}\right)}, \qquad \text{ for } N \ge n + 1,
%\end{equation*}
and choose an $N$ in that estimate (which depends on $\rho$) $N_{\rho} \ge n + 1$, sufficiently large so that 
\begin{equation}\label{Nrho}
1 + n\left(\frac{N_{\rho} - 1}{N_{\rho}}\right) \ge n + \rho,\ \ {\rm or\ equvalently}\  \ \rho \le 1 - \frac{n}{N_{\rho}},
\end{equation}
for arbitrary $\rho \in [0, 1)$. 
%This condition  is equivalent to
%\begin{equation}\label{Nrhoequiv}
%\rho \le 1 - \frac{n}{N_{\rho}}.
%\end{equation}
Then, continuing from \eqref{I12first} and switching to polar coordinates,  the estimate from Lemma \ref{Jlemma}, together with the inequality  \eqref{Nrho}, imply
{{\begin{align}\label{I12}
I_{1, 2} \le 2M_{T}{{C_{N_\rho}}} \int_{|z| > |t - t'|^{-1/2}} \frac{1}{|z|^{1 + n\left(\frac{N_{\rho} - 1}{N_{\rho}}\right)}} dz = 2M_{T}C_{N_\rho}|\mathbb{S}^{n - 1}|  \int_{|t - t'|^{-1/2}}^{\infty} r^{-1 - n\left(\frac{N_{\rho} - 1}{N_{\rho}}\right)} r^{n - 1} dr \nonumber \\
= 2M_{T}C_{N_{\rho}} |\mathbb{S}^{n - 1}| \int_{|t - t'|^{-1/2}}^{\infty} r^{-2 + \frac{n}{N_{\rho}}} dr = C_{\rho, T} |t - t'|^{\frac{1}{2}\left(1 - \frac{n}{N_{\rho}}\right)} \le C_{\rho, T} |t - t'|^{\rho/2},
\end{align}
for $\rho \in [0, 1)$, where $C_{N_{\rho}}$ denotes the constant for $N = N_{\rho}$ in the inequality in Lemma \ref{Jlemma}.}} In the last inequality, we used the fact that $t', t$ belong to a bounded interval $[0, T]$,
 and in the last step, with a slight abuse of notation, we used the same notation
 for the constant $C_{T, \rho}$. 
 
 {{Finally, we estimate
 \begin{equation*}
 I_{2} := \int_{\mathbb{R}^{n}} |K_{t - t'}(y)| \cdot |\partial_{t}u(t', x - y)| dy.
 \end{equation*}
 Since $(u, \partial_{t}u) \in C([0, T]; H^{2}(\mathbb{R}^{n})) \times C([0, T]; H^{1}(\mathbb{R}^{n}))$, we have that $\partial_{t}u(t, \cdot)$ is uniformly bounded in $H^{1}(\mathbb{R}^{n})$ for $t \in [0, T]$. 
 We note that for $n = 2$, $H^{1}(\mathbb{R}^{n})$ embeds into $L^{q}(\mathbb{R}^{n})$ for all $2 \le q < \infty$. This is because for general dimension $n$, if a function $f \in H^{n/2}(\mathbb{R}^{n})$, we can show that for all $1 < p \le 2$, $\widehat{f} \in L^{p}(\mathbb{R}^{n})$, which implies the result by the Hausdorff-Young inequality. Using H\"{o}lder's inequality with the conjugate exponents $2/p$ and $2/(2 - p)$, one can compute:
\begin{multline*}
\int_{\mathbb{R}^{n}} |\widehat{f}(\xi)|^{p} d\xi = \int_{\mathbb{R}^{n}} (1 + |\xi|^{2})^{-np/4} (1 + |\xi|^{2})^{np/4} |\widehat{f}(\xi)|^{p} d\xi \\
\le \left(\int_{\mathbb{R}^{n}} (1 + |\xi|^{2})^{-\frac{np}{2(2 - p)}} d\xi\right)^{\frac{2 - p}{2}} \cdot ||f||_{H^{n/2}(\mathbb{R}^{n})}^{p} = C_{p} ||f||_{H^{n/2}(\mathbb{R}^{n})}^{p},
\end{multline*}
since $\frac{p}{2 - p} > 1$ for $1 < p \le 2$. Hence, for $q$ such that $2 \le q < \infty$,  by the Hausdorff-Young inequality,
we have  that for conjugate exponents $p$ and $q$, 
\begin{equation*}
||f||_{L^{q}(\mathbb{R}^{n})} \le ||\widehat{f}||_{L^{p}(\mathbb{R}^{n})} \le C_{p}||f||_{H^{n/2}(\mathbb{R}^{n})},
\end{equation*}
so that $H^{1}(\mathbb{R}^{2})$ embeds continuously into $L^{q}(\mathbb{R}^{2})$ for $2 \le q < \infty$. 

Hence, since $\partial_{t}u(t, \cdot)$ is uniformly bounded in $H^{1}(\mathbb{R}^{n})$ on $t \in [0, T]$, we have that
\begin{equation}\label{dtuLq}
\sup_{0 \le t \le T} ||\partial_{t}u(t, \cdot)||_{L^{q}(\mathbb{R}^{2})} \le C_{T, q}, \qquad \text{ for } 2 \le q < \infty.
\end{equation}
In addition, we use the scaling property for the kernel $K_{t}(x)$ given by \eqref{viscouskernelscale} to deduce that
\begin{align}
||K_{t}(x)||_{L^{p}(\mathbb{R}^{n})} &= t^{1 - n} \left(\int_{\mathbb{R}^{n}} \left|K\left(\frac{x}{t}\right)\right|^{p} dx\right)^{1/p} = t^{1 - n + \frac{n}{p}} \left(\int_{\mathbb{R}^{n}} |K(y)|^{p} dy\right)^{1/p}
\nonumber \\
&= t^{1 - \frac{n}{q}} ||K||_{L^{p}(\mathbb{R}^{n})} = C_{p} t^{1 - \frac{n}{q}},
\label{KtLp}
\end{align}
{{where $p$ and $q$ are conjugate exponents.}} Therefore,  by \eqref{dtuLq} and \eqref{KtLp}, 
\begin{equation}\label{finalI2}
I_{2} \le \tilde{C}_{T, q} |t - t'|^{1 - \frac{n}{q}}, 
\end{equation}
{{for $2 \le q < \infty$,}} where $n = 2$. By choosing $2 \le q < \infty$ appropriately, this implies the desired estimate
\begin{equation*}
I_{2} \le C_{T, \rho} |t - t'|^{\rho/2}
\end{equation*}
for $\rho \in [0, 1)$. 
 
 By combining \eqref{splitI1}, \eqref{I11}, \eqref{I12}, and \eqref{finalI2}, we get the desired result:
\begin{equation}\label{timeestimate}
|u(t, x) - u(t', x)| \le C_{\rho, T} |t - t'|^{\rho/2},
\end{equation}
for $\rho \in [0, 1)$ in the case $n = 2$.

\vspace{0.2in}

The spatial estimate \eqref{spatialestimate}, and the time estimates \eqref{dim1time} for $n = 1$ and \eqref{timeestimate} for $n = 2$, which are \textit{uniform on $[0, T] \times \mathbb{R}^{n}$}, establish the continuity of $u(t, x)$ on $[0, T] \times \mathbb{R}^{n}$. }}
\end{proof}

\subsection{White noise and stochastic integration}

In this section we review the concept of spacetime white noise on $\mathbb{R}^{+} \times \mathbb{R}^{n}$ and stochastic integration against white noise. This will be used throughout the rest of the manuscript. 
Note that we will use $\mathbb{R}^{+}$ to denote $[0, \infty)$, which represents the time variable. 
While we will be primarily concerned with dimensions $n = 1, 2$, we will define white noise in full generality, as the extension to higher dimensions is no more difficult. 

{{We follow the exposition that can be found in \cite{KMini} about martingale measures and refer the reader to the original reference by Walsh \cite{Walsh} for more details.}} We note that while the forthcoming analysis can be carried out more generally for martingale measures, we will restrict to the case of white noise for simplicity. The full martingale measure theory can be found in \cite{Walsh} and \cite{KMini}. 

Recall that a \textit{Gaussian process} is a process $\{G_{i}\}_{i \in I}$, such that the finite dimensional random vectors
\begin{equation*}
(G_{i_{1}}, G_{i_{2}}, ..., G_{i_{k}}), \ i_{1}, i_{2}, ..., i_{k} \in I
\end{equation*}
have distributions that are multivariable Gaussian, for any finite collection of $i_{1}, i_{2}, ..., i_{k} \in I$.
%
%distributions, which are the distributions of the random vector
%\begin{equation*}
%(G_{i_{1}}, G_{i_{2}}, ..., G_{i_{k}})
%\end{equation*}
%for any finite collection of $i_{1}, i_{2}, ..., i_{k} \in I$, are multivariate Gaussian. We will say that a Gaussian process is \textit{mean zero} if
%\begin{equation*}
%\mathbb{E}(G_{i}) = 0, \qquad \text{ for all } i \in I.
%\end{equation*}

We will define the \textit{covariance function} to be the symmetric function $C: I \times I \to \mathbb{R}$ that gives the covariance of any two Gaussians $G_{i_{1}}$ and $G_{i_{2}}$,
\begin{equation*}
C(i_{1}, i_{2}) = \mathbb{E}[(G_{i_{1}} - \mathbb{E}(G_{i_1}))(G_{i_{2}} - \mathbb{E}(G_{i_{2}}))].
\end{equation*}
For a mean zero Gaussian process, {{which is a Gaussian process $\{G_{i}\}_{i \in I}$ such that $\mathbb{E}[G_{i}] = 0$ for all $i \in I$,}} this reduces to the simpler formula
\begin{equation*}
C(i_{1}, i_{2}) = \mathbb{E}[G_{i_{1}}G_{i_{2}}].
\end{equation*}

We will now define white noise as a Gaussian process, taking for granted the existence of such a process. 
\begin{definition}[White noise on $\mathbb{R}^{+} \times \mathbb{R}^{n}$]
Let $\mathcal{B}(\mathbb{R}^{+} \times \mathbb{R}^{n})$ denote the collection of all Borel subsets of $\mathbb{R}^{+} \times \mathbb{R}^{n}$.
White noise on $\mathbb{R}^{+} \times \mathbb{R}^{n}$ is a \textit{mean zero} Gaussian process $\{\dot{W}(A)\}_{A \in \mathcal{B}(\mathbb{R}^{+} \times \mathbb{R}^{n})}$ indexed by the Borel subsets of $\mathbb{R}^{+} \times \mathbb{R}^{n}$, with the covariance function
\begin{equation}\label{cov}
C(A, B) := \mathbb{E}[\dot{W}(A)\dot{W}(B)] = \lambda(A \cap B), \qquad \text{ for } A, B \in \mathcal{B}(\mathbb{R}^{+} \times \mathbb{R}^{n}),
\end{equation}
where $\lambda$ is Lebesgue measure in $\mathbb{R}^{+} \times \mathbb{R}^{n}$. 
\end{definition}

Some basic facts about white noise that will be useful later are summarized in the following proposition.

\begin{proposition}\label{whitenoiseprop}
Let $\{\dot{W}(A)\}_{A \in \mathcal{B}(\mathbb{R}^{+} \times \mathbb{R}^{n})}$ denote white noise. Then, the following holds true:
\begin{itemize}
\item For each {{bounded set }}$A \in \mathcal{B}(\mathbb{R}^{+} \times \mathbb{R}^{n})$, $\dot{W}(A)$ is normally distributed with mean $0$ and variance $\lambda(A)$, namely $\dot{W}(A) \sim N(0, \lambda(A))$. So $\dot{W}(A) \in L^{2}(\Omega)$, where $\Omega$ is the probability space. 
\item If $A \cap B = \varnothing$, then $\dot{W}(A)$ and $\dot{W}(B)$ are independent.
\item Given $A, B \in \mathcal{B}(\mathbb{R}^{+} \times \mathbb{R}^{n})$, 
$
\dot{W}(A \cup B) = \dot{W}(A) + \dot{W}(B) - \dot{W}(A \cap B),
$
almost surely (a.s.), as random variables.
\item White noise is a signed measure taking values in $L^{2}(\Omega)$, namely $\dot{W}: \mathcal{B}(\mathbb{R}^{+} \times \mathbb{R}^{n}) \to L^{2}(\Omega)$. Furthermore, white noise considered as a measure is $\sigma$-finite. 
\end{itemize}
\end{proposition}

\begin{proof}
The first point follows from the fact that white noise is a mean zero Gaussian process, and
$
\mathbb{E}[(\dot{W}(A))^{2}] = \lambda(A \cap A) = \lambda (A)
$
by \eqref{cov}. The second and third points are from Exercise 3.15 in \cite{KMini}. The second point follows from the fact that $\dot{W}(A)$ and $\dot{W}(B)$ are mean zero Gaussians with zero covariance, by applying \eqref{cov}. The third fact follows from the computation of the expectation
$
\mathbb{E}[(\dot{W}(A \cup B) - \dot{W}(A) - \dot{W}(B) + \dot{W}(A \cap B))^{2}] = 0.
$
One can verify this by expanding the square and applying \eqref{cov} repeatedly. Note that the third property gives the finite additivity properties of a measure. For the final property, one must check that white noise has the remaining properties of a measure, and we refer the reader to the proof of Proposition 5.1 in \cite{KMini}.
\end{proof}

\begin{remark}
Heuristically, one thinks of white noise as random noise that is ``independent" at every point in time and space. One can then interpret $\dot{W}(A)$ heuristically as being the net contribution of the noise in $A$. With this heuristic interpretation, it is at least intuitively reasonable that white noise has the properties of a measure. The fact that the noise is independent at every point in time and space is in accordance with the second property in Proposition \ref{whitenoiseprop}.
\end{remark}

{\bf Stochastic integration against white noise.}  We will first define integration of simple functions against white noise, and then proceed to the most general case by an approximation argument. For this purpose, we introduce the following nomenclature (see Sec.~5 of \cite{KMini}):
\begin{itemize}
\item  For any $A \in \mathcal{B}(\mathbb{R}^{n})$ and $t \in \mathbb{R}^+$, we use $W_{t}(A)$ to denote
$
W_{t}(A) = \dot{W}([0, t] \times A), 
$
for $A \in \mathcal{B}(\mathbb{R}^{n})$,
so that $[0, t] \times A \in \mathcal{B}(\mathbb{R}^{+} \times \mathbb{R}^{n})$. 
\item For $t > 0$, we consider the {\emph{filtration}} $\mathcal{F}_{t}$ associated to white noise to be the $\sigma$-algebra generated by the collection of random variables 
$
\{W_{s}(A) : s \in [0, t], A \in \mathcal{B}(\mathbb{R}^{n})\}.
$ 
\item We use $\mathcal{S}$ to denote the space of simple functions, which are functions of the form
\begin{equation}\label{simpledef}
f(t, x, \omega) = \sum_{i = 1}^{n} X_{i}(\omega) 1_{(a_{i}, b_{i}]}(t) 1_{A_{i}}(x),
\end{equation}
where $X_{i}$ is a bounded, $\mathcal{F}_{a_{i}}$-measurable random variable with $0 \le a_{i} < b_{i}$, and $A_{i} \in \mathcal{B}(\mathbb{R}^{n})$ is bounded. 

\end{itemize}

\begin{definition}
Let $f \in \mathcal{S}$ be a simple function. 
We define 
\begin{align}\label{simpleint}
\int_{0}^t \int_{\mathbb{R}^{n}} f(s,x,\omega) W(dx,ds) & = \int_{0}^t \int_{\mathbb{R}^{n}} \sum_{i = 1}^{n} X_{i}(\omega) 1_{(a_{i}, b_{i}]}(s) 1_{A_{i}}(x) W(dx, ds) 
\nonumber \\
&:=  \sum_{i = 1}^{n} X_{i}(\omega) [W_{t \wedge b_{i}}(A_{i}) - W_{t \wedge a_{i}}(A_{i})],
\end{align}
where the ``wedge'' notation corresponds to $\alpha \wedge \beta = \min\{\alpha,\beta\}$.
\end{definition}
It is easy to check that the definition of the integral in \eqref{simpleint} is independent of the representation of the simple function as \eqref{simpledef}. 

We have the following crucial isometry property for the stochastic integral of simple functions against spacetime white noise. This is an extension of the Itô isometry to the stochastic integral against spacetime white noise. 

\begin{proposition}\label{isometry}
For $f \in \mathcal{S}$,
\begin{equation}\label{Ito}
\mathbb{E}\left[\left(\int_{0}^{\infty} \int_{\mathbb{R}^{n}} f(t, x, \omega) W(dx, dt)\right)^{2}\right] = \mathbb{E}\left(\int_{0}^{\infty} \int_{\mathbb{R}^{n}} |f(t, x, \omega)|^{2} dx dt\right).
\end{equation}
\end{proposition}

\begin{proof}
In the case where $f(t, x, \omega)$ is a simple function of the form
\begin{equation*}
f(t, x, \omega) = X(\omega)1_{(a, b]}(t)1_{A}(x),
\end{equation*}
one easily checks that
\begin{multline*}
\mathbb{E}\left[\left(\int_{0}^{\infty} \int_{\mathbb{R}^{n}} f(t, x, \omega) W(dx, dt)\right)^{2}\right] = \mathbb{E}\left(X^{2}(\omega) \left[W_{b}(A) - W_{a}(A)\right]^{2}\right) \\
= \mathbb{E}\left(\mathbb{E}\left[X^{2}(\omega) \left(W_{b}(A) - W_{a}(A)\right)^{2} | \mathcal{F}_{a}\right]\right) = \mathbb{E} \left(X^{2}(\omega) \mathbb{E}\left[\left(W_{b}(A) - W_{a}(A)\right)^{2} | \mathcal{F}_{a}\right]\right),
\end{multline*}
where we used the fact that $X \in \mathcal{F}_{a}$ to take it out of the conditional expectation. Using the third property in Proposition \ref{whitenoiseprop},
\begin{equation*}
\mathbb{E}\left[\left(\int_{0}^{\infty} \int_{\mathbb{R}^{n}} f(t, x, \omega) W(dx, dt)\right)^{2}\right] = \mathbb{E}\left(X^{2}(\omega) \mathbb{E}\left[\dot{W}^{2}((a, b] \times A) | \mathcal{F}_{a} \right]\right).
\end{equation*}
%If $0 \le t \le a$, this is zero. If $t > a$, t
%
Using the second property in Proposition \ref{whitenoiseprop} we deduce that this is equal to 
\begin{equation*}
= \mathbb{E}\left[ X^{2}(\omega)\right] \mathbb{E}\left[\dot{W}^{2}((a, b] \times A))\right] = \lambda(A) (b - a) \mathbb{E}\left[X^{2}(\omega)\right] = \mathbb{E}\left(\int_{0}^{\infty} \int_{\mathbb{R}^{n}} |f(t, x, \omega)|^{2} dx dt\right).
\end{equation*}
%Since the rightmost side is equal to $0$ if $0 \le t \le a$, we have the desired result. 

We note that \eqref{Ito} holds for general $f \in \mathcal{S}$, by choosing a representation \eqref{simpledef} of an arbitrary simple function where the sets {{$(a_{i} \times b_{i}] \times A_{i}$}} are disjoint, and then using the independence property in the second property listed in Proposition \ref{whitenoiseprop}. 
\end{proof}

Next, we want to extend the definition of the stochastic integral to more general integrands. 
For this purpose we recall the following definitions.
\begin{definition}
Let $f(t, x, \omega)$ be a real valued function $f: \mathbb{R}^{+} \times \mathbb{R}^{n} \times \Omega \to \mathbb{R}$.
\begin{enumerate}
\item We say that $f(t, x, \omega)$ is {\emph{adapted to the filtration}} $\{\mathcal{F}_{t}\}_{t \ge 0}$ if the map $\omega \to f(t, x, \omega)$ is $\mathcal{F}_{t}$ measurable for each $x \in \mathbb{R}^{n}$ and $t \ge 0$. 
\item We say that $f(t, x, \omega)$ is {\emph{jointly measurable}} if it is measurable as a function in time, space, and the probability space, $f: \mathbb{R}^{+} \times \mathbb{R}^{n} \times \Omega \to \mathbb{R}$.
\end{enumerate}
\end{definition}

To define the stochastic integral, we must identify the class of admissible integrands, which will be called predictable processes  \cite{DF}.
To do that, we denote by  $\mathcal{H}$ the set of all jointly measurable $f(t, x, \omega)$ such that 
\begin{equation*}
\mathbb{E}\left(\int_{0}^{\infty} \int_{\mathbb{R}^{n}} f(t, x, \omega)^{2} dx dt\right) < \infty.
\end{equation*}
Note that $\mathcal{S} \subset \mathcal{H}$. 

\begin{definition}
Define $\mathcal{P}_{W}$ to be the the closure of $\mathcal{S} \subset \mathcal{H}$ under the norm 
\begin{equation}\label{predictablenorm}
||f||_{\mathcal{P}_{W}}^{2} := \mathbb{E}\left(\int_{0}^{\infty} \int_{\mathbb{R}^{n}} f(t, x, \omega)^{2} dx dt\right) < \infty.
\end{equation}
The elements of $\mathcal{P}_{W}$ are called {\emph{predictable processes}}.
\end{definition}

Finally, we define the stochastic integral for predicable processes, namely 
\begin{equation}\label{generalstochastic}
\int_{0}^{\infty} \int_{\mathbb{R}^{n}} f(t, x, \omega) W(dx, dt), \qquad \text{ for } f \in \mathcal{P}_W,
\end{equation}
 by utilizing a density argument that uses the Itô isometry. In particular, we use the fact that functions $\mathcal{S}$ are dense in $\mathcal{P}_{W}$ (see Proposition 2.3 in \cite{Walsh}). Hence, given $f \in \mathcal{P}_W$, there is a sequence $f_{k} \in \mathcal{S}$ such that $f_{k} \to f$ in $\mathcal{P}_{W}$, as $k \to \infty$. 
 Using the isometry relation in Proposition \ref{isometry}, one can show that the sequence
\begin{equation}\label{int_seq}
\left\{ \int_{0}^{\infty} \int_{\mathbb{R}^{n}} f_{k}(t, x, \omega) W(dx, dt)\right\}_{k = 1}^\infty
\end{equation}
is a Cauchy sequence in $L^{2}(\Omega)$.
\begin{definition}
The random variable obtained in the limit of integrals \eqref{int_seq} is the {\emph{stochastic integral}} \eqref{generalstochastic}.
\end{definition}

We can also define the integral on bounded time intervals, by noting that
\begin{equation*}
\int_{0}^{T} \int_{\mathbb{R}^{n}} f(t, x, \omega) W(dx, dt) = \int_{0}^{\infty} \int_{\mathbb{R}^{n}} 1_{(0, T]}(t) f(t, x, \omega) W(dx, dt).
\end{equation*}

Since the definition of the admissible integrands $\mathcal{P}_{W}$ is abstract, we list a set of criteria that will help us determine whether a given integrand is in $\mathcal{P}_{W}$ or not. Hence, we use the following proposition, which follows directly
from Proposition 2 in \cite{DF}.

{{
\begin{proposition}\label{PW}
Let $\{u(t, x)\}_{t \in [0, T], x \in \mathbb{R}^{n}}$ be a stochastic process adapted to the filtration $\{\mathcal{F}_{t}\}_{t \ge 0}$ such that the following conditions hold.
\begin{enumerate}
\item Joint measurability: $(t, x, \omega) \to u(t, x, \omega)$ is $\mathcal{B}([0, T] \times \mathbb{R}^{n}) \times \mathcal{F}_{T}$ measurable.
\item Finite second moments: $\mathbb{E}\left(|u(t, x)|^{2}\right) < \infty$ for all $t \in [0, T]$, $x \in \mathbb{R}^{n}$.
\item Continuity in $L^{2}(\Omega)$: The process $u$ considered as a map $(t, x) \in [0, T] \times \mathbb{R}^{n} \to L^{2}(\Omega)$ is continuous in $L^{2}(\Omega)$.
\item Square integrability:
$
\displaystyle{\mathbb{E}\left(\int_{0}^{T} \int_{\mathbb{R}^{n}} |u(s, y)|^{2} dy ds\right) < \infty.}
$
\end{enumerate}
Then, the stochastic integral
\begin{equation*}
\int_{0}^{t} \int_{\mathbb{R}^{n}} u(s, y) W(ds, dy)
\end{equation*}
is defined for all $t \in [0, T]$. 
\end{proposition}

\begin{proof}
This proposition follows from Proposition 2 of \cite{DF}, and is Proposition 2 of \cite{DF} adapted to the current context. Though Proposition 2 of \cite{DF} is stated for the more general case of spatially homogeneous Gaussian noise, the statement of Proposition 2 of \cite{DF} specialized to the case of white noise reads as follows:

Let $\{u(t, x)\}_{t \in \mathbb{R}^{+}, x \in \mathbb{R}^{n}}$ be a stochastic process adapted to the filtration $\{\mathcal{F}_{t}\}_{t \ge 0}$ and define $\mathcal{F} = \bigcup_{t \ge 0} \mathcal{F}_{t}$. Suppose the following conditions hold:
\begin{enumerate}
\item Joint measurability: $(t, x, \omega) \to u(t, x, \omega)$ is $\mathcal{B}(\mathbb{R}_{+} \times \mathbb{R}^{n}) \times \mathcal{F}$ measurable.
\item Finite second moments: $\mathbb{E}\left(|u(t, x)|^{2}\right) < \infty$ for all $t \in \mathbb{R}^{+}$, $x \in \mathbb{R}^{n}$.
\item Continuity in $L^{2}(\Omega)$: The process $u$ considered as a map $(t, x) \in \mathbb{R}^{+} \times \mathbb{R}^{n} \to L^{2}(\Omega)$ is continuous in $L^{2}(\Omega)$.
\item Square integrability on a compact set and finite time: There exists a compact set $K \subset \mathbb{R}^{n}$ and $t_{0} > 0$ such that 
\begin{equation*}
\mathbb{E}\left(\int_{0}^{t_{0}} \int_{K} |u(s, y)|^{2} dy ds\right) < \infty.
\end{equation*}
\end{enumerate}
Then, $1_{[0, t_{0}] \times K}(t, x) u(t, x) \in \mathcal{P}_{W}$.

\vspace{0.1in}

While the result in Proposition 2 of \cite{DF} is stated specifically for spatial dimension two, one can verify that it holds for arbitrary dimension.

To see that the statement of Proposition 2 of \cite{DF} implies the result in Proposition \ref{PW}, let $K_{i}$ be a sequence of compact sets that increase to $\mathbb{R}^{n}$, and consider $\{u(t, x)\}_{t \in [0, T], x \in \mathbb{R}^{n}}$ satisfying the four conditions in Proposition \ref{PW}. We extend $\{u(x, t)\}_{t \in [0, T], x \in \mathbb{R}^{n}}$ to be defined on all of time $t \ge 0$ by defining
\begin{equation*}
\tilde{u}(x, t) = u(x, t) \quad \text{ if } t \in [0, T], \qquad \tilde{u}(x, t) = u(x, T) \quad \text{ if } t \ge T.
\end{equation*}
Then, $\{\tilde{u}(t, x)\}_{t \in \mathbb{R}^{+}, x \in \mathbb{R}^{n}}$ along with $t_{0} = T$ and each $K_{i}$ satisfies the conditions in Proposition 2 in \cite{DF}. Therefore, $1_{[0, T] \times K_{i}}(t, x) u(t, x) = 1_{[0, T] \times K_{i}}(t, x) \tilde{u}(t, x) \in \mathcal{P}_{W}$. 

Since the fourth condition of Proposition \ref{PW} states that
\begin{equation*}
\mathbb{E}\left(\int_{0}^{T} \int_{\mathbb{R}^{n}} |u(s, y)|^{2} dy ds\right) < \infty,
\end{equation*}
we have that $1_{[0, T] \times K_{i}}(t, x) u(t, x) \to 1_{[0, T] \times \mathbb{R}^{n}}(t, x) u(t, x)$ in the norm of $\mathcal{P}_{W}$, since $K_{i}$ is a sequence of compact sets in $\mathbb{R}^{n}$ increasing to all of $\mathbb{R}^{n}$. Hence, $1_{[0, T] \times \mathbb{R}^{n}}(t, x) u(t, x) \in \mathcal{P}_{W}$ since $\mathcal{P}_{W}$ is complete with respect to its norm.
\end{proof}
}}

A couple of remarks are in order. The first one uses the concept of {\emph{modification}}, which we now recall.
\begin{definition}\label{def:modification}
Let $\{u(t, x)\}_{t \in [0, T], x \in \mathbb{R}^{n}}$ be a stochastic process. We say that
$\{\tilde{u}(t, x)\}_{t \in [0, T], x \in \mathbb{R}^{n}}$ is a \textit{modification} of $\{u(t, x)\}_{t \in [0, T], x \in \mathbb{R}^{n}}$ if 
\begin{equation*}
\mathbb{P}(u(t, x) = \tilde{u}(t, x)) = 1, \qquad \text{ for all } t \in [0, T], x \in \mathbb{R}^{n}.
\end{equation*}
\end{definition} 

\begin{remark}\label{checkless}
The third condition in Proposition~\ref{PW}  implies that there is a jointly measurable modification ({{see the discussion
on pg. 201 of \cite{DF}}}, and the proof of Theorem 13 in \cite{Dalang}). Thus, in practice, one does not need to check the first condition, as by taking a modification, the third condition implies the first. 
\end{remark}

Finally, we recall the following useful inequality, {{which is a direct consequence of a classical result}} known as the BDG (Burkholder-Davis-Gundy) inequality, which will be used frequently  \cite{KMini}.

\begin{theorem}\label{BDG}
For each $p \ge 2$, there exists a positive constant $c_{p}$ depending only on $p$ (and not on $T$) such that
\begin{equation*}
\mathbb{E}\left(\left|\int_{0}^{T} \int_{\mathbb{R}^{n}} f(t, x, \omega) W(dx, dt)\right|^{p}\right) \le c_{p} \mathbb{E}\left(\left(\int_{0}^{T} \int_{\mathbb{R}^{n}} |f(t, x, \omega)|^{2} dx dt\right)^{p/2}\right),
\end{equation*}
for all $f \in \mathcal{P}_{W}$. 
\end{theorem}

\section{The stochastic viscous wave equation in dimensions $n = 1, 2$}\label{whitenoise}

We are now in the position to study the stochastic viscous wave equation:
\begin{equation}\label{sec3eqn}
u_{tt} + \sqrt{-\Delta}u_{t} - \Delta u = f(u) W(dx, dt), \qquad \text{ in } \mathbb{R}^{n},
\end{equation}
with initial data:
\begin{equation}\label{initial}
u(0, x) = g(x), \qquad \partial_{t}u(0, x) = h(x).
\end{equation}
For simplicity, we assume that $g$ and $h$ are continuous functions in $H^{2}(\mathbb{R}^{n})$, $f: \mathbb{R} \to \mathbb{R}$ is Lipschitz continuous, and $W(dx, dt)$ is spacetime white noise. In particular, since $f$ is Lipschitz continuous, there exists a constant $L > 0$ such that
\begin{equation*}
|f(x) - f(y)| \le L|x - y|, \qquad \text{ for all } x, y \in \mathbb{R},
\end{equation*}
\begin{equation}\label{growth}
|f(x)| \le L(1 + |x|), \qquad \text{ for all } x \in \mathbb{R}.
\end{equation}
We will show that the Cauchy problem \eqref{sec3eqn}, \eqref{initial} has a \textit{mild solution} in the sense of a stochastic process satisfying a stochastic {{integral equation}} (see Definition~\ref{milddef} below), which is function-valued in dimensions $n = 1, 2$. This is in contrast to the corresponding stochastic heat and wave equations,
\begin{equation}\label{stheat}
u_{t} - \Delta u = f(u) W(dx, dt), \qquad \text{ in } \mathbb{R}^{n},
\end{equation}
\begin{equation}\label{stwave}
u_{tt} - \Delta u = f(u) W(dx, dt), \qquad \text{ in } \mathbb{R}^{n},
\end{equation}
which have function-valued mild solutions only in dimension $n = 1$. 
%%%%%%%%%%%%
\if 1 = 0
Therefore, since the viscous wave equation has behavior that is intuitively ``between" the heat and wave equations, it is surprising that the stochastic equation \eqref{sec3eqn} allows function-valued mild solutions in dimension two. {{TODO DISCUSS OTHER POSSIBILITIES FOR THE VISCOUS OPERATOR}}. This is of particular interest in the context of fluid-structure interaction, because dimension two is the physical dimension. 
\fi
%%%%%%%%%%%%%

In the next section we review the concept of mild solution for the stochastic heat and wave equations, and demonstrate the well-known fact that there are function-valued mild solutions only in dimension one. We then consider the concept of mild solutions for the stochastic viscous wave equation, showing heuristically why we will be able to consider such solutions in dimension two. In Section \ref{WellPosed} we rigorously prove existence and uniqueness of a mild solution to \eqref{sec3eqn}, \eqref{initial} using a Picard iteration argument to deal with the nonlinearity $f(u)$.

\subsection{The concept of mild solution}

To define the concept of mild solution for the {\emph{stochastic}} viscous wave equation \eqref{sec3eqn}, we first recall the solution 
for the {\emph{deterministic}} inhomogeneous problem \eqref{inhomlin}.
Namely, as shown earlier, the solution to the {\emph{deterministic}} inhomogeneous problem \eqref{inhomlin}
%\begin{equation*}
%u_{tt} + \sqrt{-\Delta} u_{t} - \Delta u = F,
%\end{equation*}
with initial data $u(0, x) = g(x)$ and $\partial_{t}u(0, x) = h(x)$, is given by the formula
\begin{equation}\label{detviscoussol}
u(t, x) = \int_{\mathbb{R}^{n}} J_{t}(x - y) g(y) dy + \int_{\mathbb{R}^{n}} K_{t}(x - y) h(y) dy + \int_{0}^{t} \int_{\mathbb{R}^{n}} K_{t - s}(x - y) F(s, y) dy ds,
\end{equation}
where $J_{t}(x)$ is defined by \eqref{Jkernel}, and $K_{t}(x)$  by \eqref{kernel}. 
For the general {\emph{stochastic}} case \eqref{sec3eqn} with initial data \eqref{initial}, 
we can formally regard the stochastic forcing $f(u) W(dx, dt)$ as the forcing term $F$ in the deterministic equation, and
formally require that the solution $u$ to the stochastic viscous wave equation satisfy the  {\emph{stochastic integral equation}} alla \eqref{detviscoussol}:
\begin{equation*}
u(t, x, \omega) = \int_{\mathbb{R}^{n}} J_{t}(x - y) g(y) dy + \int_{\mathbb{R}^{n}} K_{t}(x - y) h(y) dy + \int_{0}^{t} \int_{\mathbb{R}^{n}} K_{t - s}(x - y) f(u(s, y, \omega)) W(dy, ds).
\end{equation*}
The result of this formal argument gives rise to the concept of a mild solution.

\begin{definition}\label{milddef}
A stochastic process $u(t, x)$ is a \textit{mild solution} to the stochastic viscous wave equation \eqref{sec3eqn} with initial data \eqref{initial} if $u(t, x)$ is  jointly measurable and adapted to the filtration ${\cal{F}}_t$ with
\begin{equation}\label{mildsoln}
u(t, x, \omega) = \int_{\mathbb{R}^{n}} J_{t}(x - y) g(y) dy + \int_{\mathbb{R}^{n}} K_{t}(x - y) h(y) dy + \int_{0}^{t} \int_{\mathbb{R}^{n}} K_{t - s}(x - y) f(u(s, y, \omega)) W(dy, ds),
\end{equation}
and the stochastic integral on the right hand side of \eqref{mildsoln} is defined. 
\end{definition}

\begin{remark}[Probabilistic notation]
{{In the remainder of this manuscript, we will generally follow the probabilistic convention of not writing the explicit $\omega$ dependence of random variables and stochastic processes. In particular, while we wrote out the explicit $\omega$ dependence in the stochastic process $u(t, x, \omega)$ in \eqref{mildsoln}, we will henceforth omit the explicit $\omega$ dependence when it is clear from context that the mathematical quantity involved is a random variable. For example, we would write 
\begin{equation*}
u(t, x) = \int_{\mathbb{R}^{n}} J_{t}(x - y) g(y) dy + \int_{\mathbb{R}^{n}} K_{t}(x - y) h(y) dy + \int_{0}^{t} \int_{\mathbb{R}^{n}} K_{t - s}(x - y) f(u(s, y)) W(dy, ds),
\end{equation*}
for the full expression in \eqref{mildsoln}.}}
\end{remark}

We can define the concept of a mild solution to the stochastic heat and the stochastic wave equations \eqref{stheat} and \eqref{stwave} in the same way using the deterministic heat and wave equation representation formulas for the solutions
of the corresponding inhomogeneous equation, 
given in \eqref{heatformula} and \eqref{waveformula}.

As mentioned earlier, 
the existence of a function-valued mild solution to the stochastic heat and wave equations \eqref{stheat} and \eqref{stwave},
defined this way, can be obtained only in dimension $n = 1$, as was discussed in \cite{DMini} and \cite{KMini}.
However, we will be able to prove the existence of a function-valued mild solution to the stochastic viscous wave equation \eqref{sec3eqn} in both dimensions $n = 1, 2$. To give an idea of why we might expect this to be true, we present the following heuristic argument.

{\bf{A heuristic argument for the existence of a mild solution to \eqref{sec3eqn}, \eqref{initial} in $n=2$.}}
For simplicity, let $f: \mathbb{R} \to \mathbb{R}$ on the right hand-side in \eqref{sec3eqn}, \eqref{stheat}, and \eqref{stwave} be identically equal to 1,
so that we can just consider the case of additive noise. Therefore, we consider the equations
\begin{equation}\label{easystheat}
u_{t} - \Delta u = W(dx, dt), \qquad \text{ on } \mathbb{R}^{n},
\end{equation}
\begin{equation}\label{easystwave}
u_{tt} - \Delta u = W(dx, dt), \qquad \text{ on } \mathbb{R}^{n},
\end{equation}
\begin{equation}\label{easystviscous}
u_{tt} + \sqrt{-\Delta} u_{t} - \Delta u = W(dx, dt), \qquad \text{ on } \mathbb{R}^{n}.
\end{equation}
Furthermore, for simplicity, we consider \textit{zero initial data} for the purposes of this heuristic argument. 

%We first consider the stochastic heat and wave equations \eqref{easystheat} and \eqref{easystwave}. 
For the stochastic heat equation with additive noise \eqref{easystheat}, we have an explicit formula for the solution as a stochastic integral,
\begin{equation*}
u(t, x) = \int_{0}^{t} \int_{\mathbb{R}^{n}} K^{H}_{t - s}(x - y) W(ds, dy),
\end{equation*}
where the kernel $K^{H}_{t}$ is defined by \eqref{kernelheat}. For this stochastic integral to make sense, we must have 
\begin{equation}\label{heatcondition}
\int_{0}^{t} \int_{\mathbb{R}^{n}} |K^{H}_{t - s}(x - y)|^{2} dy ds < \infty.
\end{equation}
Using the scaling relation \eqref{scalingheat}, we can rewrite condition \eqref{heatcondition} in terms of the unit kernel as
\begin{align}\label{heatcalc}
\int_{0}^{t} \int_{\mathbb{R}^{n}} & |K^{H}_{t - s}(x - y)|^{2} dy ds = \int_{0}^{t} \int_{\mathbb{R}^{n}} (t - s)^{-n} \left|K^{H}\left(\frac{x - y}{(t - s)^{1/2}}\right)\right|^{2} dy ds \nonumber \\
&= \int_{0}^{t} \int_{\mathbb{R}^{n}} (t - s)^{-n/2} |K^{H}(y)|^{2} dy ds = \left(\int_{0}^{t} (t - s)^{-n/2} ds\right) ||K^{H}||^{2}_{L^{2}(\mathbb{R}^{n})} < \infty.
\end{align}
Because $K^{H}$ is a Gaussian in all dimensions $n$, we have that $||K^{H}||^{2}_{L^{2}(\mathbb{R}^{n})} < \infty$ for all $n$. However, the time integral $\int_{0}^{t} (t - s)^{-n/2} ds$ only converges in dimension $n = 1$. This is the  reason why a function-valued mild solution to the stochastic heat equation with additive noise \eqref{easystheat} exists only in dimension $1$.  

Let us carry out a similar analysis for the stochastic wave equation with additive noise \eqref{easystwave}. A mild solution, if it exists, must be given by the stochastic integral
\begin{equation*}
u(t, x) = \int_{0}^{t} \int_{\mathbb{R}^{n}} K^{W}_{t - s}(x - y) W(ds, dy),
\end{equation*}
where $K^{W}_{t}$ is defined by \eqref{kernelwave}. This stochastic integral exists only if the following integrability condition is satisfied:
\begin{equation}\label{wavecondition}
\int_{0}^{t} \int_{\mathbb{R}^{n}} |K^{W}_{t - s}(x - y)|^{2} dy ds < \infty.
\end{equation}
By using the scaling relation \eqref{scalingwave}, this condition can be rewritten as
\begin{align}\label{wavecalc}
\int_{0}^{t} \int_{\mathbb{R}^{n}} & |K^{W}_{t - s}(x - y)|^{2} dy ds = \int_{0}^{t} \int_{\mathbb{R}^{n}} (t - s)^{2 - 2n} \left|K^{W}\left(\frac{x - y}{t - s}\right)\right|^{2} dy ds \nonumber \\
&= \int_{0}^{t} \int_{\mathbb{R}^{n}} (t - s)^{2 - n} |K^{W}(y)|^{2} dy ds = \left(\int_{0}^{t} (t - s)^{2 - n} ds\right) ||K^{W}||^{2}_{L^{2}(\mathbb{R}^{n})} < \infty.
\end{align}
Note here that the time integral $\int_{0}^{t} (t - s)^{2 - n}$ converges for $n = 1, 2$. However, it is easy to check from the explicit form of the fundamental solution in \eqref{wavefundamental} that $K^{W}$ is in $L^{2}(\mathbb{R}^{n})$ only for dimension $n = 1$. This is the reason why a function-valued mild solution to the stochastic wave equation with additive noise \eqref{easystwave} 
exists only in dimension one.

The stochastic viscous wave equation with additive noise \eqref{easystviscous} is exactly in between these two cases,
with both factors of the unit kernel integrable in both $n = 1$ and $n = 2$. 
Namely, a function-valued mild solution to \eqref{easystviscous} would be defined by
\begin{equation*}
u(t, x) = \int_{0}^{t} \int_{\mathbb{R}^{n}} K_{t - s}(x - y) W(ds, dy),
\end{equation*}
where $K_{t}(x)$ is the kernel given by \eqref{kernel}. Using the scaling relation \eqref{viscouskernelscale} involving the unit kernel \eqref{unitkernel}, we compute, similarly as in the previous examples, that this integral exists only if the following integrability condition is satisfied:
\begin{equation}\label{viscouscalc}
\int_{0}^{t} \int_{\mathbb{R}^{n}} |K_{t - s}(x - y)|^{2} dy ds = \left(\int_{0}^{t} (t - s)^{2 - n} ds\right) ||K||^{2}_{L^{2}(\mathbb{R}^{n})} < \infty.
\end{equation}
However, by Lemma \ref{kernelLq}, $K \in L^{2}(\mathbb{R}^{n})$ for all $n$. Therefore, this integrability condition is satisfied in both dimensions one and two. Thus, the stochastic viscous wave equation with additive noise \eqref{easystviscous} has a function-valued mild solution \textit{both in dimensions one and two}. 

\if 1 = 0 %%%%%%%%%%%%%%%%%%%%
While initially surprising, the above heuristic argument shows why we get a function-valued mild solution in dimension two for \eqref{easystviscous}. 

The heat equation does not have a function-valued mild solution in dimension two because, as one can see in \eqref{heatcalc}, the time integration does not converge, due to the spacetime scaling of the heat equation of one time derivative against two spatial derivatives. However, the heat equation has the desirable property that the kernel is in $L^{2}$. 

Meanwhile, as seen in \eqref{wavecalc}, the wave equation does not have a function-valued mild solution in dimension two because its corresponding kernel is not in $L^{2}$ in dimension two, even though the time integral now does converge (due to the spacetime scaling of only one time derivative against one spatial derivative). 
\fi %%%%%%%%%%%%%%%

%Therefore, the stochastic viscous wave equation with additive noise has a function-valued mild solution in dimension two because the equation has the same desirable spacetime scaling as the wave equation, but the addition of the dissipative viscous term allows the kernel $K(x) \in L^{2}(\mathbb{R}^{2})$, which in contrast to the wave equation. 
%This allows us to get better behavior for the stochastic viscous wave equation. 

{{Finally, we note that the nature of the parabolic damping is essential for the stochastic viscous wave equation to have a function-valued mild solution in dimension two also. In particular, the stochastic damped wave equation
\begin{equation}\label{sdamped}
u_{tt} + cu_{t} - \Delta u = W(dx, dt), \qquad \text{ on } \mathbb{R}^{n}
\end{equation}
with $c > 0$ and zero initial data, has a function-valued mild solution only in dimension one. To see this, we use the explicit formula for the fundamental solution from \cite{Dalang} in frequency space,
\begin{equation}\label{Fourierdamped}
\widehat{K}^{DW}_{t}(\xi) = (c^{2} - |\xi|^{2})^{-1/2} e^{-ct} \sinh\left(t\sqrt{c^{2} - |\xi|^{2}}\right),
\end{equation}
so that the mild solution if it exists for \eqref{sdamped} must be given by
\begin{equation*}
u(t, x) = \int_{0}^{t} \int_{\mathbb{R}^{n}} K^{DW}_{t - s}(x - y) W(ds, dy).
\end{equation*}
{{Here, the superscript $DW$ indicates that we are considering the fundamental solution for the damped wave equation.}} Thus, the integrability condition for the mild solution to exist is that
\begin{equation*}
\int_{0}^{t} \int_{\mathbb{R}^{n}} |K^{DW}_{t - s}(x - y)|^{2} dy ds < \infty.
\end{equation*}
This is equivalent, by Plancherel's theorem, to
\begin{equation}\label{dampedint}
\int_{0}^{t} \int_{\mathbb{R}^{n}} |K^{DW}_{t - s}(x - y)|^{2} dy ds = \int_{0}^{t} \int_{\mathbb{R}^{n}} |K^{DW}_{t - s}(y)|^{2} dy ds = \int_{0}^{t} \int_{\mathbb{R}^{n}} |\widehat{K}^{DW}_{t - s}(\xi)|^{2} d\xi ds < \infty.
\end{equation}

However, the condition \eqref{dampedint} holds only in dimension $n = 1$. This is because $\widehat{K}_{t}(\xi) \notin L^{2}(\mathbb{R}^{n})$ for $n \ge 2$ for all $t > 0$. To see this, note that for $|\xi| > c$, the explicit formula \eqref{Fourierdamped} gives that
\begin{equation*}
\widehat{K}^{DW}_{t}(\xi) = (|\xi|^{2} - c^{2})^{-1/2} e^{-ct} \sin\left(t\sqrt{|\xi|^{2} - c^{2}}\right), \qquad \text{ for } |\xi| > c.
\end{equation*}
We then compute that 
\begin{align*}
||\widehat{K}^{DW}_{t}||_{L^{2}(\mathbb{R}^{n})}^{2} &\ge e^{-2ct} \int_{|\xi| > c} (|\xi|^{2} - c^{2})^{-1} \sin^{2}\left(t\sqrt{|\xi|^{2} - c^{2}}\right) d\xi \\
&= \alpha_{n - 1} e^{-2ct} \int_{c}^{\infty} \frac{r^{n - 1}}{r^{2} - c^{2}} \sin^{2}\left(t\sqrt{r^{2} - c^{2}}\right) dr,
\end{align*}
where $\alpha_{n - 1}$ denotes the surface area of the sphere $S^{n - 1} \subset \mathbb{R}^{n}$. We use a change of variables,
\begin{equation*}
\rho = \sqrt{r^{2} - c^{2}} \ \Longleftrightarrow \ r = \sqrt{\rho^{2} + c^{2}} \qquad d\rho = \frac{r}{\sqrt{r^{2} - c^{2}}} dr.
\end{equation*}
Then,
\begin{equation*}
||\widehat{K}^{DW}_{t}||_{L^{2}(\mathbb{R}^{n})}^{2} \ge \alpha_{n - 1}e^{-2ct} \int_{0}^{\infty} \frac{\sin^{2}(t\rho)}{\rho} (\rho^{2} + c^{2})^{\frac{n}{2} - 1} d\rho.
\end{equation*}
So for $n \ge 2$, we have that $\frac{n}{2} - 1 \ge 0$ and hence for any $t > 0$,
\begin{equation*}
||\widehat{K}^{DW}_{t}||_{L^{2}(\mathbb{R}^{n})}^{2} \ge \alpha_{n - 1}e^{-2ct}c^{n - 2} \int_{0}^{\infty} \frac{\sin^{2}(t\rho)}{\rho} d\rho = \infty \qquad \text{ for } n \ge 2,
\end{equation*}
since this integral diverges. Thus, the integrability condition \eqref{dampedint} does not hold in dimensions two and higher and holds only in dimension one. So the stochastic damped wave equation has a function-valued mild solution only in dimension one. 
}}
%{{TODO: Do we want to say something about the results of the deterministic nonlinear wave equation?
%See Remark 2.2}}

\subsection{Existence and uniqueness for the stochastic viscous wave equation}\label{WellPosed}

While the heuristic argument above was done for a simpler case of additive white noise when $f(u)=1$, we can get an existence and uniqueness result for the more general equation \eqref{sec3eqn} with a general, Lipschitz $f(u)$ in dimensions one and two, by a standard Picard iteration procedure, and by estimates of the kernel. {{We then obtain estimates on the higher moments of the solution for later use in Section \ref{Holder}. Such a Picard iteration and higher moment bound procedure are standard in the stochastic PDE literature \cite{DMini, KMini, Walsh, Dalang}}}. More precisely, we have the following main result.

\begin{theorem}[{\bf{Existence and uniqueness}}]\label{existence}
Let $n = 1$ or $n = 2$, and let $g$ and $h$ be continuous functions in $H^{2}(\mathbb{R}^{n})$. Suppose $f: \mathbb{R} \to \mathbb{R}$ is a Lipschitz continuous function. Then, there exists a  function-valued mild solution to the equation
\begin{equation}\label{ModelEquation}
u_{tt} + \sqrt{-\Delta} u_{t} - \Delta u = f(u) W(dx, dt) \qquad \text{ on } \mathbb{R}^{n}
\end{equation}
with initial data $u(0, x) = g(x)$, $\partial_{t}u(0, x) = h(x)$, which is unique up to stochastic modification.
\end{theorem}
\begin{proof}
To establish existence, we use  Picard iterations. We begin by setting the first iterate $u_{0}$ to be the deterministic function
\begin{equation}\label{initialiterate}
u_{0}(t, x) = \int_{\mathbb{R}^{n}} J_{t}(x - y) g(y) dy + \int_{\mathbb{R}^{n}} K_{t}(x - y) h(y)dy,
\end{equation}
which is the solution to the deterministic linear {{homogeneous}} viscous wave equation with initial data given by $g$ and $h$. By Lemma \ref{boundedcontinuous}, $u_{0}(t, x)$ is a bounded, continuous function on $[0, T] \times \mathbb{R}^{n}$. 

Then, define the Picard iterates $u_{k}$ for $k \ge 1$ inductively by
\begin{equation}\label{Picarddef}
u_{k}(t, x) = u_{0}(t, x) + \int_{0}^{t} \int_{\mathbb{R}^{n}} K_{t - s}(x - y) f(u_{k - 1}(s, y)) W(dy, ds),
\end{equation}
where $u_{0}$ captures the deterministic evolution of the initial data $g$ and $h$. However, we must check that the stochastic integral on the right hand side makes sense. This is the content of the following lemma.

\begin{lemma}\label{Picardwell}
The Picard iteration procedure \eqref{Picarddef} is well-defined at each step. Furthermore,
\begin{equation*}
\sup_{t \in [0, T]} \sup_{x \in \mathbb{R}^{n}} \mathbb{E}\left(|u_{k}(t, x)|^{2}\right) < \infty, \qquad \text{ for all } k \ge 0, T \ge 0.
\end{equation*}
\end{lemma}

The proof of Lemma~\ref{Picardwell} is given in the Appendix.

\begin{remark}
{{Note that because random variables are only defined up to a measure zero set, the $k$th Picard iterate $\{u_{k}(t, x)\}_{t \in [0, T], x \in \mathbb{R}^{n}}$ is defined only up to stochastic modification. However, as we show in the proof of Lemma \ref{Picardwell}, there exists a modification of $\{u_{k}(t, x)\}_{t \in [0, T], x \in \mathbb{R}^{n}}$ for which the stochastic integral 
\begin{equation*}
\int_{0}^{t} \int_{\mathbb{R}^{n}} K_{t - s}(x - y) f(u_{k}(s, y)) W(dy, ds)
\end{equation*}
is defined, where this stochastic integral is needed to obtain the next iterate $u_{k + 1}$. Hence, when defining $u_{k}$ at each point $(t, x) \in [0, T] \times \mathbb{R}^{n}$ in \eqref{Picarddef}, we choose the modification that allows the stochastic integral needed for the next step of the Picard iteration to be defined. Note that all of the arguments that follow are well suited to the fact that the Picard iterates $u_{k}$ are defined only up to stochastic modification. For example, we will later consider the quantity for each $k$ and $t \in [0, T]$,
\begin{equation*}
\sup_{0 \le s \le t, x \in \mathbb{R}^{n}} \mathbb{E}[(u_{k} - u_{k - 1})^{2}(s, x)],
\end{equation*}
when studying convergence of the iterates, and this quantity is unchanged by stochastic modification of any of the individual iterates.}}
\end{remark}

The next step is to show that $u_{k}(t, x)$ converge in an appropriate sense as $k \to \infty$,
and that the limit is a unique mild solution to the stochastic viscous wave equation \eqref{ModelEquation}. 

{\bf Convergence.} We start by considering the difference between consecutive iterates:
\begin{equation}\label{diffeqn}
u_{k}(t, x) - u_{k - 1}(t, x) = \int_{0}^{t}\int_{\mathbb{R}^{n}} K_{t - s}(x - y) [f(u_{k - 1}(s, y)) - f(u_{k - 2}(s, y))] W(dy, ds),
\ k\ge 2.
\end{equation}
Using the Itô isometry \eqref{Ito} and the fact that $f$ is Lipschitz continuous with a global Lipschitz constant $L$, 
we obtain
\begin{align*}
\mathbb{E}[(u_{k} - u_{k - 1})^{2}(t, x)] &\le L^{2} \mathbb{E}\left(\int_{0}^{t} \int_{\mathbb{R}^{n}} K_{t - s}^{2}(x - y) |u_{k - 1}(s, y) - u_{k - 2}(s, y)|^{2} dy ds\right) \\
&= L^{2} \int_{0}^{t} \int_{\mathbb{R}^{n}} K^{2}_{t - s}(x - y) \mathbb{E}[(u_{k - 1} - u_{k - 2})^{2}(s, y)] dy ds,
\end{align*}
where we used Fubini's theorem in the last step. 

Let
\begin{equation}\label{Hndef}
H_{k}^{2}(t) := \sup_{0 \le s \le t, x \in \mathbb{R}^{n}} \mathbb{E}[(u_{k} - u_{k - 1})^{2}(s, x)].
\end{equation}
We want to show that for every {{$t \ge 0$, $\sum_{k = 1}^{\infty} H_k(t) < \infty$, as this would imply that $\{u_{k}(t, x)\}_{k = 0}^{\infty}$ is a Cauchy sequence in $L^{2}(\Omega)$ for each $t \ge 0, x \in \mathbb{R}^{n}$}}.
Indeed, first notice that the following inequality holds:
\begin{equation}\label{Picardineq}
H_{k}^{2}(t) \le L^{2} \int_{0}^{t} \int_{\mathbb{R}^{n}} K_{t - s}^{2}(x - y) H_{k - 1}^{2}(s) dy ds.
\end{equation}
To further estimate the right hand side, we
estimate the kernel using a calculation as in \eqref{viscouscalc} for dimensions $n = 1$ and $2$ to obtain 
\begin{equation}\label{Ktminuss2}
\int_{\mathbb{R}^{n}} K^{2}_{t - s}(x - y) dy = (t - s)^{2 - n} ||K||^{2}_{L^{2}(\mathbb{R}^{n})} = c_{n} (t - s)^{2 - n} \le c_{n} t^{2 - n},
\end{equation}
for some constant $c_{n}$ depending only on $n$, for $s \in [0, t]$. 
Combining this estimate with \eqref{Picardineq}, one obtains
\begin{equation}\label{inductiveineq}
H_{k}^{2}(t) \le c_{n, t}\int_{0}^{t} H_{k - 1}^{2}(s) ds,\ k = 2,3,\dots
\end{equation}
for a finite constant $c_{n, t}$ that depends only on $t$ and the dimension $n = 1, 2$. 
We will use this inequality inductively, for $k = 2,3,\dots$ to obtain the desired result.
For this purpose, we must first show that $H_{1}(t)$ is finite. In particular, recalling \eqref{Hndef} and using the result in Lemma \ref{Picardwell}, we have
\begin{equation*}
H_{1}^{2}(t) \le 2\left(\sup_{0 \le s \le t, x \in \mathbb{R}^{n}} \mathbb{E}\left(|u_{1}(s, x)|^{2}\right) + \sup_{0 \le s \le t, x \in \mathbb{R}^{n}} \mathbb{E}\left(|u_{0}(s, x)|^{2}\right)\right) = A_{t} < \infty,
\end{equation*}
where $A_{t}$ is a constant depending only on $t$. {{Hence, by inductively using \eqref{inductiveineq}, we have that
\begin{equation}\label{Hk2bound}
H_{k}^{2}(t) \le \frac{A_{t} \cdot (c_{n, t})^{k}t^{k}}{k!}.
\end{equation}
Thus,
\begin{equation}\label{series}
\sum_{k = 1}^{\infty} H_{k}(t) \le A_{t}^{1/2} \sum_{k = 0}^{\infty} \frac{(c_{n, t})^{k/2} t^{k/2}}{(k!)^{1/2}} < \infty,
\end{equation}
as this series converges. Recalling the definition of $H_{k}^{2}(t)$ in \eqref{Hndef}, we conclude that $\{u_{k}(t, x)\}_{k = 0}^{\infty}$ for each $t \ge 0, x \in \mathbb{R}^{n}$ is a Cauchy sequence in $L^{2}(\Omega)$. Hence, $u_{k}(t, x)$ converges in $L^{2}(\Omega)$ 
to some $u(t, x)$ for each $t \ge 0, x \in \mathbb{R}^{n}$.}} 

{\bf Existence of a mild solution.} We now  show that the limit $u(t, x)$ is a mild solution to \eqref{ModelEquation}.
Indeed, after passing to the limit on both sides of \eqref{Picarddef} we immediately see that the left hand side of \eqref{Picarddef} converges to $u(t, x)$ in $L^{2}(\Omega)$.  
To deal with the limit on the right hand side of \eqref{Picarddef}, we first calculate the following estimate:
by the Lipschitz property of $f$ 
and the Itô isometry \eqref{Ito} we have
\begin{multline}\label{mildconv}
\left|\left|\int_{0}^{t} \int_{\mathbb{R}^{n}} K_{t - s}(x - y) [f(u_{k - 1}(s, y)) - f(u(s, y))] W(dy, ds)\right|\right|_{L^{2}(\Omega)} \\
\le L^{2} \int_{0}^{t} \int_{\mathbb{R}^{n}} K_{t - s}^{2}(x - y) \mathbb{E}\left(|u_{k - 1}(y, s) - u(y, s)|^{2}\right) dy ds.
\end{multline}
To further estimate the right hand side, we recall the convergence of the series \eqref{series} and the definition \eqref{Hndef} of $H_{k}^{2}(t)$, to conclude:
\begin{equation}\label{uniformconv}
\mathbb{E}(|u_{k - 1}(s, y) - u(s, y)|^{2}) \to 0 \qquad \text{ as } k \to \infty, \qquad \text{uniformly for $0 \le s \le t$ and $y \in \mathbb{R}^{n}$}.
\end{equation}
Additionally, by recalling \eqref{Ktminuss2}, we get:
\begin{equation}\label{sqestimatekernel}
\int_{0}^{t}\int_{\mathbb{R}^{n}} K^{2}_{t - s}(x - y) dy ds = \frac{c_{n}}{3 - n} t^{3 - n}.
\end{equation}
Therefore, combining this equality 
 with \eqref{uniformconv}, and using it in the right hand side of \eqref{mildconv}, we obtain that for every fixed $t \ge 0$, the following
 convergence result holds:
\begin{equation*}
\int_{0}^{t} \int_{\mathbb{R}^{n}} K_{t - s}(x, y) f(u_{k - 1}(s, y)) W(dy, ds) \to \int_{0}^{t} \int_{\mathbb{R}^{n}} K_{t - s}(x, y) f(u(s, y)) W(dy, ds) \qquad \text{ in } L^{2}(\Omega).
\end{equation*}
This shows that $u$ satisfies \eqref{mildsoln}. 

{{To complete the proof that $u$ is a mild solution, we must show according to Definition \ref{milddef} that $\{u(t, x)\}_{t \in \mathbb{R}^{+} \times \mathbb{R}^{n}}$ is jointly measurable and adapted to $\mathcal{F}_{t}$. Since each $u_{k}$ is adapted to $\mathcal{F}_{t}$, so is the limit $u$. In addition, by the uniform convergence \eqref{uniformconv}, $u$ is continuous in $L^{2}(\Omega)$ on $\mathbb{R}^{+} \times \mathbb{R}^{n}$ since each $u_{k}$ has this property, by the proof of Lemma \ref{Picardwell} in the Appendix. Hence, by Remark \ref{checkless}, $u$ has a stochastic modification that is jointly measurable. This completes the proof that $u$ is a mild solution.}}

{\bf Uniqueness.} Uniqueness follows from Gronwall's inequality. More precisely,
 suppose that $u$ and $v$ are both mild solutions with the same initial data \eqref{initial}. Then, their difference $w := u - v$ satisfies the following stochastic integral equation:
\begin{equation*}
w(t, x) = \int_{0}^{t} \int_{\mathbb{R}^{n}} K_{t - s}(x - y) [f(u(s, y)) - f(v(s, y))] W(dy, ds).
\end{equation*}
Taking the $L^{2}(\Omega)$ norm of both sides, we get that
\begin{equation*}
\mathbb{E}[w^{2}(t, x)] \le L^{2}\int_{0}^{t}\int_{\mathbb{R}^{n}} K_{t - s}^{2}(x - y) \mathbb{E}\left(|u(s, y) - v(s, y)|^{2}\right) dy ds.
\end{equation*}
So defining
\begin{equation*}
H(t) := \sup_{0 \le s \le t, x \in \mathbb{R}^{n}} \mathbb{E}[w^{2}(s, x)],
\end{equation*}
we get after using \eqref{Ktminuss2}, the following inequality:
\begin{equation*}
H(t) \le L^{2}\int_{0}^{t} H(s) \left(\int_{\mathbb{R}^{n}} K^{2}_{t - s}(x - y) dy\right) ds = c_{n, t} L^{2} \int_{0}^{t} H(s) ds.
\end{equation*}
Since $H(0) = 0$, using Gronwall's inequality then implies that $H(t)$ is identically zero for all $t$. 
In particular, $u$ is unique up to stochastic modification since the expectation
 $\mathbb{E}[w^2(t,x)] = 0$ for all $t \ge 0$ and $x\in \mathbb{R}^n$. 
This completes the uniqueness proof, and the proof of Theorem~\ref{existence}. 
\end{proof}

Now that we have shown an appropriate notion of existence and uniqueness of a mild solution for \eqref{sec3eqn} in dimensions one and two, we would like to understand more details of the solution behavior.  In particular, we would like to study the {\emph{H\"{o}lder  continuity of the sample paths}}, defined below in Section~\ref{Holder}. 
In order to do that, it is useful to  obtain uniform boundedness of 
 $L^{p}$ moments of the unique mild solution, for $p \ge 2$, uniformly in space and time on a bounded time interval. The proof of this result will rely on the BDG inequality, stated in Theorem \ref{BDG}. 

\begin{theorem}\label{boundedmoments}
Let $n = 1$ or $2$, and let $g$ and $h$ be continuous functions in $H^{2}(\mathbb{R}^{n})$. Let $u(t, x)$ be the unique function-valued mild solution to \eqref{sec3eqn} with initial data \eqref{initial}. Then, for each $T > 0$ and $p \ge 2$, 
\begin{equation}\label{highermoment}
\sup_{0 \le t \le T} \sup_{x \in \mathbb{R}^{n}} \mathbb{E}(|u(t, x)|^{p}) < \infty.
\end{equation}
\end{theorem}

\begin{proof}
Note that we have already established this result for $p = 2$ {{by using Lemma \ref{Picardwell} and the uniform convergence in $L^{2}(\Omega)$ given by \eqref{uniformconv}}}. 
To prove the higher moment bound \eqref{highermoment}, we reexamine our Picard iterates \eqref{Picarddef}:
\begin{equation*}
u_{k}(t, x) - u_{k - 1}(t, x) = \int_{0}^{t} \int_{\mathbb{R}^{n}} K_{t - s}(x - y) [f(u_{k - 1}(s, y)) - f(u_{k - 2}(s, y))] W(dy, ds).
\end{equation*}
Using the BDG inequality stated in Theorem \ref{BDG}, for $k \ge 2$ we get
\begin{align*}
\mathbb{E}(|u_{k}(t, x) - u_{k - 1}(t, x)|^{p}) &\le c_{p}\mathbb{E}\left[\left(\int_{0}^{t} \int_{\mathbb{R}^{n}} K_{t - s}^{2}(x - y) |f(u_{k - 1}(s, y)) - f(u_{k - 2}(s, y))|^{2} dy ds\right)^{p/2}\right].
\end{align*}
Since $f$ is Lipschitz, we can further estimate the right hand side to obtain:
\begin{align}\label{highermoment1}
\mathbb{E}(|u_{k}(t, x) - u_{k - 1}(t, x)|^{p}) &\le c_{p}L^{p} \cdot \mathbb{E}\left[\left(\int_{0}^{t} \int_{\mathbb{R}^{n}} K_{t - s}^{2}(x - y) |u_{k - 1}(s, y) - u_{k - 2}(s, y)|^{2} dy ds\right)^{p/2}\right].
\end{align}
We would like to move the expectation inside the integral sign on the right hand side, but we cannot do this yet because of the exponent of $p/2$. To handle this, we will separate $K_{t - s}^{2}(x - y)$ into \begin{equation}\label{splitKp}
K_{t - s}^{2}(x - y) = |K_{t - s}(x - y)|^{\frac{2p - 4}{p}} \cdot |K_{t - s}(x - y)|^{\frac{4}{p}}.
\end{equation}
We then apply H\"{o}lder's inequality with the conjugate exponents $p/2$ and $p/(p - 2)$ in \eqref{highermoment1} to obtain
\begin{align}\label{highermoment2}
& \mathbb{E}(|u_{k}(t, x) - u_{k - 1}(t, x)|^{p}) \nonumber \\
&\le c_{p}L^{p} \left(\int_{0}^{t} \int_{\mathbb{R}^{n}} K_{t - s}^{2}(x - y) dy ds\right)^{\frac{p}{2} - 1} \mathbb{E} \left(\int_{0}^{t} \int_{\mathbb{R}^{n}} K_{t - s}^{2}(x - y) |u_{k - 1}(s, y) - u_{k - 2}(s, y)|^{p} dy ds\right) \nonumber \\
&= c_{p}L^{p} \left(\int_{0}^{t} \int_{\mathbb{R}^{n}} K_{t - s}^{2}(x - y) dy ds\right)^{\frac{p}{2} - 1} \left(\int_{0}^{t} \int_{\mathbb{R}^{n}} K_{t - s}^{2}(x - y) \mathbb{E} \left(|u_{k - 1}(s, y) - u_{k - 2}(s, y)|^{p}\right)dy ds\right).
\end{align}
Therefore, defining
\begin{equation*}
J_{k}^{p}(t) := \sup_{0 \le s \le t, x \in \mathbb{R}^{n}} \mathbb{E}(|u_{k}(s, x) - u_{k - 1}(s, x)|^{p}),
\end{equation*}
we get that
\begin{equation}\label{highermoment3}
J_{k}^{p}(t) \le c_{p}L^{p} \left(\int_{0}^{t} \int_{\mathbb{R}^{n}} K_{t - s}^{2}(x - y) dy ds\right)^{\frac{p}{2} - 1} \int_{0}^{t} J_{k - 1}^{p}(s) \left(\int_{\mathbb{R}^{n}} K_{t - s}^{2}(x - y) dy\right) ds.
\end{equation}
Using \eqref{Ktminuss2} and \eqref{sqestimatekernel}, we obtain the following recursive inequality:
\begin{equation}\label{LpGronwall}
J_{k}^{p}(t) \le c_{p, n} t^{(3 - n)\left(\frac{p}{2} - 1\right)} \int_{0}^{t} J_{k - 1}^{p}(s) (t - s)^{2 - n} ds \le c_{p, n} t^{(3 - n)\frac{p}{2} - 1} \int_{0}^{t} J_{k - 1}^{p}(s) ds.
\end{equation}
Note that $J_{1}^{p}(t)$ is finite. Namely, by using the BDG inequality 
from Theorem \ref{BDG} one obtains
\begin{align*}
\mathbb{E}(|u_{1}(t, x) - u_{0}(t, x)|^{p}) &= \mathbb{E} \left(\left|\int_{0}^{t} \int_{\mathbb{R}^{n}} K_{t - s}(x - y) f(u_{0}(s, y)) W(dy, ds)\right|^{p}\right) \\
&\le c_{p} \mathbb{E}\left[\left(\int_{0}^{t} \int_{\mathbb{R}^{n}} K_{t - s}^{2}(x - y) |f(u_{0}(s, y))|^{2} dy ds\right)^{p/2}\right] \\
&\le c_{p} L^{p} \cdot \left[\left(\int_{0}^{t} \int_{\mathbb{R}^{n}} K_{t - s}^{2}(x - y) (1 + |u_{0}(s, y)|)^{2} dy ds\right)^{p/2}\right],
\end{align*}
where we eliminated the expectation because $u_{0}(t, x)$ is deterministic. 
We then use the splitting from \eqref{splitKp} above, 
and the same H\"{o}lder inequality argument as before, to obtain
\begin{align*}
\mathbb{E}(|u_{1}(t, x) &- u_{0}(t, x)|^{p}) 
\le c_{p}L^{p} \left(\int_{0}^{t} \int_{\mathbb{R}^{n}} K_{t - s}^{2}(x - y) dy ds\right)^{\frac{p}{2} - 1} \left(\int_{0}^{t} \int_{\mathbb{R}^{n}} K_{t - s}^{2}(x - y) (1 + |u_{0}(s, y)|)^{p}dy ds\right).
\end{align*}
The right hand-side is uniformly bounded for $t \in [0, T]$ and $x \in \mathbb{R}^{n}$ by Lemma \ref{boundedcontinuous} and \eqref{viscouscalc}. So $J_{1}^{p}(t)$ is finite for each $t \ge 0$. 

{{The recursive inequality \eqref{LpGronwall} implies that for any fixed $T > 0$, we have that for all $0 \le t \le T$ and $k \ge 2$,
\begin{equation}\label{pmomentrecursive}
J_{k}^{p}(t) \le C_{T, p, n} \int_{0}^{t} J_{k - 1}^{p}(s) ds.
\end{equation}
Since $J_{1}^{p}(t)$ is finite and bounded by a constant $A_{T}$ for all $0 \le t \le T$, we then apply \eqref{pmomentrecursive} inductively to conclude that 
\begin{equation}\label{pseries}
\sum_{k = 1}^{\infty} J_{k}(t) \le A_{T}^{1/p}\sum_{k = 0}^{\infty} \frac{(C_{T, p, n})^{k/p}t^{k/p}}{(k!)^{1/p}} < \infty, \qquad \text{ for all } t \in [0, T].
\end{equation}

Since $u_{0}$ is bounded on $[0, T] \times \mathbb{R}^{n}$ for all $T > 0$, deterministic, and continuous by Lemma \ref{boundedcontinuous}, we have that $u_{0} \in C([0, T] \times \mathbb{R}^{n}; L^{p}(\Omega))$. Since $\sum_{k = 1}^{\infty} J_{k}(T) < \infty$, we have that the sequence $u_{k}, k = 1,2,\dots$ is a Cauchy sequence in the complete space of bounded functions on $[0, T] \times \mathbb{R}^{n}$, taking values in $L^{p}(\Omega)$, equipped with the appropriate supremum norm:
\begin{equation*}
||f|| = \sup_{0 \le t \le T} \sup_{x \in \mathbb{R}^{n}} \left(\mathbb{E}(|f(t, x)|^{p})\right)^{1/p}.
\end{equation*}
Hence, the sequence $u_{k}$ converges in this space as $k \to \infty$, and the limit must be $u$. 
Thus, $ ||u|| = \sup_{0 \le t \le T} \sup_{x \in \mathbb{R}^{n}} \left(\mathbb{E}(|u(t, x)|^{p})\right)^{1/p}$ is bounded. }}
\end{proof}

\section{H\"{o}lder continuity of sample paths for the stochastic viscous wave equation}\label{Holder}

In this section we investigate  additional properties of our unique mild solution by focusing on what the sample paths
of the solution  look like. 
In particular, we study H\"{o}lder continuity of sample paths. 
%Now that we have shown an appropriate notion of existence and uniqueness of a mild solution for \eqref{sec3eqn} in dimensions one and two, we want to understand what the sample paths look like. In particular, we want to study the H\"{o}lder regularity continuity of the solution. 
%
Because we are working with a stochastic process, we have to precisely define what we mean by H\"{o}lder continuity of the sample paths. 

For this purpose, we recall the notion of a {\emph{modification}}. 
If $\{X_{i}\}_{i \in I}$ where $I$ is an index set is a stochastic process on a complete probability space $(\Omega, \mathcal{F}, \mathbb{P})$, then $\{\tilde{X_{i}}\}$ is a modification if
$
\mathbb{P}(\tilde{X_{i}} = X_{i}) = 1, \forall  i \in I.
$

We also recall that, given a stochastic process $\{X_{i}\}_{i \in I}$, the {\emph{finite dimensional distributions}} are the distributions of the random vectors
$
(X_{i_{1}}, X_{i_{2}}, ..., X_{i_{k}})
$
for all finite collections $(i_{1}, i_{2}, ..., i_{k})$ of indices in $I$. 

Note that $\{X_{i}\}_{i \in I}$ and a modification have the same finite dimensional distributions. Because the uniqueness result for the equation \eqref{sec3eqn} is up to modification, we will show that \eqref{sec3eqn} has a suitable modification such that the sample paths are H\"{o}lder continuous with a certain degree of H\"{o}lder regularity. 
%In particular, our main goal in this section will be to establish the following theorem, which holds in both one and two spatial dimensions. 

\begin{theorem}[{\bf{H\"{o}lder continuity of sample paths}}]\label{holder}
Let $g, h$ be continuous functions in $H^{2}(\mathbb{R}^{n})$, and let $f: \mathbb{R} \to \mathbb{R}$ be a Lipschitz continuous function. For each $\alpha \in [0, 1)$ in the case of $n = 1$ and for each $\alpha \in [0, 1/2)$ in the case of $n = 2$, the mild solution to \eqref{sec3eqn} has a modification that is (locally) $\alpha$-H\"{o}lder continuous on $\mathbb{R}^{+} \times \mathbb{R}^{n}$ in space and time. 
\end{theorem}

\begin{remark}
There are analogous results for the stochastic heat and wave equations \eqref{stheat} and \eqref{stwave}, but only in one dimension, since existence and uniqueness hold only in one dimension. For the stochastic {\emph{heat}} equation in $n = 1$ \eqref{stheat}, there is a modification that is $\alpha$-H\"{o}lder continuous in time and $\beta$-H\"{o}lder continuous in space for each $\alpha \in [0, 1/4)$ and each $\beta \in [0, 1/2)$. For the stochastic {\emph{wave}} equation in $n=1$ \eqref{stwave}, there is a modification that is $\alpha$-H\"{o}lder continuous in time and space for $\alpha \in [0, 1/2)$. The difference in the degree of H\"{o}lder regularity is due to the differences in spacetime scaling. 
%Since the viscous wave equation and wave equation share the same spacetime scaling of one time derivative versus one space derivative, it is to be expected that the H\"{o}lder regularity in space and time for the solution to the stochastic viscous wave equation have the same degree of H\"{o}lder regularity. 
We emphasize that our result for the stochastic viscous wave equation \eqref{sec3eqn} 
considers both $n=1$ and $n=2$. 
\end{remark}

The proof of Theorem \ref{holder} follows from a version of the Kolmogorov continuity criterion (see, e.g., Theorem 2.1 in Revuz and Yor \cite{RY}). 

\begin{theorem}[{\bf{Kolmogorov continuity criterion}}]\label{Kolm}
Let $\{X_{i}\}_{i \in [0, 1]^{N}}$ be a real-valued stochastic process. If there exist two positive constants $\gamma$ and $\epsilon$ such that 
\begin{equation*}
\mathbb{E}(|X_{i_{1}} - X_{i_{2}}|^{\gamma}) \le C|i_{1} - i_{2}|^{N + \epsilon},
\end{equation*}
then for each $\alpha$ such that
$
\displaystyle{0 \le \alpha < \frac{\epsilon}{\gamma},}
$
the stochastic process $\{X_{i}\}_{i \in [0, 1]^{N}}$ has a modification that is $\alpha$-H\"{o}lder continuous.
\end{theorem}

We can extend this to stochastic processes on unbounded Euclidean domains. In particular, we will reframe the Kolmogorov continuity criterion for our current case of a stochastic process indexed by $(t, x) \in \mathbb{R}^{+} \times \mathbb{R}^{n}$. This is similar to Theorem 2.5.1 in \cite{DK}.
%except that index set is $\mathbb{R}^{n}_{+}$ instead of $\mathbb{R}^{+} \times \mathbb{R}^{n}$. However,  as we will see in the explicit proof given below, the proof follows similarly by the same patching argument used in the proof of Theorem 2.5.1 in \cite{DK}.

\begin{corollary}\label{Kolm2}
Let $\{X(t, x)\}_{(t, x) \in \mathbb{R}^{+} \times \mathbb{R}^{n}}$ be a real-valued stochastic process. {{If there exist two positive constants $\gamma$ and $\epsilon$ such that for each compact set $K \subset \mathbb{R}^{+} \times \mathbb{R}^{n}$, 
\begin{equation*}
\mathbb{E}(|X(t, x) - X(s, y)|^{\gamma}) \le C_{K}|(t, x) - (s, y)|^{n + 1 + \epsilon} \qquad \text{ for all } (t, x), (s, y) \in K,
\end{equation*}
where $C_{K}$ can depend on $K$,}} then for each $\alpha$ such that
$
\displaystyle{0 \le \alpha < \frac{\epsilon}{\gamma}},
$
the stochastic process $\{X(t, x)\}_{(t, x) \in \mathbb{R}^{+} \times \mathbb{R}^{n}}$ has a modification 
{{$ \{ \tilde{X}(t,x) \}_{(t,x)\in \mathbb{R}^{+} \times \mathbb{R}^{n}}$  }}that is locally $\alpha$-H\"{o}lder continuous on $\mathbb{R}^{+} \times \mathbb{R}^{n}$.
\end{corollary}

\begin{proof}
Since the Kolmogorov continuity theorem appears more often in the form listed in Theorem \ref{Kolm}, we provide an explicit proof of Corollary \ref{Kolm2}, using an idea called a patching argument, as described on pg.~160 of \cite{DK}.
The corollary follows from the Kolmogorov continuity criterion in Theorem \ref{Kolm} by considering compact cubes, for example 
\begin{equation*}
A_{k} := [0, k] \times [-k/2, k/2]^{n} \subset \mathbb{R}^{+} \times \mathbb{R}^{n},
\end{equation*}
that increase to all of $\mathbb{R}^{+} \times \mathbb{R}^{n}$. 
{{We will construct the desired modification $ \{ \tilde{X}(t,x) \}_{(t,x)\in \mathbb{R}^{+} \times \mathbb{R}^{n}}$
as the limit $k\to\infty$ of $\alpha$-H\"{o}lder continuous modifications $\{X_{k}(t, x)\}$ defined for ${(t, x) \in A_{k}}$, constructed as follows.}}

Fix $\alpha$ such that $0 \le \alpha < \frac{\epsilon}{\gamma}$. By the usual Kolmogorov continuity criterion given in Theorem \ref{Kolm}, we can construct a modification $\{X_{k}(t, x)\}_{(t, x) \in A_{k}}$ of $\{X(t, x)\}_{(t, x) \in A_{k}}$ that is $\alpha$-H\"{o}lder continuous on $A_{k}$. 
The modifications $\{X_{k}(t, x)\}_{(t, x) \in A_{k}}$ in particular are continuous.  

We claim that any two of these modifications $X_{k}(t, x)$ and $X_{l}(t, x)$
must agree with probability one on their overlap because they are continuous modifications. Otherwise there exists a ball with rational radius and center with rational coordinates on which the two modifications have disjoint range with positive probability.
More precisely, consider $k \le l$ so that $A_{k} \subset A_{l}$ is the overlap. We claim that 
\begin{equation*}
\mathbb{P}(X_{k}(t, x) = X_{l}(t, x) \text{ for all } (t, x) \in A_{k}) = 1.
\end{equation*}
We argue by contradiction. Suppose that $\mathbb{P}(X_{k}(t_0, x_0) \ne X_{l}(t_0, x_0) \text{ for some } (t_0, x_0) \in A_{k}) > 0$. Since $X_{k}(t, x)$ and $X_{l}(t, x)$ are continuous on $A_{k}$, for every outcome $\omega \in \Omega$ for which $X_{k}(t_0, x_0) \ne X_{l}(t_0, x_0) \text{ for some } (t_0, x_0) \in A_{k}$, we can find an open ball $B_{r}(q)$ with rational radius $r$ 
centered at a point $q=(t,x) \in \mathbb{Q}^{+} \times \mathbb{Q}^{n} \cap A_{k} \subset \mathbb{R}^{+} \times \mathbb{R}^{n}$, such that $X_{k}$ and $X_{l}$ have \textit{``disjoint" range} on the ball $B_{r}(q) \cap A_{k}$ in the sense that there exist two closed intervals $K_{1} \subset \mathbb{R}$ and $K_{2} \subset \mathbb{R}$ such that
\begin{equation*}
X_{k}(B_{r}(q) \cap A_{k}) \subset K_{1}, \qquad X_{l}(B_{r}(q) \cap A_{k}) \subset K_{2}, \qquad K_{1} \cap K_{2} = \varnothing.
\end{equation*}
Hence,
\begin{align*}
\mathbb{P}(\bigcup_{r \in \mathbb{Q}, q \in \mathbb{Q}^{+} \times \mathbb{Q}^{n} \cap A_{k}} 
&\{X_{k} \text{ and } X_{l} \text{ have ``disjoint" range on } B_{r}(q) \cap A_{k} \}) \\
\ge \mathbb{P}&(X_{k}(t_0, x_0) \ne X_{l}(t_0, x_0) \text{ for some } (t_0, x_0) \in A_{k}) > 0.
\end{align*}
Therefore,  {{by the countability of the index set}}, there exist $r_{0} \in \mathbb{Q}^{+}$ and $q_{0} \in \mathbb{Q}^{+} \times \mathbb{Q}^{n} \cap A_{k}$ such that 
\begin{equation*}
\mathbb{P}(X_{k} \text{ and } X_{l} \text{ have ``disjoint" range on } B_{r_{0}}(q_{0}) \cap A_{k}) > 0.
\end{equation*}
But this implies that $\mathbb{P}(X_{k}(q_{0}) \ne X_{l}(q_{0})) > 0$, which contradicts that $\{X_{k}(t, x)\}_{(t, x) \in A_{k}}$ and $\{X_{l}(t, x)\}_{(t, x) \in A_{l}}$ are modifications of the same stochastic process on $A_{k}$ and $A_{l}$ respectively, since $q_{0} \in A_{k} \subset A_{l}$. Therefore, $X_{k} = X_{l}$ on $A_{k}$ almost surely.

This implies that, with probability one (up to a null set), any two modifications from this collection of modifications $\{X_{k}(t, x)\}_{(t, x) \in A_{k}}$ on increasing cubes must agree. 

To define the desired modification $\tilde{X}(t, x)$ we now focus on the null sets 
$E_{k, l}$ for $k < l$ on which the modifications $\{X_{k}(t, x)\}_{(t, x) \in A_{k}}$ and $\{X_{l}(t, x)\}_{(t, x) \in A_{l}}$ do not agree on $A_{k}$. Define 
\begin{equation*}
E = \bigcup_{k < l \text{ and } k, l \in \mathbb{Z}^{+}} E_{k, l}
\end{equation*}
and note that $\mathbb{P}(E) = 0$. Then, the desired modification $\{\tilde{X}(t, x)\}_{(t, x) \in \mathbb{R}^{+} \times \mathbb{R}^{n}}$ is
\begin{align*}
\tilde{X}(t, x, \omega) &= \lim_{k \to \infty} X_{k}(t, x, \omega) \ \qquad \text{ for } \omega \in E^{c}, \\
\tilde{X}(t, x, \omega) &= 0, \qquad \qquad \qquad \qquad \text{ for } \omega \in E.
\end{align*}
This limit exists since the sequence $X_{k}(t, x, \omega)$ for $\omega \in E^{c}$ is eventually constant, because the modifications $\{X_{k}(t, x)\}_{(t, x) \in A_{k}}$ all agree pairwise on their common domains for $\omega \in E^{c}$. It is easy to check that $\{\tilde{X}(t, x)\}_{(t, x) \in \mathbb{R}^{+} \times \mathbb{R}^{n}}$ is a modification that is $\alpha$-H\"{o}lder continuous, by using the properties that each of the $\{X_{k}(t, x)\}_{(t, x) \in A_{k}}$ are modification on $A_{k}$ that are $\alpha$-H\"{o}lder continuous.
This completes the proof of the corollary.
\end{proof}

%This argument above is known in the literature as a patching argument, as mentioned on pg.~160 of \cite{DK}, since we are patching together the H\"{o}lder continuous versions on increasing compact cubes to get the full process on $\mathbb{R}^{+} \times \mathbb{R}^{n}$.

\subsection{Proof of Theorem \ref{holder}}

We will prove the theorem for $n = 1$ and $n = 2$. Though the specific estimates will be slightly different for each dimension, the general computations are the same for both and hence we prove the results for $n = 1$ and $n = 2$ simultaneously.

By Corollary \ref{Kolm2}, it follows that to prove Theorem \ref{holder}, it suffices to show that for all $T > 0$, for all $\delta \in (0, 1)$, and for all $p \ge 2$, there exists a constant $C_{T, p, \delta}$ depending on $T$, $p$, and $\delta$ such that:
\begin{enumerate}
\item The following two estimates hold if $n = 1$:
\begin{equation}\label{timeincrest}
\mathbb{E}(|u(t, x) - u(t', x)|^{p}) \le C_{T, p, \delta} |t - t'|^{\frac{(1 + \delta) p}{2}}, \text{ for all } t, t' \in [0, T], \text{ and } x \in \mathbb{R},
\end{equation}
\begin{equation}\label{spaceincrest}
\mathbb{E}(|u(t, x) - u(t, x')|^{p}) \le C_{T, p, \delta} |x - x'|^{\frac{(1 + \delta) p}{2}}, \text{ for all } t \in [0, T], \text{ and } x, x' \in \mathbb{R}.
\end{equation}
\item The following two estimates hold if $n = 2$:
\begin{equation}\label{timeincrest2}
\mathbb{E}(|u(t, x) - u(t', x)|^{p}) \le C_{T, p, \delta} |t - t'|^{\frac{\delta p}{2}}, \text{ for all } t, t' \in [0, T], \text{ and } x \in \mathbb{R}^{2},
\end{equation}
\begin{equation}\label{spaceincrest2}
\mathbb{E}(|u(t, x) - u(t, x')|^{p}) \le C_{T, p, \delta} |x - x'|^{\frac{\delta p}{2}}, \text{ for all } t \in [0, T], \text{ and } x, x' \in \mathbb{R}^{2}.
\end{equation}
\end{enumerate}

%We establish the estimates \eqref{timeincrest} and \eqref{spaceincrest} for the time and spatial increments in dimension $n = 1$. These estimates are similar to the estimates for the stochastic heat equation in \cite{KMini}. 

 \textbf{{Estimate for the time increments.}} To prove estimates \eqref{timeincrest} and \eqref{timeincrest2}, we consider for $p \ge 2$,
\begin{equation*}
\mathbb{E}(|u(t, x) - u(t', x)|^{p}), \qquad \text{ for } t, t' \in [0, T],
\end{equation*}
for $n = 1, 2$, where $T > 0$ is fixed but arbitrary. Recall from the definition of a mild solution \eqref{mildsoln} that 
\begin{equation*}
u(t, x) = u_{0}(t, x) + \int_{0}^{t} \int_{\mathbb{R}^{n}} K_{t - s}(x - y) f(u(s, y)) W(dy, ds),
\end{equation*}
where $u_{0}(t, x)$ is the deterministic function solving the homogeneous deterministic viscous wave equation with initial data $g, h$. 
We assume that $0 \le t' < t \le T$ and express the time increment as
\begin{multline*}
u(t, x) - u(t', x) = u_{0}(t, x) - u_{0}(t', x) \\
+ \int_{0}^{t'} \int_{\mathbb{R}^{n}} [K_{t - s}(x - y) - K_{t' - s}(x - y)]f(u(s, y)) W(dy, ds) + \int_{t'}^{t} \int_{\mathbb{R}^{n}} K_{t - s}(x - y) f(u(s, y)) W(dy, ds).
\end{multline*}
Next, we use the BDG inequality from Theorem \ref{BDG}, along with $(a + b + c)^{p} \le c_{p}(|a|^{p} + |b|^{p} + |c|^{p})$ for $a, b, c \ge 0$, to obtain
\begin{align}\label{timesplit}
\mathbb{E}&\left(|u(t, x) - u(t', x)|^{p}\right) \nonumber \\
\le c_{p} &\Bigg[|u_{0}(t, x) - u_{0}(t', x)|^{p} + \mathbb{E}\left(\int_{0}^{t'}\int_{\mathbb{R}^{n}} |K_{t - s}(x - y) - K_{t' - s}(x - y)|^{2} |f(u(s, y))|^{2} dy ds\right)^{p/2} \nonumber \\
 + \mathbb{E}&\left(\int_{t'}^{t} \int_{\mathbb{R}^{n}} |K_{t - s}(x - y)|^{2} |f(u(s, y))|^{2} dy ds\right)^{p/2}\Bigg] := c_{p}(I_{1} + I_{2} + I_{3}).
\end{align}
By Lemma \ref{boundedcontinuous}, there exists $C_{T}$ such that 
\begin{equation}\label{I11d}
I_{1} = |u_{0}(t, x) - u_{0}(t', x)|^{p} \le C_{T} |t - t'|^{p}, \qquad \text{ in the case of } n = 1,
\end{equation}
and for every $\delta \in (0, 1)$, there exists a constant $C_{T, \delta}$ depending only on $T > 0$ and $\delta$ such that
\begin{equation}\label{I12d}
I_{1} = |u_{0}(t, x) - u_{0}(t', x)|^{p} \le C_{T, \delta} |t - t'|^{\frac{\delta p}{2}}, \qquad \text{ in the case of } n = 2.
\end{equation}
For the integral in $I_{2}$ defined in \eqref{timesplit}, we use the same idea as in \eqref{splitKp} and separate the term involving the kernel into two factors by using H\"{o}lder's inequality with $p/2$ and $p/(p - 2)$ to obtain
\begin{align*}
I_{2} \le &\left(\int_{0}^{t'}\int_{\mathbb{R}^{n}} |K_{t - s}(x - y) - K_{t' - s}(x - y)|^{2} dy ds\right)^{\frac{p}{2} - 1} \\
\cdot \mathbb{E}&\left(\int_{0}^{t'} \int_{\mathbb{R}^{n}} |K_{t - s}(x - y) - K_{t' - s}(x - y)|^{2} |f(u(s, y))|^{p}dy ds \right).
\end{align*}
In the second factor, we can move the expectation into the integrand, and use the Lipschitz property of $f$ to obtain 
that for all $s \in [0, t'], y \in \mathbb{R}^{n}$, the following estimate holds:
\begin{equation}\label{fp}
\mathbb{E}\left(|f(u(s, y))|^{p}\right) \le 2^{p}L^{p} \mathbb{E}(1 + |u(s, y)|^{p}) = C_{T, p} < \infty,
\end{equation}
where the last inequality follows from the boundedness of $p$th moments in \eqref{highermoment}. Therefore,
\begin{equation*}
I_{2} \le C_{T, p}\left(\int_{0}^{t'}\int_{\mathbb{R}^{n}} |K_{t - s}(x - y) - K_{t' - s}(x - y)|^{2} dy ds\right)^{\frac{p}{2}}.
\end{equation*} 
Using Plancherel's theorem and absorbing constants into $C_{T, p}$, 
\begin{equation}\label{I2}
I_{2} \le C_{T, p} \left(\int_{0}^{t'}\int_{\mathbb{R}^{n}} \left|e^{-\frac{|\xi|(t - s)}{2}} \frac{\sin\left(\frac{\sqrt{3}}{2}|\xi|(t - s)\right)}{|\xi|} - e^{-\frac{|\xi|(t' - s)}{2}} \frac{\sin\left(\frac{\sqrt{3}}{2}|\xi|(t' - s)\right)}{|\xi|}\right|^{2}d\xi ds\right)^{\frac{p}{2}}. 
\end{equation}
Continuing to absorb constants into $C_{T, p}$ as necessary, we separate this into
\begin{equation}\label{I2break}
I_{2} \le C_{T, p}(J_{1} + J_{2})^{\frac{p}{2}},
\end{equation}
where
\begin{equation}\label{J1}
J_{1} = \int_{0}^{t'} \int_{\mathbb{R}^{n}} \frac{\sin^{2}\left(\frac{\sqrt{3}}{2} |\xi| (t - s)\right)}{|\xi|^{2}}\left|e^{-\frac{|\xi|(t - s)}{2}} - e^{-\frac{|\xi|(t' - s)}{2}}\right|^{2} d\xi ds,
\end{equation}
\begin{equation}\label{J2}
J_{2} = \int_{0}^{t'} \int_{\mathbb{R}^{n}} e^{-|\xi|(t' - s)}\left(\frac{\sin\left(\frac{\sqrt{3}}{2}|\xi|(t - s)\right)}{|\xi|} - \frac{\sin\left(\frac{\sqrt{3}}{2}|\xi|(t' - s)\right)}{|\xi|}\right)^{2} d\xi ds.
\end{equation}
To estimate $J_{1}$, we first simplify $J_{1}$ to get
\begin{equation*}
J_{1} = \int_{0}^{t'} \int_{\mathbb{R}^{n}} e^{-|\xi|(t' - s)}\frac{\sin^{2}\left(\frac{\sqrt{3}}{2} |\xi| (t - s)\right)}{|\xi|^{2}}\left(1 - e^{-\frac{|\xi|(t - t')}{2}}\right)^{2} d\xi ds.
\end{equation*}
Next, we use the fact that there exists a uniform constant $C$ such that
\begin{equation}\label{expbound}
0 \le 1 - e^{-r} \le \min(1, r), \qquad \text{ for all } r \ge 0.
\end{equation}
In addition, there exists a uniform constant depending only on $\delta \in (0, 1)$ such that 
\begin{equation}\label{expdecay}
r^{\delta}e^{-r} \le C_{\delta}, \qquad \text{ for all } r > 0, \delta \in (0, 1).
\end{equation}
Thus, for each $\delta \in (0, 1)$ we have:
\begin{align}\label{J1timeest}
J_{1} &\le C_{\delta} \int_{0}^{t'} \frac{1}{(t' - s)^{\delta}} \int_{\mathbb{R}^{n}} \frac{1}{|\xi|^{\delta}} \cdot \frac{1}{|\xi|^{2}} \min\left(1, \frac{1}{4} |\xi|^{2} (t - t')^{2}\right) d\xi ds \nonumber \\
&\le C_{\delta} \int_{0}^{t'} \frac{1}{(t' - s)^{\delta}} \int_{\mathbb{R}^{n}} \frac{1}{|\xi|^{2 + \delta}} \min(1, |\xi|^{2}(t - t')^{2}) d\xi ds \nonumber \\
& = C_{\delta} \left(\int_{0}^{t'} \frac{1}{(t' - s)^{\delta}} \left(\int_{|\xi| \le (t - t')^{-1}} \frac{1}{|\xi|^{\delta}} (t - t')^{2} d\xi + \int_{|\xi| \ge (t - t')^{-1}} \frac{1}{|\xi|^{2 + \delta}} d\xi\right) ds \right),
\end{align}
where $C_{T, \delta}$ is a constant depending on $\delta \in (0, 1)$, and on the fixed but arbitrary $T > 0$. Note that we have restricted $\delta$ to the range of $\delta \in (0, 1)$ so that the appropriate integrals converge in both spatial dimension $n = 1$ and $n = 2$. Computing the integrals in \eqref{J1timeest} gives
\begin{equation}\label{J11d}
J_{1} \le C_{\delta} \left(\int_{0}^{t'} \frac{1}{(t' - s)^{\delta}}(t - t')^{1 + \delta} ds \right) = C_{T, \delta} |t - t'|^{1 + \delta}, \qquad \text{ for } n = 1,
\end{equation}
\begin{equation}\label{J12d}
J_{1} \le C_{\delta} \left(\int_{0}^{t'} \frac{1}{(t' - s)^{\delta}}(t - t')^{\delta} ds \right) = C_{T, \delta} |t - t'|^{\delta}, \qquad \text{ for } n = 2.
\end{equation}

We now consider $J_{2}$ as defined in \eqref{J2}. By the mean value theorem, 
\begin{equation}\label{sindiff}
\left|\sin\left(\frac{\sqrt{3}}{2} |\xi| (t - s)\right) - \sin\left(\frac{\sqrt{3}}{2}|\xi|(t' - s)\right)\right| \le \min\left(2, \frac{\sqrt{3}}{2}|\xi|(t - t')\right) \le \min(2, |\xi|(t - t')).
\end{equation}
Combining the estimates \eqref{expdecay} and  \eqref{sindiff} gives for arbitrary $\delta \in (0, 1)$,
\begin{equation*}
J_{2} \le C_{\delta} \int_{0}^{t'} \frac{1}{(t' - s)^{\delta}} \int_{\mathbb{R}^{n}} \frac{1}{|\xi|^{\delta}} \cdot \frac{1}{|\xi|^{2}} \min\left(1, |\xi|^{2} (t - t')^{2}\right) d\xi ds.
\end{equation*}
The rest proceeds exactly as for the computation for $J_{1}$, see \eqref{J1timeest}, and thus we obtain
\begin{equation*}
J_{2} \le C_{T, \delta} |t - t'|^{1 + \delta} \qquad \text{ for } n = 1, \qquad \qquad J_{2} \le C_{T, \delta}|t - t'|^{\delta} \qquad \text{ for } n = 2,
\end{equation*}
for $C_{\delta, T}$ depending only on $\delta \in (0, 1)$ and $T$. Substituting into \eqref{I2break}, we have for arbitrary $\delta \in (0, 1)$,
\begin{equation}\label{I21d2d}
I_{2} \le C_{T, p, \delta} |t - t'|^{\frac{(1 + \delta) p}{2}} \qquad \text{ for $n = 1$}, \qquad \qquad I_{2} \le C_{T, p, \delta} |t - t'|^{\frac{\delta p}{2}} \qquad \text{ for $n = 2$}.
\end{equation}

For $I_{3}$ as defined in \eqref{timesplit}, we use the idea from \eqref{splitKp}, combined with
the Lipschitz property of $f$, the boundedness of $p$th moments of  $u(t, x)$ on finite time intervals,
and a calculation similar to \eqref{viscouscalc} to obtain for $n = 1, 2$,
\begin{align}\label{I31d2d}
I_{3} &\le c_{p}\left(\int_{t'}^{t} \int_{\mathbb{R}^{n}} |K_{t - s}(x - y)|^{2} dy ds\right)^{\frac{p}{2} - 1}\left(\int_{t'}^{t} \int_{\mathbb{R}^{n}} |K_{t - s}(x - y)|^{2} \mathbb{E}(|f(u(s, y))|^{p}) dy ds\right) \nonumber \\
&\le c_{T, p} \left(\int_{t'}^{t} \int_{\mathbb{R}^{n}} |K_{t - s}(x - y)|^{2} dy ds\right)^{\frac{p}{2}} = c_{T, p} \left(\int_{t'}^{t} (t - s)^{2 - n} ds\right)^{p/2} ||K||_{L^{2}(\mathbb{R}^{n})}^{p} = c_{T, p} |t - t'|^{\frac{p}{2}(3 - n)}.
\end{align}
The estimates \eqref{I11d}, \eqref{I12d}, \eqref{I21d2d}, \eqref{I31d2d} for $I_{1}$, $I_{2}$, and $I_{3}$ and \eqref{timesplit} establish the desired time increment estimates \eqref{timeincrest} for $n = 1$ and \eqref{timeincrest2} for $n = 2$.

 \textbf{{Estimate for the spatial increments.}}
We examine the spatial regularity of the stochastic solution $u(t, x)$ by establishing \eqref{spaceincrest} and \eqref{spaceincrest2}. 
 For $0 \le t \le T$ and $x, x' \in \mathbb{R}^{n}$, we have that
\begin{equation*}
u(t, x) - u(t, x') = (u_{0}(t, x) - u_{0}(t, x')) + \int_{0}^{t} \int_{\mathbb{R}^{n}} (K_{t - s}(x - y) - K_{t - s}(x' - y)) f(u(s, y)) W(dy, ds),
\end{equation*}
and hence, for $p \ge 2$, 
\begin{align}\label{spaceincsplit}
\mathbb{E}&(|u(t, x) - u(t, x')|^{p}) \nonumber \\
&= c_{p}\left(|u_{0}(t, x) - u_{0}(t, x')|^{p} + \mathbb{E}\left(\left|\int_{0}^{t} \int_{\mathbb{R}^{n}} (K_{t - s}(x - y) - K_{t - s}(x' - y)) f(u(s, y)) W(dy, ds)\right|^{p}\right)\right) \nonumber \\
&:= c_{p}(I_4 + I_5).
\end{align}

We bound $I_4$ by using Lemma \ref{boundedcontinuous} to obtain
\begin{equation}\label{I41d}
I_4 = |u_{0}(t, x) - u_{0}(t, x')|^{p} \le C_{T} |x - x'|^{p} \qquad \text{ for } n = 1,
\end{equation}
and for arbitrary $\delta \in (0, 1)$, 
\begin{equation}\label{I42d}
I_4 = |u_{0}(t, x) - u_{0}(t, x')|^{p} \le C_{T, \delta} |x- x'|^{\frac{\delta p}{2}}, \qquad \text{ for } n = 2.
\end{equation}
To estimate $I_5$, we use the BDG inequality stated in Theorem \ref{BDG} to obtain
\begin{equation*}
I_5 \le \mathbb{E}\left(\int_{0}^{t} \int_{\mathbb{R}^{n}} |K_{t - s}(x - y) - K_{t - s}(x' - y)|^{2} |f(u(s, y))|^{2} dy ds\right)^{\frac{p}{2}}.
\end{equation*}
By using the same computation as in \eqref{splitKp},
\begin{align*}
I_5 \le &\left(\int_{0}^{t} \int_{\mathbb{R}^{n}} |K_{t - s}(x - y) - K_{t - s}(x' - y)|^{2} dy ds\right)^{\frac{p}{2} - 1} \\
&\cdot \left(\int_{0}^{t} \int_{\mathbb{R}^{n}} |K_{t - s}(x - y) - K_{t - s}(x' - y)|^{2} \mathbb{E}\left(|f(u(s, y))|^{p}\right) dy ds\right).
\end{align*} 
Using the higher moment bound on $u$ as in \eqref{fp}, we have
\begin{align*}
I_5 &\le C_{T, p} \left(\int_{0}^{t} \int_{\mathbb{R}^{n}} |K_{t - s}(x - y) - K_{t - s}(x' - y)|^{2}dy ds\right)^{\frac{p}{2}} \\
&= C_{T, p} \left(\int_{0}^{t} \int_{\mathbb{R}^{n}} |K_{t - s}(y) - K_{t - s}(y + x' - x)|^{2} dy ds\right)^{\frac{p}{2}}.
\end{align*}
Absorbing constants into $C_{T, p}$ as necessary and using Plancherel's formula gives that
\begin{align*}
I_5 &\le C_{T, p} \left(\int_{0}^{t} \int_{\mathbb{R}^{n}} \left|e^{-\frac{|\xi|}{2}(t - s)}\frac{\sin\left(\frac{\sqrt{3}}{2}|\xi|(t - s)\right)}{|\xi|}(1 - e^{i\xi \cdot (x' - x)})\right|^{2} d\xi ds\right)^{\frac{p}{2}} \\
&= C_{T, p}\left(\int_{0}^{t} \int_{\mathbb{R}^{n}} e^{-|\xi|(t - s)}\frac{1}{|\xi|^{2}}\left[1 - \cos(\xi \cdot (x' - x))\right] d\xi ds\right)^{\frac{p}{2}}.
\end{align*}
We use the inequality $1 - \cos(\xi \cdot (x' - x)) \le \min(2, |\xi|^{2} |x' - x|^{2})$ and \eqref{expdecay} to obtain for $\delta \in (0, 1)$, 
\begin{align*}
I_5 &\le C_{T, p}\left(\int_{0}^{t} \int_{\mathbb{R}^{n}} e^{-|\xi|(t - s)}\frac{1}{|\xi|^{2}}\left[1 - \cos(\xi \cdot (x' - x))\right] d\xi ds\right)^{\frac{p}{2}} \\
&\le C_{T, p, \delta} \left(\int_{0}^{t} \frac{1}{(t - s)^{\delta}} \int_{\mathbb{R}^{n}} \frac{1}{|\xi|^{2 + \delta}}\min(2, |\xi|^{2}|x' - x|^{2}) d\xi ds\right)^{\frac{p}{2}} \\
&\le C_{T, p, \delta} \left(\int_{0}^{t} \frac{1}{(t - s)^{\delta}} \left(\int_{|\xi| \le |x - x'|^{-1}} \frac{1}{|\xi|^{\delta}} |x' - x|^{2} d\xi + \int_{|\xi| \ge |x - x'|^{-1}} \frac{2}{|\xi|^{2 + \delta}} d\xi\right) ds\right)^{\frac{p}{2}}.
\end{align*}
These integrals converge for $n = 1, 2$ since $\delta \in (0, 1)$. We can compute these integrals to obtain
\begin{equation}\label{I51d}
I_{5} \le C_{T, p, \delta} \left(\int_{0}^{t} \frac{1}{(t - s)^{\delta}} |x - x'|^{1 + \delta} ds\right)^{\frac{p}{2}} = C_{T, p, \delta} |x - x'|^{\frac{(1 + \delta) p}{2}}, \qquad \text{ for } n = 1,
\end{equation}
\begin{equation}\label{I52d}
I_{5} \le C_{T, p, \delta} \left(\int_{0}^{t} \frac{1}{(t - s)^{\delta}} |x - x'|^{\delta} ds\right)^{\frac{p}{2}} = C_{T, p, \delta} |x - x'|^{\frac{\delta p}{2}}, \qquad \text{ for } n = 2.
\end{equation}
The estimates \eqref{spaceincsplit}, \eqref{I41d}, \eqref{I42d}, \eqref{I51d}, and \eqref{I52d} establish the required spatial increment estimates \eqref{spaceincrest} for $n = 1$ and \eqref{spaceincrest2} for $n = 2$. This completes the proof of Theorem \ref{holder}.

\section{Conclusion}
We have shown that a Cauchy problem for the stochastically perturbed viscous wave equation \eqref{model} has a unique (up to a modification) mild solution in both $n=1$ and $n=2$, and that the stochastic mild solution has a modification which is $\alpha$-H\"{o}lder continuous, where  
 $\alpha$-H\"{o}lder continuity is up to $\alpha = 1$ in $n=1$, and up to $\alpha = 1/2$ in $n=2$. 
 This result is significant, especially for $n=2$, since it indicates that
stochastically perturbed fluid-structure interaction problems involving viscous, incompressible fluids at low-to-medium Reynolds numbers, will have H\"{o}lder continuous solutions for almost all realizations of sample paths,
even in the case when the stochasticity in the forcing (or data)  is represented by the very rough 
spacetime white noise. We remark that this would not be the case if the structure itself, modeled by the stochastically 
perturbed wave equation in $n=2$, were considered without the fluid,
as it is well known that for the spacetime white noise perturbed wave and heat equations,
stochastic mild solutions do not exist in dimensions $n=2$ and higher. 
It is the coupled problem that provides the 
right scaling and sufficient dissipation that damps high-order frequencies exponentially fast in time,
thereby allowing a unique stochastic, H\"{o}lder continuous mild solution to exist for almost all realizations.

\section{Appendix}
\begin{proof}[Proof of Lemma 3.1]
To prove this lemma we use induction, presented in several steps below.

{\bf Step 1.}
For the inductive step, suppose that the following properties of $u_{k - 1}$ are satisfied:
\\
\phantom {1} \hskip 0.3in 1. $u_{k - 1}$ is adapted to the filtration $\{\mathcal{F}_{t}\}_{t \ge 0}$, \\
\phantom {1} \hskip 0.3in 2. $u_{k - 1}$ is jointly measurable, \\
\phantom {1} \hskip 0.3in 3. $u_{k - 1}$ satisfies {{for every $T > 0$,}}
\begin{equation}\label{inductive1}
\sup_{t \in [0, T]} \sup_{x \in \mathbb{R}^{n}} \mathbb{E}\left(|u_{k - 1}(s, y))|^{2}\right) := C_{k - 1, T} < \infty.
\end{equation}
\\
\phantom {1} \hskip 0.3in 4. $u_{k - 1}$ is continuous as a map from $(t, x) \in [0, T] \times \mathbb{R}^{n}$ to $L^{2}(\Omega)$, for arbitrary $T > 0$.
\\
Certainly, the base case holds. This is because $u_{0}$ is deterministic, hence it  immediately satisfies the adaptedness and joint measurability conditions. For \eqref{inductive1}, we can get rid of the expectation, since $u_{0}$ is deterministic. Then, \eqref{inductive1} follows from the fact that $u_{0}$ is bounded by Lemma \ref{boundedcontinuous}. The $L^{2}(\Omega)$ continuity, since $u_{0}$ is deterministic, follows from the continuity statement in Lemma \ref{boundedcontinuous}. 

{\bf Step 2.} We want to show that with this inductive assumption, the stochastic integral in \eqref{Picarddef} is well-defined. 
So given arbitrary $t > 0$, we must show that the integrand $K_{t - s}(x - y) f(u_{k - 1}(y, s))$ for $s \in [0, t), y \in \mathbb{R}^{n}$, satisfies the conditions in Proposition \ref{PW}. {{Recall from \eqref{Picarddef} that $(t, x)$ is a fixed but arbitrary point in $\mathbb{R}^{+} \times \mathbb{R}^{n}$ and $(s, y)$ here indicates the variables that are integrated against the spacetime white noise.}} Since the kernel $K_{t-s}(x-y)$ is singular at $s = t$ and $x = y$,
we first show that the conditions in Proposition \ref{PW} hold for $s\in[0,t)$ and then also in the limit $s \to t$.
We start by showing that the conditions in Proposition \ref{PW} hold for $s\in[0,t)$:
\begin{itemize}
\item Since $u_{k - 1}(s, y)$ for $s \in [0, t), y \in \mathbb{R}^{n}$ is adapted, so is $K_{t - s}(x, y) f(u_{k - 1}(s, y))$ since $f$ is continuous.
\item For each $s \in [0, t)$, $y \in \mathbb{R}^{n}$, we have
\begin{equation*}
\mathbb{E}\left(|K_{t - s}(x - y) f(u_{k - 1}(s, y))|^{2}\right) \le 2L^{2}C_{t - s} \left(1 + \mathbb{E}\left(|u_{k - 1}(s, y)|^{2}\right)\right) < \infty,
\end{equation*}
by the inductive assumption \eqref{inductive1}. Here, $L$ is the Lipschitz constant for $f$, and we used the fact that $K_{t - s}(\cdot)$ is bounded by a finite constant $C_{t - s}$ depending on the parameter $t - s$. 
\item To show that $K_{t - s}(x - y)f(u_{k - 1}(y, s))$ is $L^{2}(\Omega)$-continuous for $s \in [0, t)$ and $y \in \mathbb{R}^{n}$,
we fix $s_{0} \in [0, t)$ and $y_{0} \in \mathbb{R}^{n}$ and compute
\begin{align*}
\mathbb{E}(|K_{t - s_{1}}(x - y_{1})&f(u_{k - 1}(s_{1}, y_{1})) - K_{t - s_{0}}(x - y_{0})f(u_{k - 1}(s_{0}, y_{0}))|^{2})\\
&\le 2\mathbb{E}\left(|K_{t - s_{1}}(x - y_{1})f(u_{k - 1}(s_{1}, y_{1})) - K_{t - s_{1}}(x - y_{1})f(u_{k - 1}(s_{0}, y_{0}))|^{2}\right) \\
&+ 2\mathbb{E}\left(|K_{t - s_{1}}(x - y_{1})f(u_{k - 1}(s_{0}, y_{0})) - K_{t - s_{0}}(x - y_{0})f(u_{k - 1}(s_{0}, y_{0}))|^{2}\right).
\end{align*}
Using the Lipschitz condition for $f$ in the first term on the right hand side, and the fact that $f$ is linearly bounded by $|f(x)| \le L(1 + |x|)$ in the second term on the right hand side, 
\begin{align*}
\mathbb{E}(|K_{t - s_{1}}(x - y_{1})&f(u_{k - 1}(s_{1}, y_{1})) - K_{t - s_{0}}(x - y_{0})f(u_{k - 1}(s_{0}, y_{0}))|^{2}) \\
&\le 2L^{2}|K_{t - s_{1}}(x - y_{1})|^{2} \mathbb{E}\left(|u_{k - 1}(s_{1}, y_{1}) - u_{k - 1}(s_{0}, y_{0})|^{2}\right) \\
&+ 4|K_{t - s_{1}}(x - y_{1}) - K_{t - s_{0}}(x - y_{0})|^{2} \mathbb{E}\left(L^{2}(1 + |u_{k - 1}(s_{0}, y_{0})|^{2})\right).
\end{align*}
By using the inductive assumption \eqref{inductive1}, in the second term above we can bound the expectation of $|u_{k-1}(s_0,y_0)|^2$ to obtain the following estimate:
\begin{align}\label{L2contest}
\mathbb{E}(|K_{t - s_{1}}(x - y_{1})&f(u_{k - 1}(y_{1}, s_{1})) - K_{t - s_{0}}(x - y_{0})f(u_{k - 1}(s_{0}, y_{0}))|^{2}) \nonumber \\
&\le 2L^{2}|K_{t - s_{1}}(x - y_{1})|^{2} \mathbb{E}\left(|u_{k - 1}(s_{1}, y_{1}) - u_{k - 1}(s_{0}, y_{0})|^{2}\right) \nonumber \\
&+ \tilde{C}_{k - 1, t} |K_{t - s_{1}}(x - y_{1}) - K_{t - s_{0}}(x - y_{0})|^{2},
\end{align}
for some constant $\tilde{C}_{k - 1, t}$ depending only on $k - 1$ and $t$. 
To show continuity, we want to make the right hand-side of \eqref{L2contest} arbitrarily small whenever $|(s_{1}, y_{1}) - (s_{0}, y_{0})|$ is small.
Indeed, in the first term on the right hand-side, $K_{t - s}(x - y)$ is locally bounded for $s \in [0, t)$ and $y \in \mathbb{R}^{n}$, and $u_{k - 1}$ is $L^{2}(\Omega)$ continuous by the inductive assumption, 
so the first term on the right hand-side of \eqref{L2contest} can be made arbitrarily small for 
$|(s_{1}, y_{1}) - (s_{0}, y_{0})| < \delta$, for $\delta$ sufficiently small. 
This is also true for the second term on the right hand side because $K_{t - s}(x - y)$ is continuous for $s \in [0, t)$ and $y \in \mathbb{R}^{n}$. This establishes the claim.
\item To check the square integrability condition, we compute
\begin{align}\label{squareintegrabilityPicard}
\mathbb{E} \int_{0}^{t} \int_{\mathbb{R}^{n}} & |K_{t - s}(x - y)|^{2} |f(u_{k - 1}(s, y))|^{2} dy ds \nonumber \\
&\le 2L^{2} \int_{0}^{t} \int_{\mathbb{R}^{n}} |K_{t - s}(x - y)|^{2} \left(1 + \mathbb{E}\left(|u_{k - 1}(s, y)|^{2}\right)\right) dy ds \nonumber \\
&= 2L^{2} \left(1 + \sup_{s' \in [0, t], y' \in \mathbb{R}^{n}} \mathbb{E}\left(|u_{k - 1}(s', y')|\right)^{2}\right) \left(\int_{0}^{t} (t - s)^{2 - n} ds\right)||K||^{2}_{L^{2}(\mathbb{R}^{n})} < \infty,
\end{align}
where we used the identity in \eqref{viscouscalc}, the fact that $n = 1$ or $2$, Lemma \ref{kernelLq}, and the inductive assumption \eqref{inductive1}. 
\end{itemize}

To show that the stochastic integral in \eqref{Picarddef} is still well-defined for $s \in [0,t]$, 
we claim that this stochastic integral can be defined as the $L^{2}(\Omega)$ limit of stochastic integrals whose integrands are explicitly in the admissible class $\mathcal{P}_{W}$ of integrands. To see this, choose an increasing sequence $t_{i}, i = 1,2,\dots$ of positive real numbers such that $t_{i} \to t$ as $i\to\infty$. Note that 
\begin{equation}\label{Cauchydef}
\int_{0}^{t_{i}} \int_{\mathbb{R}^{n}} K_{t - s}(x - y) f(u_{k - 1}(s, y)) W(dy, ds)
\end{equation}
is a well-defined stochastic integral by the properties verified above, by Proposition \ref{PW}. By \eqref{squareintegrabilityPicard},
\begin{equation*}
\mathbb{E} \int_{t_{i}}^{t} \int_{\mathbb{R}^{n}} |K_{t - s}(x - y)|^{2} |f(u_{k - 1}(s, y))|^{2} dy ds \to 0,\ {\rm as} \ t_i \to t.
\end{equation*}
Hence, since $\mathcal{P}_{W}$ is a closed Banach space, the integrand
in \eqref{Picarddef} is in $\mathcal{P}_{W}$, and can be defined rigorously as the limit of the Cauchy sequence \eqref{Cauchydef} in $L^{2}(\Omega)$, by the Itô isometry and the finiteness of the quantity in \eqref{squareintegrabilityPicard}.

{\bf Step 3.} It remains to show that $u_{k}(t, x)$ satisfies the conditions in the inductive assumption in Step 1. Indeed, $u_{k}(t, x)$ is adapted by the construction of the stochastic integral. Joint measurability (up to modification) will follow from the later verification of continuity in $L^{2}(\Omega)$, as noted in Remark \ref{checkless} below. Thus, properties 1 and 2 in Step 1 are verified. 

To verify property 3 in Step 1, we check that for each $T > 0$, 
\begin{equation}\label{nextinductive}
\sup_{t \in [0, T]} \sup_{x \in \mathbb{R}^{n}} \mathbb{E}\left(|u_{k}(t, x)|^{2}\right) = C_{k, T} < \infty.
\end{equation}
This follows by direct calculation. Fix arbitrary $T > 0$ and consider $t \in [0, T]$, $x \in \mathbb{R}^{n}$. By \eqref{Ito} and \eqref{Picarddef}, we get
\begin{align}\label{inductivestep}
\mathbb{E}\left(|u_{k}(t, x)|^{2}\right) &= 2\mathbb{E}(|u_{0}(t, x)|^{2}) + 2\mathbb{E} \int_{0}^{t} \int_{\mathbb{R}^{n}} |K_{t - s}(x - y)|^{2} |f(u_{k - 1}(s, y))|^{2} dy ds \nonumber \\
&= 2|u_{0}(t, x)|^{2} + 2\mathbb{E} \int_{0}^{t} \int_{\mathbb{R}^{n}} |K_{t - s}(x - y)|^{2} |f(u_{k - 1}(s, y))|^{2} dy ds.
\end{align}
Note that by Lemma \ref{boundedcontinuous}, $u_{0}(t, x)$ is bounded on $t \in [0, T]$, $x \in \mathbb{R}^{n}$. So we consider the remaining term. Using the calculation in \eqref{squareintegrabilityPicard} and the bound $|f(x)| \le L(1 + |x|)$ for some $L$ by the Lipschitz condition,
\begin{align*}
\mathbb{E} \int_{0}^{t} \int_{\mathbb{R}^{n}} &|K_{t - s}(x, y)|^{2} |f(u_{k - 1}(s, y))|^{2} dy ds \\
&\le 2L^{2} \left(1 + \sup_{s' \in [0, t], y' \in \mathbb{R}^{n}} \mathbb{E}\left(|u_{k - 1}(s', y')|\right)^{2}\right) \left(\int_{0}^{t} (t - s)^{2 - n} ds\right)||K||^{2}_{L^{2}(\mathbb{R}^{n})} \le \tilde{C}_{k, T},
\end{align*}
where $\tilde{C}_{k, T}$ is the finite constant, independent of $t \in [0, T]$ and $x \in \mathbb{R}^{n}$, 
\begin{equation*}
\tilde{C}_{k, T} := 2L^{2} \left(1 + \sup_{s' \in [0, T], y' \in \mathbb{R}^{n}} \mathbb{E}\left(|u_{k - 1}(s', y')|\right)^{2}\right) \left(\int_{0}^{T} (T - s)^{2 - n} ds\right)||K||^{2}_{L^{2}(\mathbb{R}^{n})},
\end{equation*}
which is finite by the inductive assumption \eqref{inductive1}, Lemma \ref{kernelLq}, and the fact that $n = 1$ or $2$. This verifies \eqref{nextinductive}.

Finally, we show that property 4 in Step 1 holds, namely that the mapping
$(t, x) \mapsto u_{k}(t, x)$ taking values in $L^{2}(\Omega)$ is continuous on $\mathbb{R}^{+} \times \mathbb{R}^{n}$. We decompose $u_{k}(t, x)$ in \eqref{Picarddef} as
\begin{equation*}
u_{k}(t, x) = u_{0}(t, x) + \int_{0}^{t}\int_{\mathbb{R}^{n}} K_{t - s}(x - y) f(u_{k - 1}(s, y)) W(dy, ds) := u_{0}(t, x) + u^{stoch}_{k}(t, x).
\end{equation*}
Because $u_{0}(t, x)$ is deterministic and continuous by Lemma \ref{boundedcontinuous}, it suffices to show that $u^{stoch}_{k}(t, x)$ is continuous in $L^{2}(\Omega)$. Consider $t_{0} > 0$ and $x_{0} \in \mathbb{R}^{n}$. 
(The argument for $t_0 = 0$ is similar.) Let 
\begin{equation}\label{Sdelta}
S_{\delta} = \{(t, x) \in \mathbb{R}^{+} \times \mathbb{R}^{n} : |t - t_{0}| < \delta, |x - x_{0}| < \delta\}.
\end{equation}
Continuity would follow if we can show that given arbitrary $\epsilon > 0$, there exists $\delta > 0$ sufficiently small such that
\begin{equation}\label{cont1}
\mathbb{E}\left(|u_{k}^{stoch}(t, x_{0}) - u_{k}^{stoch}(t_{0}, x_{0})|^{2}\right) < \epsilon, \qquad \text{ for } |t - t_{0}| < \delta,
\end{equation}
\begin{equation}\label{cont2}
\mathbb{E}\left(|u_{k}^{stoch}(t, x_{1}) - u_{k}^{stoch}(t, x_{0})|^{2}\right) < \epsilon, \qquad \text{ for all } (t, x_{1}), (t, x_{0}) \in S_{\delta}.
\end{equation}

Denote
\begin{equation}\label{T*}
T^{*} = t_{0} + 1.
\end{equation}

Let us show the first part of the continuity estimate \eqref{cont1}. For every $\epsilon > 0$, we need to find a $\delta > 0$ such 
that \eqref{cont1} holds. We begin by first assuming that $\delta > 0 $ is such that
\begin{equation}\label{delta1}
\delta < \min \{ 1, \frac{t_0}{2} \} = \frac{t_{0}}{2} \wedge 1
\end{equation}
(the reason for this choice will be clear later), 
and we denote 
\begin{equation}\label{tau}
\tau_{m} = t \wedge t_{0} > 0, \qquad \tau_{M} = t \vee t_{0} > 0,
\end{equation}
where $t \vee t_{0} := \max \{t, t_0\}$.
By using a change of variables, 
\begin{align*}
| & u_{k}^{stoch}(t, x_{0}) - u_{k}^{stoch}(t_{0}, x_{0})| \\
&= \left|\int_{0}^{\tau_{m}} \int_{\mathbb{R}^{n}} K_{\tau_{m} - s}(x - y) [f(u_{k - 1}(s + \tau_{M} - \tau_{m}, x_{0})) - f(u_{k - 1}(s, x_{0})) W(dy, ds) \right| \\
&+ \left|\int_{0}^{\tau_{M} - \tau_{m}} \int_{\mathbb{R}^{n}} K_{\tau_{M} - s}(x - y) f(u_{k - 1}(s, x_{0})) W(dy, ds)\right|.
\end{align*}
Using the Lipschitz condition and the growth condition  \eqref{growth} on $f$, together with the Itô isometry \eqref{Ito},
we can bound the expectation $\mathbb{E}(|u_{k}^{stoch}(t, x_{0}) - u_{k}^{stoch}(t_{0}, x_{0})|^{2})$ by two integrals, 
$J_1$ and $J_2$, one integrated from $0$ to $\tau_m$ and the other from $\tau_m$ to $\tau_M$:
\begin{align}\label{J12}
\mathbb{E}(| &u_{k}^{stoch}(t, x_{0}) - u_{k}^{stoch}(t_{0}, x_{0})|^{2}) \nonumber \\
&\le 2L^{2}\int_{0}^{\tau_{m}} \int_{\mathbb{R}^{n}} |K_{\tau_{m} - s}(x - y)|^{2} \mathbb{E} \left(|u_{k - 1}(s + \tau_{M} - \tau_{m}, x_{0}) - u_{k - 1}(s, x_{0})|^{2}\right) dy ds \nonumber \\
&+ 4L^{2}\int_{0}^{\tau_{M} - \tau_{m}} \int_{\mathbb{R}^{n}} |K_{\tau_{M} - s}(x - y)|^{2} \left(1 + \mathbb{E}\left(|u_{k - 1}(s, x_{0})|^{2}\right)\right) dy ds := 2L^{2}(J_{1} + 2J_{2}).
\end{align}
To handle $J_{1}$, as long as  the condition \eqref{delta1} on $\delta$ 
is satisfied, we have $\tau_{m} \le T^{*}$, where $T^*$ is defined in \eqref{T*}. Hence, by \eqref{viscouscalc},
\begin{equation*}
\int_{0}^{\tau_{m}} \int_{\mathbb{R}^{n}} |K_{\tau_{m} - s}(x - y)|^{2} dy ds \le \left(\int_{0}^{T^{*}} (T^{*} - s)^{2 - n} ds\right) \cdot ||K||^{2}_{L^{2}(\mathbb{R}^{n})} := C_{1}.
\end{equation*}
Since continuous functions are uniformly continuous on compact sets, by using the fact that $u_{k - 1}(t, x)$ is $L^{2}(\Omega)$ continuous, along with $0 < \tau_{m} < \tau_{M} \le T^{*}$ and $|\tau_{M} - \tau_{m}| < \delta$, we can make 
\begin{equation}\label{J1cond}
J_{1} < \frac{\epsilon}{4L^{2}},
\end{equation}
by choosing $\delta < \frac{t_{0}}{2} \wedge 1$ sufficiently small so that
\begin{equation*}
\mathbb{E} \left(|u_{k - 1}(t_{1}, x_{0}) - u_{k - 1}(t_{2}, x_{0})|^{2}\right) < C_{1}^{-1}\frac{\epsilon}{4L^{2}},
\text{ whenever } |t_{1} - t_{2}| < \delta \text{ and } t_{1}, t_{2} \in [0, T^{*}].
\end{equation*}
{{To handle $J_{2}$, we note that by \eqref{nextinductive} and a calculation similar to \eqref{viscouscalc}, 
\begin{align*}
J_{2} \le (1 + C_{k - 1, T^{*}}) \int_{0}^{\tau_{M} - \tau_{m}} \int_{\mathbb{R}^{n}} |K_{\tau_{M} - s}(x - y)|^{2} dy ds &= \tilde{C}_{k - 1, T^{*}} \int_{0}^{\tau_{M} - \tau_{m}} (\tau_{M} - s)^{2 - n} ds \\
&= \frac{1}{3 - n}\tilde{C}_{k - 1, T^{*}} (\tau_{M}^{3 - n} - \tau_{m}^{3 - n}).
\end{align*}
Therefore, because $|\tau_{M} - \tau_{m}| < \delta$ and $0 \le \tau_{m} \le \tau_{M} < T^{*}$ by \eqref{T*} and \eqref{delta1}, we can choose $\delta$ satisfying condition \eqref{delta1} sufficiently small such that 
\begin{equation}\label{J2cond}
J_{2} < \frac{\epsilon}{8L^{2}}.
\end{equation}
So by \eqref{J12}, \eqref{J1cond}, \eqref{J2cond}, we can choose $\delta$ sufficiently small so that \eqref{cont1} holds. }}

Next, we verify \eqref{cont2}. By the Itô isometry \eqref{Ito} and the bound in Lemma \ref{Picardwell}, 
\begin{align*}
\mathbb{E}\left(|u_{k}^{stoch}(t, x_{1}) - u_{k}^{stoch}(t, x_{0})|^{2}\right) &= \int_{0}^{t} \int_{\mathbb{R}^{n}} |K_{t - s}(x_{1} - y) - K_{t - s}(x_{0} - y)|^{2} \mathbb{E}\left(|u_{k - 1}(s, y)|^{2}\right) dy ds \\
&\le C_{k - 1, T^{*}} \int_{0}^{t} \int_{\mathbb{R}^{n}} |K_{t - s}(x_{1} - y) - K_{t - s}(x_{0} - y)|^{2} dy ds \\
&= C_{k - 1, T^{*}} \int_{0}^{t} \int_{\mathbb{R}^{n}} |K_{s}(y) - K_{s}(y + x_{0} - x_{1})|^{2} dy ds.
\end{align*}
Recall that the Fourier transform of $K_{t}(x)$ is $e^{-\frac{|\xi|}{2}t}\frac{\sin\left(\frac{\sqrt{3}}{2}|\xi|t\right)}{\frac{\sqrt{3}}{2}|\xi|}$. Therefore, by Plancherel's formula,
\begin{align}\label{spacecontest}
\mathbb{E}(|u_{k}^{stoch}(t, x_{1}) &- u_{k}^{stoch}(t, x_{0})|^{2}) \le C_{k - 1, T^{*}} \int_{0}^{t} \int_{\mathbb{R}^{n}} \left|e^{-\frac{|\xi|}{2}s}\frac{\sin\left(\frac{\sqrt{3}}{2}|\xi|s\right)}{\frac{\sqrt{3}}{2}|\xi|}\right|^{2} |1 - e^{i(x_{0} - x_{1})\cdot \xi}|^{2} d\xi ds \nonumber \\
&= 2C_{k - 1, T^{*}} \int_{0}^{t} \int_{\mathbb{R}^{n}} \left|e^{-\frac{|\xi|}{2}s}\frac{\sin\left(\frac{\sqrt{3}}{2}|\xi|s\right)}{\frac{\sqrt{3}}{2}|\xi|}\right|^{2} [1 - \cos((x_{0} - x_{1})\cdot \xi)] d\xi ds \nonumber \\
&\le  2C_{k - 1, T^{*}} \int_{0}^{T^{*}} \int_{\mathbb{R}^{n}} \left|e^{-\frac{|\xi|}{2}s}\frac{\sin\left(\frac{\sqrt{3}}{2}|\xi|s\right)}{\frac{\sqrt{3}}{2}|\xi|}\right|^{2} [1 - \cos((x_{0} - x_{1})\cdot \xi)] d\xi ds \nonumber \\
&\le 4C_{k - 1, T^{*}} \int_{0}^{\tau} \int_{\mathbb{R}^{n}} \left|e^{-\frac{|\xi|}{2}s}\frac{\sin\left(\frac{\sqrt{3}}{2}|\xi|s\right)}{\frac{\sqrt{3}}{2}|\xi|}\right|^{2} d\xi ds \nonumber \\
&+ 2C_{k - 1, T^{*}} \int_{\tau}^{T^{*}} \int_{\mathbb{R}^{n}} \left|e^{-\frac{|\xi|}{2}s}\frac{\sin\left(\frac{\sqrt{3}}{2}|\xi|s\right)}{\frac{\sqrt{3}}{2}|\xi|}\right|^{2} [1 - \cos((x_{0} - x_{1})\cdot \xi)] d\xi ds \nonumber \\
&:= 4C_{k - 1, T^{*}}J_{3} + 2C_{k - 1, T^{*}}J_{4},
\end{align}
where $\tau > 0$ will be chosen later. We have repeatedly used the fact that as long as $\delta$ is chosen  so that it is also less than one (see \eqref{delta1}), then $t \in [0, T^{*}]$ for $t \in S_{\delta}$. Note that 
\begin{equation*}
\int_{0}^{T^{*}} \int_{\mathbb{R}^{n}} \left|e^{-\frac{|\xi|}{2}s}\frac{\sin\left(\frac{\sqrt{3}}{2}|\xi|s\right)}{\frac{\sqrt{3}}{2}|\xi|}\right|^{2} d\xi ds = \int_{0}^{T^{*}} \int_{\mathbb{R}^{n}} |K_{s}(y)|^{2} dy ds < \infty,
\end{equation*}
by a calculation similar to \eqref{viscouscalc}. Therefore, by choosing $\tau \in (0, T^{*})$ sufficiently small, we can make 
\begin{equation}\label{J3}
4C_{k - 1, T^{*}}J_{3} < \frac{\epsilon}{2}.
\end{equation}
Now that we have fixed a choice of $\tau$, we consider $J_{4}$. We split it into two integrals,
one over the frequencies $\xi$ such that $| \xi | > Ms^{-1}$, and the other over $| \xi | \le Ms^{-1}$,
where $M > 0$ we will be chosen later:
\begin{equation*}
J_{4} \le \int_{\tau}^{T^{*}} \int_{|\xi| > Ms^{-1}} \cdot + \int_{\tau}^{T^{*}} \int_{|\xi| \le Ms^{-1}} \left|e^{-\frac{|\xi|}{2}s}\frac{\sin\left(\frac{\sqrt{3}}{2}|\xi|s\right)}{\frac{\sqrt{3}}{2}|\xi|}\right|^{2} [1 - \cos((x_{0} - x_{1})\cdot \xi)] d\xi ds.
\end{equation*}
By noting that  $\frac{\sin\left(\frac{\sqrt{3}}{2}|\xi|s\right)}{\frac{\sqrt{3}}{2}|\xi|} \le s \le T^{*}$ and $0 \le 1 - \cos((x_{0} - x_{1})\cdot \xi) \le 2$ in the first integral, and $0 \le 1 - \cos(\theta) \le \frac{1}{2}\theta^{2}$ in the second integral, we get:
\begin{align*}
J_{4} &\le 2(T^{*})^{2} \int_{\tau}^{T^{*}} \int_{|\xi| > Ms^{-1}} e^{-|\xi|s} d\xi ds + \int_{\tau}^{T^{*}} \int_{|\xi| \le Ms^{-1}} e^{-|\xi|s} |x_{0} - x_{1}|^{2} d\xi ds \\
&= 2(T^{*})^{2} \int_{\tau}^{T^{*}} s^{-n} \int_{|\eta| > M} e^{-|\eta|} d\eta ds + |x_{0} - x_{1}|^{2} \int_{\tau}^{T^{*}} s^{-n} \int_{|\eta| \le M} e^{-|\eta|} d\eta ds \\
&\le 2T^{*} \tau^{-n} \left((T^{*})^{2}\int_{|\eta| > M} e^{-|\eta|} d\eta + |x_{0} - x_{1}|^{2} \int_{\mathbb{R}^{n}} e^{-|\eta|} d\eta \right).
\end{align*}
By taking $M$ sufficiently large such that 
\begin{equation*}
\int_{|\eta| > M} e^{-|\eta|} d\eta < \frac{1}{2C_{k - 1, T^{*}}}\frac{1}{2(T^{*})^{3}\tau^{-n}}\frac{\epsilon}{4},
\end{equation*}
and then taking $\delta > 0$ sufficiently small satisfying the condition \eqref{delta1}, such that
\begin{equation*}
2T^{*} \tau^{-n} \delta^{2} \int_{\mathbb{R}^{n}} e^{-|\eta|}d\eta < \frac{1}{2C_{k - 1, T^{*}}}\frac{\epsilon}{4},
\end{equation*}
we have that $2C_{k - 1, T^{*}} J_{4} < \frac{\epsilon}{2}$ whenever $(t, x_{0}), (t, x_{1}) \in S_{\delta}$ with $|x_{0} - x_{1}| < \delta$. Using this fact along with \eqref{J3} in \eqref{spacecontest} establishes the desired result \eqref{cont2}.
\end{proof}

\section{Acknowledgements}
This work was partially supported by the National Science Foundation under grants DMS-1613757, DMS-1853340, and DMS-2011319. 
This material is based upon work supported by the National Science Foundation under grant DMS-1928930 while the authors participated in a program hosted by the Mathematical Sciences Research Institute in Berkeley, California, during the 
Spring 2021 semester.

\bibliography{SVWEbib}
\bibliographystyle{plain}

\end{document}